\documentclass[a4paper, 12pt, reqno]{amsart}
\usepackage{bookmark}
\usepackage{amsfonts}
\usepackage{amsmath}
\usepackage{amsthm}
\usepackage{amssymb}
\usepackage{verbatim}
\usepackage{enumitem}
\usepackage[font=small,labelfont={bf,sf},margin=0em]{caption}
\usepackage{mathtools}
\usepackage{autonum}  

\newcommand{\KN}{\mathbin{\bigcirc\mspace{-15mu}\wedge\mspace{3mu}}}
\newcommand{\ip}[2]{\langle #1 , #2 \rangle}

\newcommand{\fop}{\mathcal{F}}
\newcommand{\dop}{\mathcal{D}}
\newcommand{\lop}{\mathcal{L}}
\newcommand{\qop}{\mathcal{Q}}
\newcommand{\rop}{\mathcal{R}}

\newcommand{\twoeta}[3]{\left| #1 \right|_{2, \eta; #2, #3}}

\newcommand{\lie}{\mathcal{L}}
\newcommand{\real}{\mathbb{R}}

\newcommand{\nats}{\mathbb{N}}

\newcommand{\lap}{\Delta}

\newcommand{\oh}{\tfrac{1}{2}}
\newcommand{\of}{\tfrac{1}{4}}
\newcommand{\on}[1]{\tfrac{1}{#1}}
\newcommand{\onf}[1]{\frac{1}{#1}}

\newcommand{\symme}[3]{\left[#3 \right]_{#1 \leftrightarrow #2}}

\newcommand{\rmplus}{\Lambda_{\Rm}}

\newcommand{\dd}[2]{{#2}^{[#1]}}
\newcommand{\barr}[1]{\overline{#1}}
\newcommand{\berr}[1]{\underline{#1}}

\newcommand{\uf}[1]{\mathrm{#1}}

\newcommand{\V}{{\uf{V}}}

\newcommand{\mr}{k}
\newcommand{\ein}{\mu}

\newcommand{\fund}{\kappa}

\newcommand{\hu}{{\hat u}}

\newenvironment{claim}[1]{\par\noindent\underline{Claim:}\space#1}{}
\newenvironment{claimproof}[1]{\par\noindent\underline{Proof of Claim:}\space#1}{\hfill $\blacksquare$}

\newcommand{\Bry}{{Bry}}
\newcommand{\Pert}{{Pert}}

\newcommand{\defeq}{\mathrel{\mathop:}=}
\newcommand{\eqdef}{\mathrel{\mathop=}:}

\newtheorem{lemma}{Lemma}
\newtheorem{theorem}[lemma]{Theorem}
\newtheorem{corollary}[lemma]{Corollary}
\theoremstyle{definition}
\newtheorem{definition}[lemma]{Definition}

\numberwithin{lemma}{section}
\theoremstyle{remark}
\newtheorem*{remark}{Remark}

\DeclareMathOperator{\Rc}{Rc}

\DeclareMathOperator{\Rf}{Rf}
\DeclareMathOperator{\Rm}{Rm}

\DeclareMathOperator{\grad}{grad}

\DeclareMathOperator{\Cov}{Cov}
\DeclareMathOperator{\UT}{UT}
\DeclareMathOperator{\Vol}{Vol}

\begin{document}
\author{Timothy Carson}
\title[Asymptotically cylindrical singularities]{Ricci flow from some spaces with asymptotically cylindrical singularities}
\email{carstimon@gmail.com}
\begin{abstract}
  We prove the existence of Ricci flow starting from a class of metrics with unbounded curvature, which are doubly-warped products over an interval with a spherical factor pinched off at an end.  These provide a forward evolution from some known and conjectured finite-time local singularities of Ricci flow, generalizing previous examples.  The class also includes metrics with non-compact singular ends which become instantaneously compact.  Furthermore, we prove local stability of the forward evolution, which allows us to glue it to other manifolds and create a forward evolution from spaces which are not globally warped products.
\end{abstract}
\maketitle

\section{Introduction}
For any complete Riemannnian manifold $(M, g)$ with bounded curvature, there is a smooth solution to the Ricci flow,
\begin{align}
  \partial_t g(t) = - 2 \Rc[g(t)]
\end{align}
with $g(0) = g$ \cite{Shi}.  The solution exists up to some time $T \in (0, \infty]$.  In this paper, we prove the existence of a forward evolution of Ricci flow from a certain class of Riemannian manifolds with unbounded curvature.  The initial metrics we consider have singular neighborhoods which are asymptotically cylindrical warped products of spheres.

Our primary motivation for considering this problem is the continuation of Ricci flow after singularities.  The forward evolution of a smooth manifold often encounters local singularities in finite time.  In some cases, the local singularities can be understood well enough so that Ricci flow with surgery can be implemented, e.g. \cite{h_posiso}, \cite{Perelman2}, \cite{brendlesurg}.  In the three dimensional case the body of knowledge is by now quite powerful \cite{singular}, \cite{uniquenessBK}.

All of these surgery examples work by proving that every local singularity encountered has a part close to a shrinking cylinder $\real \times S^{n-1}$.  The ideal situation for Ricci flow encountering such a singularity is when the metric is a warped product on $I \times S^{n-1}$ for some interval $I$:
\begin{align}
  g = a(x)^2dx^2 + \phi(x)^2 g_{S^{n-1}}.
\end{align}
(Here $g_{S^{n-1}}$ is the standard metric on the $S^{n-1}$ factor.)  By choosing $\phi$ correctly, the  forward evolution from $g$ encounters a local singularity, named a neckpinch.  This was conjectured in \cite{h_formation} and first shown by Simon \cite{simon_pinch}.  In \cite{AK}, \cite{AKPrecise} Angenent and Knopf expanded on these singularities and gave a precise asymptotic description.  Their description in particular gives a description of the metric at the final time, when the singularity has occurred.  In \cite{ACK}, Angenent, Caputo, and Knopf proved the existence of a forward evolution of Ricci flow from these final-time singular metrics.  

The first main theorem in the present work provides the forward evolution from a family of singular metrics which includes those explored in \cite{ACK}.  It also includes the forward evolution, in the ideal (doubly-warped product) case, from (conjectured) singularities which are modeled on $\real^k \times S^{n-k}$.  Our description of the forward evolution is very precise, and we hope to provide a testbed for a more general theory that can deal with these singularities.  We hope that our generalizations clarify the role played by various pieces.

The general question of which singular spaces have a forward Ricci flow has received attention from many authors.  Particular success has been had with curvature bounds from below \cite{Simon_1}, \cite{Simon_2}, \cite{CRW}, \cite{lowerricci_integralcurvature}, \cite{almostnonneg}.  Another work addressing spaces with specific singularity models is \cite{conicalSing}. For some results with low regularity on the initial metric, see \cite{Simon_0} and \cite{kochlamm}. Furthermore, the Ricci flow of warped products lends itself to comparison to reaction-diffusion equations in Euclidean space, where there are quite general existence and uniqueness theories \cite{continuation}.

\subsection{Model Pinches}
We now give a definition of the singular metrics, which we call model pinches.  Let $q \geq 2$, and let $(S^q, g_{S^q})$ be the round sphere of sectional curvature 1, which satisfies $2\Rc[g_{S^q}] = \ein g_{S^q}$ for $\ein = 2(q-1)$.  Also let $(F, g_F)$ be any Einstein manifold with $2 \Rc[g_F] = \ein_F g_F$.  The metrics will be metrics on $I \times S^q \times F$ of the form
\begin{align}
  g_{mp} = dx^2 + \phi(x)^2 g_{S^q} + \psi(x)^2 g_F.
\end{align}
The main case of interest is $F = S^p$ but $F$ may be zero dimensional (landing us in the singly warped product case) or have negative Ricci curvature.  The function $\phi$ will be increasing, so we can use $u = \phi^2$ as a coordinate and write
\begin{align}\label{mp_form}
  g_{mp} = \frac{du^2}{u V_0(u)} + u g_{S^q} + W_0(u) g_F.
\end{align}
Here $V_0(u) = u^{-1}|du|^2_{g_{mp}} = 4 |d\phi|^2_{g_{mp}}$.  

For the rest of the paper we fix some $\eta \in (0, \oh)$. For any metric $g$, function $f:M \to \real$, and scale function $\rho: M \to \real_+$ we use the notation
$
  \twoeta{f}{\rho}{g} : M \to \real_{\geq 0}
$
to mean the following.  Take any point $p \in M$, scale the metric $g$ to $\hat g = \frac{g}{\rho(p)^2}$, and then take the $C^{2, \eta}$ norm in the ball of radius $1$ around $p$ with respect to $\hat g$.

\begin{definition}\label{definition:model_pinch}
  A metric on $M = (0, \infty) \times S^q \times F$ of the form \eqref{mp_form} is a \emph{model pinch} if
  \begin{enumerate}[label=(MP\arabic{*}), ref=(MP\arabic{*})]
  \item \label{modelpinch_vsmall}As $u \searrow 0$, $V_0(u) \searrow 0$.
  \item \label{w_big} If $\ein_F > 0$, there is a $c > 0$ such that $\frac{W_0(u)}{u} \geq (1 + c) \frac{\ein_F}{\ein} $.
  \item \label{modelpinch_reg}
    For some $C>0$
    \begin{align}
      \frac{\twoeta{V_0}{u/2}{(du)^2}}{V_0}
      +
      \frac{\twoeta{W_0}{u/2}{(du)^2}}{W_0}
      \leq C
    \end{align}
  \item  For any $u_1 > 0$, on the set $\{u > u_1\}$ the curvature of $g_{mp}$ is strictly bounded, $V_0$ and $W_0$ are $C^{\infty}$ and strictly positive, and $V_0$ is bounded.
  \end{enumerate}
\end{definition}
One way to interpret this definition is that as $u \searrow 0$ the metric is asymptotically some sort of cylinder.  At the distance scale given by $\sqrt{u}$, if $W_0(u) \gg u$ the metric is close to the product $(\real, g_{\real})\times (S^q, g_{S^q})\times(\real^p, g_{\real^p})$, and if $W_0(u) \sim au$ it is close to $(\real, g_{\real}) \times (S^q, g_{S^q}) \times (F, ag_F)$.   For some precision, see Corollaries \ref{prish_asymptotic_cylinder} and \ref{prish_curvature_control}, which are stated for the forward evolution but in particular hold for the initial metric.  

The H\"older condition implies the first and second derivatives of $V_0$ satisfy $|u \partial_u V_0| + |u^2\partial_u^2V_0| \leq CV_0$ (for a different $C$) and similarly for $W_0$.  This allows for change by a factor of $(1 + O(\epsilon C))$ in the region where $u$ is $(1 + O(\epsilon))u$.

Our first theorem is the following short-time existence result in the class of warped products.  We identify $(0, \infty) \times S^q$ with $\real^{1 + q} \setminus \{0\}$ and write $M \defeq \left( \real^{1+q}\setminus\{0\} \right) \times F \subset \real^{1+q} \times F \eqdef \bar M$.

\begin{theorem}\label{theorem:model_pinch_flow}
  Let $(M, g_{mp})$ be a model pinch.   For some $T_2 \in (0, \infty]$ there is a Ricci flow $(\bar M, g_{wp}(t))$ for $t \in (0, T_2)$.  As $t \searrow 0$, $g_{wp}(t) \to g_{mp}$ in $C^{\infty}_{loc}(M)$.   There are choices of the parameters of Definitions \ref{productish_barricaded} and \ref{tip_barricaded} such that $g_{wp}$ is controlled in the productish region and in the tip region.
\end{theorem}
The last sentence of the theorem gives a good description of the forward evolution near the origin, we will give an overview in Section \ref{section:shape_description}.  Here we just give a rapid tour of some properties that we think are important.    In the forward evolution a small Bryant soliton appears at the origin, and the radius of the $F$ factor is strictly positive for positive time even if it was not for the initial metric $g_{mp}$.  The distance-squared scale of the Bryant soliton is on the order of $t V_0(t)$, and the largest Ricci curvature forward in time, which occurs at the origin, is on the order of $1/(t V_0(t))$.  At the origin, the distance-squared scale of the $F$ factor is $W_0(t) - \ein_F t$.  If $\ein_F \neq 0$ then this is at least order $t$: property \ref{w_big} says $W_0(t) \gtrsim t$ if $\ein_F > 0$, and if $\ein_F < 0$ the term $- \ein_F t$ helps.  This is not the case if $\ein_F = 0$ and $W_0(u) = o(u)$, so the largest Riemannian curvature could be on $F$ (and see cancellation in the Ricci curvature). Finally, we can use the control to provide a more precise rate of convergence of $g_{wp}(t)$ to $g_{mp}$; see Corollary \ref{improved_convergence}.

The next theorem removes the global warped product part of the model pinch assumption.  For this, we need some additional assumptions on the curvature of the factor $F$, which is inevitable since we allow perturbations of the initial metric in any direction.  In particular, we rule out the case $W_0(t) - \ein_F t \lesssim t\nu(t)$.
For the metric $g_F$, let
$
\Lambda_F = \sup_{p \in F}\max_{h \in Sym_2(T_pF), |h|=1}(\Rm)_{abcd}h^{ac}h^{bd}
$.  
For example, if $F$ has dimension $p$ and constant sectional curvature $k$ then $\Lambda_F = k(p-1)$.  In particular, $2\Lambda_{S^q} = \ein = 2(q-1)$.
\begin{definition}\label{reasonable_def}
  A model pinch is \emph{$\Rm$-permissible} if the following is satisfied.
  \begin{enumerate}[label=(RP\arabic{*}), ref=(RP\arabic{*})]
  \item \label{lambda_assump}
    In the case $\Lambda_F > 0$, we additionally require
    $\frac{W_0(u)}{u} \geq \frac{\Lambda_F}{\Lambda_{S^q}}$ 
  \item \label{w_big2}
    In the case $\ein_F = 0$ and $\Lambda_F = 0$ (i.e. $(F, g_F)$ is flat) we additionally require $\frac{W_0(u)}{u V_0(u)} \to \infty$ as $u \searrow 0$.
  \end{enumerate}
\end{definition}
\begin{theorem}\label{theorem:unsymmetrical_flow}
  Let $g_{mp}$ be an $\Rm$-permissible model pinch.  There is an $\epsilon_0$ depending on $g_{mp}$ with the following property.

  Let $(N^n, g)$ be a (possibly non-complete) Riemannian manifold.  Let $U \subset N$ be open, and assume that $(N \setminus U, g)$ is a complete manifold with boundary, satisfying, for some $r_0 > 0$ and all $p \in N\setminus U$, $|\Rm|(p) \leq r_0^{-2}$ and $\Vol(B(p, r_0)) \geq (1-\epsilon_0) \omega_n$.
  
    Suppose that $u_1>0$ and $\Phi: U \to (0, u_1) \times S^q \times F$ is a diffeomorphism such that in $U$,
  \begin{align}
    \twoeta{g - \Phi^* g_{mp}}{r_0 |\Rm_{\Phi^*g_{mp}}|}{\Phi^*g_{mp}}
    \leq
    \epsilon_0 V_0 \circ \Phi.
  \end{align}
  
  Let $\bar N \supset N$ be the differential manifold obtained by replacing $U\sim (L, L') \times S^q \times F$ with $\bar U \sim D^{1+q} \times F$.  For some $T_* > 0$, there is a Ricci flow $g(t)$, for $t \in [0, T_*]$ on $\bar N$ such that $g(t) \to g$ in $C^{\infty}_{loc}(N)$ as $t \searrow 0$.
\end{theorem}

Immediate extensions of our theorems allow for multiple singular neighborhoods of $N$ each close to some model pinch, or for multiple extra warped factors $g_{F_i}$ each satisfying the requirements in the definition of model pinch.

\subsection{Overview of the proofs}
Both theorems are proven by constructing smooth mollified initial metrics, which agree with the singular initial metrics outside of a small set, and controlling the forward evolution of the smooth mollified metrics.  By sending the size of the mollification to zero, we construct a forward evolution from the singular initial metric.

To prove Theorem \ref{theorem:model_pinch_flow}, we control the relevant functions for the mollified initial metrics in terms of $u$.  The advantage of this is that the control is diffeomorphism-invariant-- for example, the value of $w$ or $v = u^{-1}|\nabla u|^2$ at the point where $u = u_1 \in \real_{> 0}$ is a diffeomorphism-invariant property.  A usual difficulty in controlling solutions to Ricci flow is that the linearization is only weakly parabolic because of the diffeomorphism invariance of Ricci flow, and this gets around that issue.  The most common response is to use Ricci-DeTurck flow, but we were not able to find a sufficiently good background metric to use in our case (and we tried some exotic possibilities).  Another option in our case would be to use an arclength coordinate, but that introduces an annoying nonlocal term.

The forward evolution is split into two regions-- the tip region, where a Bryant soliton forms, and the productish region, which includes the initial value and is where the metric continues to look locally like a product metric on $\real \times S^q \times F$.  In Lemmas \ref{main_prish_estimates} and \ref{main_tip_estimates}, we obtain local control in the productish and tip regions, assuming a priori boundary control.  In Section \ref{section:full_flow} we put this control together.  Section \ref{section:buckling} shows that the boundary control needed at the right of the tip region is ensured by the local estimates in the productish region, and the boundary control needed at the left of the productish region is ensured by the local estimates in the tip region.  We now have to prove the boundary conditions at the right boundary of the productish region, where the metrics are uniformly smooth. This is accomplished in Section \ref{section:ccc}.

\begin{figure}[tp]
  \includegraphics[width=\textwidth]{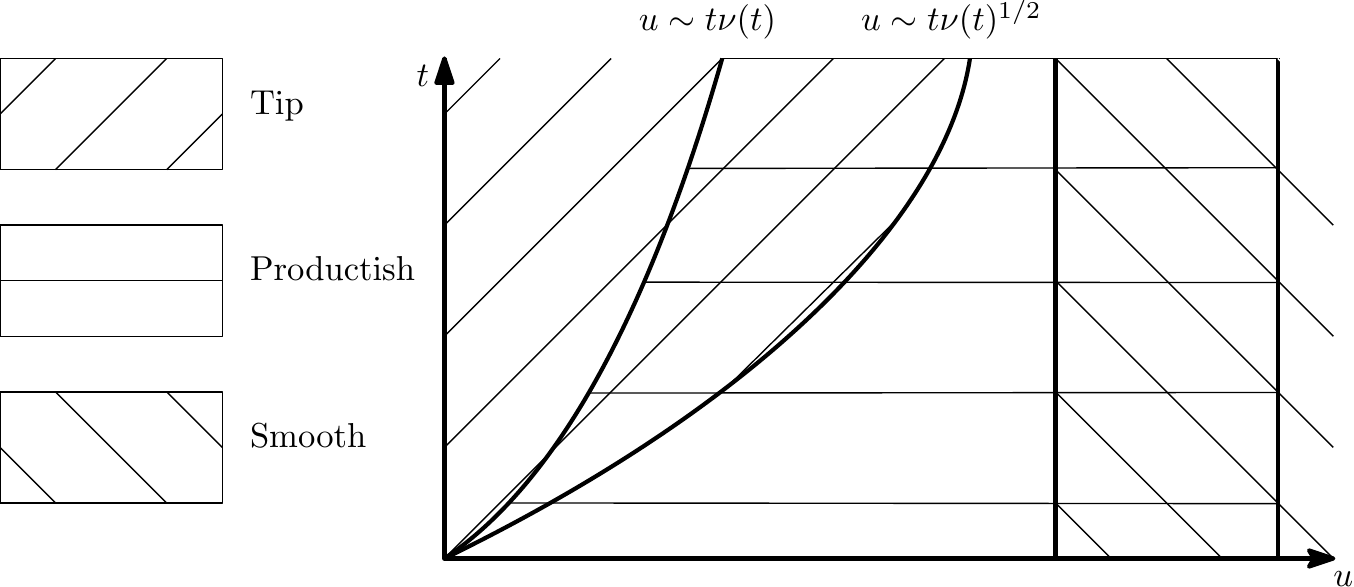}
  \caption{
    Map of the tip, productish, and uniformly smooth regions.  Here $\nu(t) = V_0(\ein t)$.\label{region_map}
    }
\end{figure}

The local estimates in the productish region use generic estimates for the solution to some reaction-diffusion equations in regions where they are nearly constant, which is dealt with in Appendix \ref{nearly_cnst_pde_sect}.  The local estimates in the tip region are more specific.

The proof of Theorem \ref{theorem:unsymmetrical_flow} uses Ricci-DeTurck flow around the already-constructed warped product evolution to control an arbitrary metric.  Theorem \ref{local_stability} is the main point in the proof, this gives us control of the Ricci-DeTurck flow in a neighborhood $U$ of the form $\{u < u_*\}$, assuming a priori boundary control.  Once we have Theorem \ref{local_stability}, we wrap up by controlling the evolution in the boundary region (where everything is uniformly smooth) in Section \ref{asymmetric}.

Again, the local control in Theorem \ref{local_stability} is split into two parts: control in the productish region, and control in the tip region.  As in the warped-product case, in the productish region we control the evolution using the results of Appendix \ref{nearly_cnst_pde_sect}.  On the other hand, in the tip region (where the solution is close to a small perturbation of the Bryant soliton), we use a contradiction-compactness argument to move the situation to the Bryant soliton.  Then, we use a stability result for the Bryant soliton, Theorem \ref{bry_stabil}.  Theorem \ref{bry_stabil} might be compared to results from Section 7 of \cite{uniquenessBK}; see the remark after the statement of the theorem.

\subsection{Infinitely long pinched ends}
In an attempt to simplify the initial exposition we have hidden that the left end of a model pinch, $\{u \leq u_1\}$, may have infinite length.  An arc length coordinate for the interval factor for a metric of the form \eqref{mp_form} is given by $ds =  \frac{1}{\sqrt{u V_0(u)}} du$, so the length is hidden in the integrability of $\frac{1}{\sqrt{uV_0(u)}}$ near $0$.  In the case when the left end of the initial model pinch has infinite length, the left end of the evolution on $\bar M$ is compact for positive time.

In two dimensions, Topping \cite{topping_reversecusp} constructed similar examples of noncompact surfaces which immediately become compact.  These examples actually have initial metrics with bounded curvature, so it is especially interesting when compared with Shi's existence result, which guarantees that the initial metric has a unique forward complete Ricci flow on the same topology.  This means that in two dimensions there is an alternative, perhaps more natural, forward evolution besides the instantaneously compact one.  In more than two dimensions, the analysis is different because the $S^q$ factor in the singly warped product has positive curvature and the initial metric must have unbounded curvature.  We do not expect a natural forward evolution on the same topology in this case.

\subsection{Some related short-time existence results}
Recent work that is close in spirit to ours is \cite{conicalExpanders} and \cite{conicalSing}.  In \cite{conicalExpanders}, Deruelle showed that for any cone with positive curvature,
i.e. a metric $ds^2 + s^2 g_{X}$ where $\Rm[g_{X}] \geq 1$,
there is an expanding Ricci soliton which limits, backwards in time, to the cone.  This can be considered as Ricci flow starting from the singular conical space.  In \cite{conicalSing}, Gianniotis and Schulze allow us to start Ricci flow from any manifold which has local singularities modeled on these cones, by using local stabilitiy similarly to our Theorem \ref{theorem:unsymmetrical_flow}.  Such cones that are especially relevant to us are the singly-warped products $ds^2 + a s^2 g_{S^q}$, for $a \in (0,1)$; these are singly warped products over intervals which are not covered by our theorem.

Alexakis, Chen, and Fournodavlos \cite{singlargerderiv} show the existence of a steady Ricci soliton of the form $ds^2 + \phi(s)^2 g_{S^q}$ with $\phi(s) \sim s^{1/\sqrt{q}}$.  They also examine forward evolutions of metrics close to their steady Ricci soliton.

Bamler, Cabezas-Rivas, and Wilking \cite{almostnonneg} examine the Ricci flow of manifolds with a variety of assumptions that curvature is bounded from below.  In particular, they deal with complete, bounded curvature manifolds $(M, g)$ satisfying
\begin{align}\label{almostnonneg_conditions}
 \Rm \geq -1, \quad \Vol_g (B_g(p, 1)) \geq v_0 \text{ for all } p \in M.
\end{align}
They show that there is a forward evolution for a time which \emph{only} depends on $v_0$ and the dimension.  An application is creating forward evolutions from singular spaces which can be approximated by manifolds with curvature bounded from below.  This gives an alternative approach to \emph{some} of the initial spaces considered by Gianniotis and Schulze in \cite{conicalSing}.

We wish to remark that we cannot apply the results in \cite{almostnonneg} in our case, but we need to use two different reasons.  First note that in the examples with an infinitely long pinched end, the assumption on the volume of balls in \eqref{almostnonneg_conditions} cannot be satisfied by approximating metrics, since the left end has balls of radius one with arbitrarily small volume.  We claim that in the compact case the curvature condition in \eqref{almostnonneg_conditions} is not satisfied.  Consider just the singly-warped metrics of form \eqref{mp_form}, so $g = \frac{du^2}{u V_0(u)} + u g_{S^q}$.  The curvature of such a metric is
\begin{align}
  \Rm =
  L \left( (u g_{S^q}) \KN (u g_{S^q}) \right)
  + K \left( (u g_{S^q}) \KN \frac{du^2}{u V_0(u)} \right)
\end{align}
where $L = u^{-1} ( 1 - \on4 V_0)$ and $K = - \on2 \partial_u V_0$.
The distance between $\{u = 0\}$ and $\{u = u_2\}$ is
$
  \int_{0}^{u_2}\frac{1}{\sqrt{u V_0(u)}} du
  = \int_{0}^{u_2} \frac{1}{u} \sqrt{ \frac{u}{V_0(u)}} du.
$
If $K$ is bounded from below, $\partial_u V_0 \leq C$ and so $V_0 \leq C u$ and this integral diverges.  So in the compact case $K$ goes to $-\infty$ and \eqref{almostnonneg_conditions} is not satisfied.  The other possible conditions of Theorem 2 from \cite{almostnonneg} are also not satisfied: $\Rm$ as an operator on $\bigwedge^2TM$ has the negative eigenvalue $K$ with multiplicity $q \geq 2$ so the curvature is not 2-non-negative, and we can check that the curvature operator is never weakly PIC1, although in the singly-warped case it has positive isotropic curvature.

Note that the model pinches do (in the case when $g_F$ has positive curvature or $W_0(u) \gg u$)  satisfy an almost-nonnegativity condition relevant to singularity analysis of Ricci flow, namely $\Rm \geq - f(|\Rm|)|\Rm|$ for a function $f$ satisfying $f(x) \to 0$ as $x \to \infty$.  This comes up, for example, in 12.1 of \cite{Perelman}.  In this case $f$ is a multiple of $V_0$ and we can use the assumption \ref{modelpinch_reg} to bound $K$, and the assumptions \ref{w_big} and \ref{modelpinch_reg} to bound curvatures involving the other factor.

\subsection{Model Pinches that arise as final-time limits}
Here we list some examples of smooth Ricci flows which have a model pinch has final-time limits.
\subsubsection{Singly-warped product singularities}
In \cite{AKPrecise}, Angenent and Knopf considered neckpinches occuring on singly warped products over an interval.  They proved that the warping function of the final-time limit of a neckpinch satisfies the asymptotics $\phi = \sqrt{u} \sim \frac{s}{\sqrt{|\log s|}}$,  where $s$ is the arclength from the singular end.  This implies $V_0(u) \sim \frac{1}{\log u}$.
Another singularity that may arise in the category of warped products of spheres over an interval is the degenerate neckpinch.  In this case, Angenent, Isenberg, and Knopf showed in \cite{AIK} that the final-time limit has the asymptotics $\phi \sim s^{\beta_k}$ where $\beta_k = \frac{2}{2k+1}$, $k \in \nats \setminus \{0\}$.  Forward evolutions from these specific cases were created in \cite{ACK} and \cite{recoverdegen}, respectively.

\subsubsection{Generalized cylinder singularities}\label{generalized_cylinder}
\begin{figure}[t]
  \includegraphics[width=\textwidth]{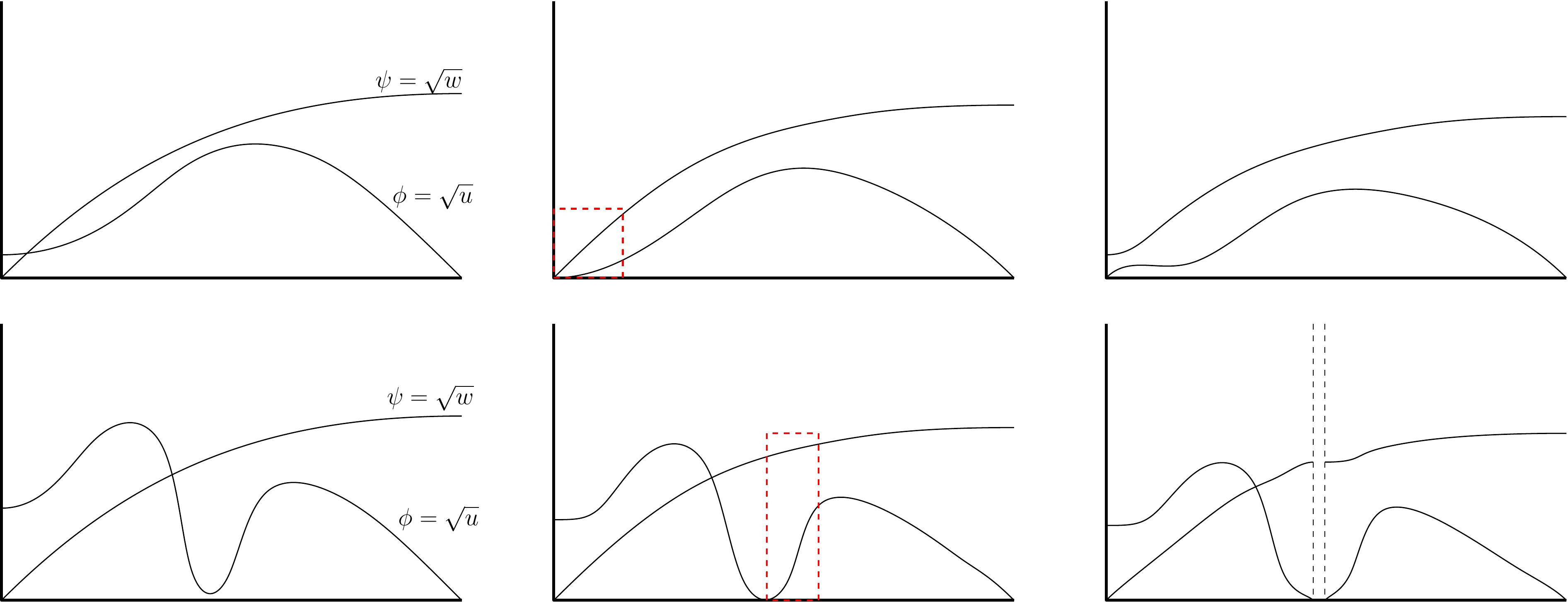}
  \caption
  [Simplified drawings of Ricci flow through model pinches]
  {
    Top: the singularity described in Section \ref{generalized_cylinder}.  Bottom: the singularity described in Section \ref{submanifold_neckpinches}.  The pictures depict the manifold, from left to right, before, during, and after the singular time.  The horizontal axis is the arclength from the left side.   On each row, the rectangle in the middle picture shows a neighborhood which is a part of a model pinch.  (In the second row, there are actually two model pinches: to the left and to the right). 
    \label{singularities_phipsigraphs}
    }
\end{figure}
For another example of a singularity, consider the doubly-warped product depicted in the top row of Figure \ref{singularities_phipsigraphs}.  A more stylized picture of a neighborhood of the singularity is Figure \ref{pancake_pinch}.  The metric is a doubly warped product over an interval, with $(F, g_F) = (S^p, g_{S^p})$, and the singularity occurs at the left endpoint of the interval.  Before the singular time, the metric satisfies the following boundary conditions at the left endpoint:
\begin{align}
  \phi > 0, \quad \partial_s \phi = 0, \quad \psi = 0, \quad \partial_s \psi = 1.
\end{align}
Here $s$ is the distance from the left endpoint.  A neighborhood of the left endpoint has topology $S^q \times D^{1+p}$ before the singular time.  For the initial metric, the size of the $S^q$ factor has a deep minimum at the center of the $D^{1+p}$.

As time goes on, the $S^q$ factor shrinks drastically, and the metric encounters a singularity which can be rescaled to a generalized cylinder $S^q \times \real^{1+p}$.  Without rescaling, at the singular time the metric takes on the topology of the cone over $S^q \times S^p$ (but is not asymptotically a metric cone).  This singularity has not been rigorously constructed, but formal calculations suggest that the singular pinched metric should have asymptotics
\begin{align}\label{pinched_pancake_asymptotics}
  \phi \sim \frac{s}{\sqrt{|\log s|}}, \quad \psi \sim s.
\end{align}
This is an unsurprising guess.  The factor corresponding to the $S^q$ behaves similarly to a standard neckpinch.  The $1+p$ dimensional part of the metric, $dx^2 + \psi^2 g_{S^p}$, is close to being a flat $D^{1+p}$, which corresponds to $\psi = x$ exactly.  The flat metric is stable enough that the perturbation from the pinching factor does not affect it too much.

In the forward evolution of metrics with asymptotics \eqref{pinched_pancake_asymptotics}, which we do investigate here, the size of the $S^p$ factor expands and the neighborhood takes on the topology $D^{1+q} \times S^p$.  
\begin{figure}[t]
  \includegraphics[width=\textwidth]{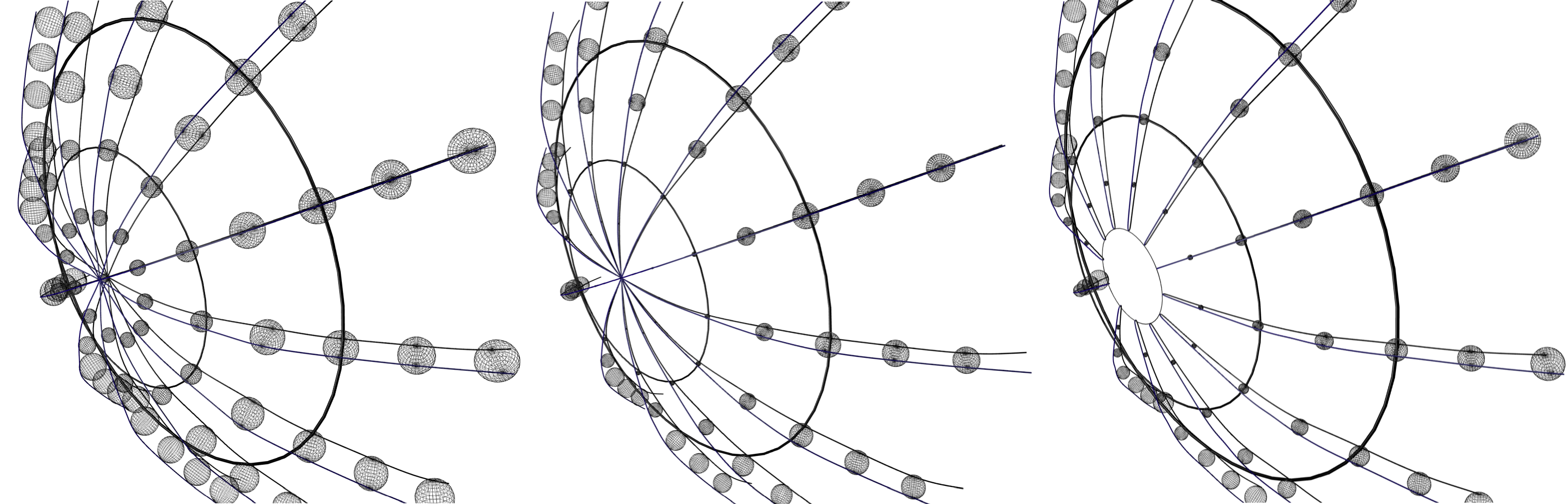}
  \caption
  [Ricci flow through the cone over $S^1 \times S^2$]
  {
    A Ricci flow through a model pinch with $q = 2$ and $(F, g_F) = (S^1, g_{S^1})$. The initial picture is a neighborhood with topology $S^2 \times D^{1+1}$, the middle picture has topology of the cone over $S^1 \times S^2$, and last picture has topology $D^{1+2} \times S^1$.  \label{pancake_pinch}
    }
\end{figure}

\subsubsection{Families of neckpinches}\label{submanifold_neckpinches}
Here is another example which is a singularity modeled on $\real^{1+p} \times S^q$, but which is qualitatively different from the previous.  We can also consider a doubly-warped product over an interval where $\phi$ has a neck somewhere in the interior of the interval.  Then we can force a singularity to occur in the interior of the interval modeled on $\real^{1 + p} \times S^q$.  Here there is an $S^p$ worth of one-dimensional neckpinches forming.  A trivial example of this is when we just cross a standard neckpinch with $S^p$.  While the previous example was also modeled on $\real^{1+p} \times S^q$, this one is qualitatively different: for example, the topological change through the singularity is different.

This type of singularity should be stable in the class of doubly warped products; perturbations leave $\phi$ with a local minimum.  However, in contrast to the previous example, it should not be stable in the full class of Riemannian metrics.  It should not even be stable in the class of singly warped products $g_B + \phi(b)^2 g_{S^q}$  where $B = \real \times S^p$ and $g_B$ is now an arbitrary metric on $B$.  (The original metric has $g_B = dx^2 + \psi(x)^2 g_{S^p}$.)  Indeed, if we allow $\phi$ to also depend on the $S^p$ factor and perturb it so it has a strict local minimum at some point on that factor, we should approach a singularity at a single point on the $S^p$ factor.  (Intuition for this may come from \cite{blowupsinglept_general}, which shows in particular that we can perturb the constant solution of simple reaction-diffusion equations on $\real^n$ to get a single point blowup.)
\subsubsection{Scarred neckpinches}\label{scarring}
Here is an example which leads to a metric which is not quite a model pinch.  First consider a standard singly warped neckpinch with spheres of dimension $S^{q}$: the initial metric is of the form $dx^2 + u(x)g_{S^{q}}$ and the metric at the singular time is a model pinch. This has a forward evolution, which recovers with a smooth disc of dimension $1 + q$ at the tip.  So, we have a Ricci flow of a singly warped product, at least on $((-1, 0) \cup (0, 1)) \times S^{q}$, for times $t \in [T_1, T_2]$, $T_1 < 0 < T_2$.  

Now, the Ricci flow of warped products with Einstein fibers does not care about the Riemannian curvature tensor of the fiber metric, it only cares about the Ricci curvature.  In other words: suppose we have a Ricci flow on $B \times F_1$ of the form
$g_B(t) + u(t)g_{F_1}$
(where for each $t$, $u(t):B \to \real_+$) and $\Rc_{g_{F_1}} = \ein g_{F_1}$.  Suppose $(F_2, g_{F_2})$ is another Einstein manifold with $\Rc_{g_{F_2}} = \ein g_{F_2}$.  Then
$g_B(t) + u(t)g_{F_2}$
is also a Ricci flow.

Therefore, in the Ricci flow through a standard neckpinch, we can swap out $g_{S^q}$ with any Einstein manifold $(F_{2}^q, g_{F_2})$ of our choosing, provided it has the same scalar curvature as $g_{S^q}$.  The resulting object satisfies Ricci flow wherever $u > 0$, but is not a manifold for $t > 0$.  Around the new points at the tip, the result has the topology of the cone over $F_2$.  The forward evolution has a scar as a result of its surgery.  

A special case of this situation is when $g_{F_2}$ is the standard metric on $F_2 = S^q / \Gamma$ for some group $\Gamma$.  This case is important because it cannot be ruled out by a pointwise curvature condition, and so it is relevant to trying to implement Ricci flow with surgery under curvature assumptions.  The resulting object after the singularity is an orbifold.   This case was dealt with in four dimensions in \cite{completecompactposiso}, and they removed a topological assumption of Hamilton's work in \cite{h_posiso} by considering Ricci flow with orbifold singularities.

Of relevance to us is the case $q = 2k$ and $F_2 = S^k \times S^k$. In this case, the metric at the singular time has the form
\footnote{
  We are always lazy with writing the lifts of metrics and tensors etc.  Here we use the notation $\oplus$ to emphasize that the two terms $g_{S^q}$ which appear are different, one is the lift of the $g_{S^q}$ on the first factor, and the other is the list of the $g_{S^q}$ on the second factor.
}
\begin{align}
  g = dx^2 + u(x) g_{S^k} \oplus u(x) g_{S^k}.
\end{align}
It satisfies all of the conditions of a model pinch except for \ref{w_big}, since $u = w$ and $\ein_F = \ein$.  Since $S^k \times S^k$ is unstable under Ricci flow (we can perturb the size of one of the factors) we thought perhaps there could be two alternative forward evolutions where either of the factors becomes positive after the singular time.  We now believe that this is not possible, see Section \ref{w_big_sharpness}.

\subsection{Shape of the forward evolution}\label{section:shape_description}
In this section we describe various properties of the forward evolution $g(t)$ of a model pinch.  As time goes on, the metric continues to be a doubly warped product:
\begin{align}
  g(t) = a(x,t) dx^2 + u(x,t) g_{S^q} + w(x,t) g_{F}.
\end{align}
Furthermore, we prove that $u$ continues to be increasing in $x$.  Therefore we may continue to consider $v = u^{-1}|\nabla u|^2$ and $w$ as functions of $u$, and write the metric as
\begin{align}
  g(t) = \frac{du^2}{uv} + u g_{S^q} + w g_F.
\end{align}
For the initial metric, the derivatives of $u$ and $w$ are relatively small.  Therefore after investigating the curvature of warped products we see 
\begin{align}\label{rc_prish}
  -2\Rc(X,Y) \approx
   - \ein g_{S^q} - \ein_F g_F.
\end{align}

Forward in time, this approximation continues to hold for a short time, while the derivatives of $u$ and $w$ continue to be small.  We call the region where $v = u^{-1}|\nabla u|^2$ continues to be small the ``productish'' region.  Let $\nu(t) = V_0(\ein t)$.  The productish region is the set
\begin{align}
  \Omega_{prish} = \left\{(x,t): \frac{u(x,t)}{t \nu(t)} \geq \sigma_* \text{ and } u < u_* \right\}
\end{align}
for some sufficiently large $\sigma_*$ and small $u_*$.  In this region, we have $v \leq \epsilon$; by choosing $\sigma_*$ and $u_*$ we can have $\epsilon$ as small as we wish.

In the productish region, we get the approximations
\begin{align}
  v &\approx V_{prish}
      \defeq \frac{u + \ein t}{u} V_0\left(u + \ein t \right)
      \label{v_approx_prish}
  \\
  w &\approx W_{prish} \defeq W_0\left( u + \ein t \right) - \ein_F t
      \label{w_approx_prish}
      .
\end{align}
Note that these approximations would be exact if the approximation \eqref{rc_prish} were exact so $u(x,t) = u(x,0) - \ein t$, $w(x,t) = w(x,0)-\ein_F t$, and
\begin{align}
  v(x,t)
  = \frac{|\nabla u(x,t)|^2}{u(x,t)}
  = \frac{|\nabla u(x,0)|^2}{u(x,0)} \frac{u(x,0)}{u(x,t)}
  = V_0(u+\ein t) \frac{u + \ein t}{u}.
\end{align}

In Section \ref{corollaries_productish} we give some corollaries of our control in the productish region.

Now we come to a crucial juncture in the calculation of our approximate solution.  The approximations \eqref{v_approx_prish} and \eqref{w_approx_prish} work for $u(x,t) \gtrsim t \nu(t)$- in particular they work for $u \ll t$.  To understand the approximations for small $u$, put $\nu(t) = V_0(\ein t)$, $\omega(t) = W_0(\ein t)$ and write
\begin{align}
  V_{prish} &= (1 + \ein t/u)\nu(t)
  \frac{V_0\left( \ein t( 1 + \ein^{-1}u/t)\right)}{V_0(\ein t)} \\
  W_{prish} + \ein_F t &= \omega(t)
   \frac{W_0\left( \ein t( 1 + \ein^{-1}u/t)\right)}{W_0(\ein t)} 
\end{align}
Using our assumptions on $V_0$ and $W_0$, particularly \ref{modelpinch_reg}, we can estimate the quotients for $u/t \ll 1$.  
Then our approximations say
\begin{align}
  v &\approx \ein \sigma^{-1}(1 + \ein^{-1}(1 + \dd1{\nu}(t)) \nu(t) \sigma)\label{expose_v_approx}\\
  w + \ein_F t
    &\approx \omega(t)(1 + \ein^{-1}\dd1{\omega}(t)\nu(t) \sigma) \label{expose_w_approx} 
\end{align}
where $\sigma = u/(t\nu(t))$, $\dd1{\nu} = t\nu'(t)/\nu$, $\dd1{\omega} = t\omega'(t)/\omega$.

If the left end of the manifold is to be smooth and compact, $v$ cannot be small up to $u=0$.  In fact, $v \to 4$ is a necessary condition to have a smooth closed disc at the left endpoint.  At the left end, on the factor $I \times S^q$, we glue in a steady Bryant soliton of size $\approx t\nu(t) \eqdef \alpha(t)$.  This is a metric on $\real^{1+q}$ that moves only by diffeomorphisms under Ricci flow. We call the region where $\sigma$ stays small, where we see the Bryant soliton, the ``tip region''.  The asymptotics of the Bryant soliton as $u \to \infty$ match with the term $\ein \sigma^{-1}$ in \eqref{expose_v_approx}.   A steady soliton is in accordance with the fact that we expect scaling at a rate faster than $t$: as a general principle, if we scaled at rate $t$ we would expect an expanding soliton, whereas if we scale at a faster rate we find a steady soliton.

For the factor $F$, the warping function is approximately constant.  Therefore we expect to be able to attach a large $F$ factor to our Bryant soliton.  The approximate size of the unrescaled $F$ factor is $\omega(t)-\ein_F t = W(\ein t) - \ein_F t$.  Taking for simplicity the case $\ein_F \neq 0$, our assumptions imply that $\omega-\ein_F t \gtrsim t \gg t \nu(t)$.  Therefore when we scale by $t \nu(t)$ the size of this factor goes to infinity, and around any point it approaches a Euclidean factor.

Thus, the zeroth order approximation of the metric near the tip (in other words, the expected limit of the rescaled metric as $t \searrow 0$) is $(\text{Bryant Soliton}) \times (\text{Euclidean metric})$.  We can get this approximation in a region of the form
\begin{align}
  \Omega_{tip} = \left\{(x,t): \sigma < \nu^{-1/2} \right\}.
\end{align}
As $t \searrow 0$ (so $\nu \searrow 0$) this region covers the whole Bryant soliton.

We also need to find the first order approximation near the tip.  The perturbation has size $\approx \nu$.  The equation we get in space is
\begin{align}
  (\text{Linearization of Ricci Flow})[g_1] = g_0,
\end{align}
where $g_0$ and $g_1$ represent the zeroth and first order approximations.  This gives us an equation to solve for $g_1$.  On the $F$ factor, the solution coincides with the soliton potential, times $g_F$.  Our first order approximation matches with all of the terms in \eqref{expose_v_approx}, \eqref{expose_w_approx}.

In Section \ref{corollaries_tip} we give some corollaries of our control in the tip region.

\subsection{Sharpness and further questions}
\subsubsection{Regularity conditions \ref{modelpinch_reg}}
Note that an implication of $|u \partial_u W_0| + |u^2 \partial_u^2 W_0| < CW_0$ is that
$
  \frac{W_0(ru)}{W_0(u)}
$
can be bounded for small $r$, independently of $u$.  In particular, $W_0(u) = e^{u^{-1}}$ and $W_0(u) = e^{-u^{-1}}$ both do not satisfy our assumptions.  We cannot offer any guess as to whether our results hold for these functions.

As examples of wild profiles for $W_0$, consider $W_0(u) = 2 + \sin(\log(u))$ or $W_0(u) = u^{-1}$.  Note that if the initial metric has bounded length near $u = 0$, then this may appear bad.  Still, around any point where $u = u_\sharp$, rescaling by $u_{\sharp}$ we will see approximately a product metric on a long (length $\approx 1/\sqrt{V_0(u_\sharp)}$) scale.  Note in all cases, in the forward evolution $w$ is bounded and positive near the tip for finite time.

We can think of the conditions on $V_0$ in the same way, but it may be more reasonable to look at examples in terms of the arclength coordinate $s$.  So, consider the $I \times S^q$ part of the metric written as
\begin{align}
  ds^2 + u(s)^2 g_{S^q}, \quad s \in (L_0, \infty), \quad L_0 = 0 \text{ or } L_0 - \infty .
\end{align}
The condition that $\dd1{V_0} \defeq \frac{u \partial_u V_0(u)}{V_0(u)} < C$ actually says, in a sense, that $u$ must be small enough in terms of $s$.  (Written in terms of $s$, this condition will involve the functional inverse of $u$.)  The following functions satisfy the regularity conditions on $V$:
\begin{itemize}
\item $L_0 = 0$ and $u(s) = s^a |\log(s)|^b$, where $a > 2$ and $b \in \real$, or $a = 2$ and $b < 0$.
\item $L_0 = -\infty$ and $u(s) = |s|^{-a}\log(|s|)^{b}$, where $a > 0$ and $b \in \real$.
\item If we write $u(s) = \exp(-f)$ where $f \to \infty$ as $s \searrow L$, then the  condition that $|\dd1{V_0}|<C$ is equivalent to $(1/f')' < C$.  For example, $u(s) = \exp(-1/s), L_0 = 0$ or $u(s) = \exp(s)$, $L_0 = -\infty$ are both valid model pinches.
\end{itemize}

\subsubsection{The profile $\phi(s) = \log(|s|)^{-1}$}
Our results do not provide a forward evolution from the initial metric with $I = (-\infty, \infty)$, and $u(s) \sim \log(|s|)^{-2}$ at $s = -\infty$.  Note in that case 
\begin{align}
  v = u^{-1}|\nabla u|^2 = 4 \log(|s|)^{-4}s^{-2}
\end{align}
so
$
  V_0(u) = u^2 \exp(-2/u)
$.
Then
$u \partial_uV_0/V_0 = 2u^{-1} + 2$,
which violates condition \ref{modelpinch_reg}.  It would be interesting to know whether there is a solution to Ricci flow emerging from this example.

In this example, for any $r > 0$, the region which looks approximately like a skinny cylinder of radius $r$ is quite long in comparison to $r$.  More precisely, fixing $\epsilon$ there is a $C_\epsilon > 1$ such that for any $r$ we have the following. The region where the radius $\phi$ is within a factor of $(1 \pm \epsilon)$ of $r$ has length $(C_\epsilon)^{1/r^2}$.  Maybe this means the cylinder must collapse before anything far away can save it.

\subsubsection{The conditions on the size of $W_0$ \ref{w_big}}\label{w_big_sharpness}
For simplicity say $(F, g_F) = (S^q, g_{S^q})$.
We find it striking that in the case $W_0(u) = (1 + c)u$, for the initial metric $w$ and $u$ are comparable, but if we rescale the forward evolution to keep the curvature bounded at the origin, the $w$ factor goes to infinity.

  We believe that it is possible to relax the condition \ref{w_big} and still have a forward evolution with the same asymptotics.  Let's rapidly go through a calculation.  Suppose $W_0(u) = (1 + H_0(u))u$, where $H_0(u) \searrow 0$ (violating \ref{w_big}).  Calculating from \eqref{w_approx_prish}, in the productish region where $u > C t \nu(t)$,
\begin{align}
  w
  &\approx (1 + H_0(u + \ein t))(u + \ein t) - \ein t\\
  &= u + H_0(u + \ein t)(u + \ein t).
\end{align}
If we write $\eta(t) = H_0(\ein t)$ then for points where $C t \nu(t) < u \ll t$ we have (recall $\sigma \defeq \frac{u}{t\nu(t)}$):
\begin{align}
  \frac{w}{t\nu(t)} &\approx \sigma  + \ein \frac{\eta(t)}{\nu(t)}.\label{speculative_tip_asympts}
\end{align}

First consider the case $H_0(u) \gg V_0(u)$ (i.e. $\eta(t) \gg \nu(t)$), which is still a weaker condition than \ref{w_big}.  Then scaling $w$ in the same way we scale $u$ sends it to infinity, and $w$ is approximately a constant.  We expect this case to behave similarly to the case that is rigorously dealt with in this paper.  The major road block in dealing with it, for us, is reproving Lemma \ref{lemma:y_control_tip} which controls the derivative of $w$ and therefore controls the level of interaction between the evolution of $v$ and $w$.  Unfortunately our method gives us no more wiggle room in this lemma, but we think that our control on the distance from $w$ to a constant is not optimal.

To continue with our speculation, consider the case when $H_0(u) = c_0 V_0(u)$.  Then in \eqref{speculative_tip_asympts} we find
$  \frac{w}{t\nu(t)} = \sigma + c_0 \ein$.
We still would have the approximation \eqref{expose_v_approx} for $v$.  This gives us the asymptotics for an \emph{Ivey soliton} \cite{ivey}, which is a complete soliton on $\real^{1+q} \times F$ of the form $dx^2 + u_{sol}(x)g_{S^q} + w_{sol}(x)g_F$. (The function $u(x)$ goes to zero at $x = 0$, and $w(x)$ stays positive.)  So, in this case we expect to see the Ivey soliton in the rescaled limit at the tip.  This case should be more difficult, because the system is more strongly coupled.

In the case when $H_0(u) \ll V_0(u)$, we do not think that there is a smooth forward evolution, but there may be a forward evolution with bounded Ricci curvature everywhere.  In this forward evolution we glue in a Bryant of dimension $1 + (q + q)$, but with the sphere fibers $S^{q + q}$ replaced with the Einstein manifold $S^q \times S^q$ (with proper scaling to make the scalar curvatures match).  The case $H_0(u) = 0$ is the situation discussed in Section \ref{scarring}.

The reason we do not expect a smooth forward evolution is the following: consider $H_0(u) = \epsilon V_0(u)$. Then, we are in the case when we expect the Ivey soliton.  The exact asymptotics of the Ivey soliton we get are determined by $\epsilon$, and as $\epsilon \searrow 0$, this family of Ivey solitons approaches the Bryant soliton with $S^{q + q}$ replaced with $S^q \times S^q$.  Therefore, even trying to approximate the singular initial metric with smooth ones it seems we are led to the nonsmooth case.

\subsubsection{Pinched sphere warped products over other bases}
Consider a manifold with boundary $(B, \partial B)$ with a metric $g_B$ and a function $u: B \to \real_+$ which tends to zero at the boundary such that $v = u^{-1}|\nabla u|^2$ also goes to zero.  Let's stipulate that everywhere $|\Rm_{g_B}| \ll u^{-2}|\nabla u|^2$.  Now we want to ask whether there is a forward evolution from the metric $g = g_B + u g_{S^q}$.  Note that model pinches with $W_0(u) \gg u$ are a special case, where $g_B = \frac{du^2}{u V_0(u)} + W_0(u)g_F$.

The nice property of the doubly warped products is that the hessian of $u$ is easier to control, because the level sets of $u$ are equidistant.  It should be relatively possible to extend to other such cases, like cohomogeneity-one manifolds.

\subsubsection{The closeness required in the asymmetric case}
Our condition for Theorem \ref{theorem:unsymmetrical_flow} is that the distance between the asymmetric metric and the model pinch goes to zero near the tip at least as fast as a specific rate.  There is a sense in which this is probably not optimal.  Our proof technique yields more than is stated in Theorem \ref{theorem:unsymmetrical_flow}: it says that $g(t)$ actually stays close to the forward evolution from $g_{mp}$.  We make no attempt to update the approximate model pinch, whereas perhaps the best warped-product forward evolution not the forward evolution from the initial warped-product.

A theorem that we can compare Theorem \ref{theorem:unsymmetrical_flow} to is Theorem 1.3 of \cite{conicalSing}.  That theorem constructs forward evolution from metrics close to having conical singularities.  There, $g_{c}$ is a cone and the requirement (1.1) is that near the singularity the singular metric $g$ satisfies $|g - \Phi^* g_c| \leq \epsilon_0$.  This seems stronger than our theorem, because it makes no exact assumption on the rate at which it approaches the model singularity.  On the other hand, the case of a singly-warped cone (which our theorem does not handle) is the case when $V_0$ is constant, so perhaps our condition is not dissimilar.

\subsection{Notation and preliminaries}\label{overview_symmetric}
More notation is densely listed in Appendix \ref{appendix_notation}.

Partial derivatives are denoted with $\partial_{\cdot}$.  For an arbitrary function $u$ with nonzero derivative, we have $\partial_u = |\nabla u|^{-2}\nabla_{\grad u}$ which is the derivative with respect to $u$, using a metric.  We define $\partial_{t;u} = \partial_t - (\partial_t u)\partial_u$ which is the derivative with respect to time along a curve which moves orthogonally to the level sets of $u$ in order to keep $u$ constant.

We adopt the shorthand that when stating hypotheses, the statement ``$x \leq \barr x(y, z)$'' means ``there exists an $\barr x$, depending on $y$ and $z$, such that if $x \leq \barr x$, the following holds.''  This allows us to quickly state ``if $x \leq \barr x(y,z)$ and $w \leq \barr w(x, y)$ then \dots''.

\subsubsection{Equations}
We can consider our metrics as singly warped products of spheres over a general base: $g(t) = g_B(t) + u(t)g_{S^q}$ where for each $t$, $u(t):B\to \real_+$.  Under Ricci flow, $u$ evolves by
\begin{align}
  \square_B u = - \ein + \on4 (\ein-2) v,\label{evo_u_in_overview}
\end{align}
where $\square_B$ is the heat operator $\square_B = \partial_t - \lap_B$ and $\lap_B$ is the laplacian for $g_B(t)$.  Equivalently,
\begin{align}
  \square_M u = - \ein -  v.
\end{align}
where $\square_M$ is the heat operator for $g$.  Similarly, the function $w$ which controls the size of $g_F$ evolves by
\begin{align}
  \square_M w = - \ein_F - w^{-1}|\nabla w|^2 = -\ein_F - y,
\end{align}
where we have defined $y = w^{-1}|\nabla w|^2$.
We use this point of view to find the approximate solutions in the productish region.  For an exposition of these equations for Ricci flow on warped products, see Section \ref{warped_product_section}.  

For finer control, we need the evolution of $v$ and $w$ as functions of $u$.  These are derived in Sections \ref{v_deriving_section} and \ref{additional_wp}. We have
\begin{align}
  \partial_{t;u} v
  &= u v \partial_u^2 v - \oh u (\partial_u v)^2 \label{vevo_basic2}\\
  &+ \ein \left(1 - \on4 v \right) u^{-1} v +  \ein \partial_u v \\
  &- 2 (\fund^2) v,
\end{align}
where $\kappa^2 = \on4 (dim(F))w^{-2}u^2 v^2 (\partial_u w)^2$, and 
\begin{align}
    \partial_{t;u}w - u v \partial_u^2 w 
    = - \ein_F - y + \ein \partial_u w - \ein/2 v\partial_u w \label{evo_w_in_u2}.
\end{align}

\subsubsection{Regularity}\label{regularity}
We work in $C^{2, \eta}$ H\"older spaces using interior Schauder estimates.  Bamler wrote a clean statement of the interior Schauder estimates he needed in \cite{stabilBamler} (Section 2.5).  We co-opt this statement, because it is exactly what we need except for standard generalizations.  His statement does not allow for the time-dependence of the coefficients that we will have, but in fact the proof carries through exactly; the time dependence enters in the estimate on the $C^{2m-2, 2\alpha; m-1, \alpha}$ norm of $f_i$ in the middle of page 424.  Furthermore, his statement does not allow the parabolic ball to hit the initial time, as we will need to. Accounting for this is also standard. In the proof of Lemma 2.6 of \cite{stabilBamler}, one may apply Exercise 9.2.5 of \cite{krylov} rather than Theorem 8.11.1 of \cite{krylov}.  

\subsubsection{Ricci-DeTurck flow}\label{section:rdt}
We use Ricci-DeTurck flow to control the Ricci flow of metrics near our warped product forward evolutions.  For two metrics $(M, g)$ and $(M, \tilde g)$ we define
\begin{align}
  (V[g,\tilde g])^i = g^{ab}\left(
  \left( \Gamma_g\right)_{ab}^i  
  - \left( \Gamma_{\tilde g} \right)_{ab}^i
  \right)
\end{align}
which is the map Laplacian of the identity map from $(M, g)$ to $(M, \tilde g)$.  For a time-dependent metric we define $\Rf[g] = \partial_t g - (-2 \Rc[g])$.  The Ricci-DeTurck flow from $g(0)$ with background metric $\tilde g$ is the solution to 
\begin{align}
  g(0)&\quad  \text{ given},\\
  \Rf[g]
  &= \lie_{V[g,\tilde g]}g. \label{RDT_def}
\end{align}
We allow $\tilde g$ to also be time-dependent.
It will be useful to consider Ricci flow and Ricci-DeTurck flow modified by a vector field.  We set $\Rf_X[g] = \partial_t g - (-2 \Rc[g] - \lie_X g)$, and if $\Rf_X[g] = \lie_{V[g, \tilde g]}g$ then we say that $g$ is a solution to Ricci-DeTurck flow, modified by $X$, with background metric $\tilde g$.

We will not use the exact form of the evolution of $h$, except to know that we can apply regularity.  What we will use is the following evolution of $|h|$ and $|h|^2$.  For $p \in M$ We set 
\begin{align}\label{rmplus_def}
  \rmplus(p) = \max_{h \in Sym_2(T_pM) : |h| = 1}\ip{\Rm[h]}{h}(p).
\end{align}
Now, assuming that $\Rf_X[\tilde g] = 0$, and that $|h| \leq \oh$, for $y = |h|^2$ we have (we allow $c_0$ to change from line to line)
\begin{align}
  \square_{X, \tilde g, g} y
  &\leq 4 \rmplus y - 2(1 - c_0 y^{1/2})|\nabla h|^2 + c_0 |\Rm|y^{3/2}. \label{dtevo_square}
\end{align}
For $z = |h|$ we have
\begin{align}
  \square_{X, \tilde g, g}z
  &\leq 2 \rmplus z + c_0 \left(|\Rm| z^2 + |\nabla h|^2 \right).\label{rdt_norm}
\end{align}
We show these in Appendix \ref{rcdt}.


\section{Control in the productish region}\label{section:productish}
In this section we create some interior estimates for our warped-product forward evolution.  We define the productish region as a region of the form
\begin{align}\label{omega_prish_def}
  \Omega_{prish}=
  \left\{
  (u, t): u + \ein t < u_* \text{ and } \sigma = \frac{u}{tV_0(\ein t)} > \sigma_*
  \right\}.
\end{align}
In particular, $\Omega_{prish}$ touches an open part of the initial time slice (see Figure \ref{region_map}).  All constants and definitions in this section implicitly depend on dimensions, $g_F$, and the chosen functions satisfying the model pinch conditions $V_0$ and $W_0$.
We define $\hat w = w + \ein_F t$.  In the productish region, we will have approximations of the form
\begin{align}
  v \approx V
  \defeq
  \left(\frac{u + \ein t}{u} \right)
  V_0(u + \ein t),
\quad
  \hat w 
  \approx \hat W
  \defeq
  W_0(u + \ein t).
\end{align}
These come directly from the calculations in Appendix \ref{nearly_cnst_pde_sect}.  They may be guessed by ignoring all terms in the evolution of $u$ and $w$ which depend on space derivatives of $u$ or $w$.  We will prove that $v$ is between $V^-$ and $V^+$, and $\hat w$ is between $\hat W^-$ and $\hat W^+$, where
\begin{align}
  V^{\pm} = (1 \pm D V)V, \quad \hat W^{\pm} = (1 \pm D V)\hat W. \label{barriers_def}
\end{align}
We call $V^{\pm}$ and $W^{\pm}$ the barriers.

We make some definitions to state the main result of this section.  We will assume that $g(t) = a(x,t) dx^2 + u(x,t) g_{S^q} + w(x,t) g_F$ is a solution to Ricci flow on $[T_1, T_2]$.  Our definitions depend on constants $u_*$, $\sigma_*$, controlling the size of the produtish region, and $D$ controlling the separation of the barriers, as well as $c_{safe}$ and $C_{reg}$.  

\begin{definition}\label{productish_barricaded}
  We say that $g(t)$ is \emph{barricaded} (by the productish barriers)
  \footnote{In this section we only say ``barricaded'' but in Section \ref{section:full_flow} we will have to refer to either barricaded by the productish barriers, or barricaded by the tip barriers.}
  at a point if it satisfies
  $
    V^- < v < V^+$ and $\hat W^- < \hat w < \hat W^+$
  at that point.
  
  We say that $g(t)$ is \emph{initially controlled in the productish region} if
  at $t = T_1$ and for all points satisfying $(1/2)\sigma_* T_1 \nu(T_1) < u < 2u_*$
  it is barricaded
  and
  \begin{align}
      \frac{\twoeta{v-V}{u/2}{(du)^2}}{V}
      +
      \frac{\twoeta{w-W}{u/2}{(du)^2}}{W}
    < c_{safe}C_{reg}DV, \label{prish_regularity_ineq_v}
  \end{align}

  We say that $g(t)$ is \emph{barricaded at the left of the productish region} if it is barricaded for all points satisfying $(1/2)\sigma_* t \nu(t) < u < \sigma_* t \nu(t)$ and $t \in [T_1, T_2]$.

  We say that $g(t)$ is \emph{barricaded at the right of the productish region} if it is barricaded for all points satisfying $u_* < u < 2 u_*$ and $t \in [T_1, T_2]$.

  We say that $g(t)$ is \emph{controlled in the productish region} if for all points in $\Omega_{prish}$,
  \begin{enumerate}[label=(P\arabic{*}), ref=(P\arabic{*})]
  \item \label{conc:prish_barrier} The solution is barricaded.
  \item \label{conc:prish_reg} We have the inequality
    \begin{align} 
      \frac{\twoeta{v-V}{u/2}{(du)^2}}{V}
      +
      \frac{\twoeta{w-W}{u/2}{(du)^2}}{W} < C_{reg}DV. \label{prish_regularity_ineq_v}
    \end{align}
  \end{enumerate}
\end{definition}

\begin{lemma}\label{main_prish_estimates}
  There is a $c_{safe}$ such that if we let $C_{reg} > \berr C_{reg}$,  $D > \berr D$, $u_* < \barr u_*(D, C_{reg})$, and $\sigma_* > \berr \sigma_*(D, C_{reg})$, there is a $T_*$ depending on all other parameters with the following property.

  Suppose $0 < T_1 < T_2 < T_*$ and $g(t)$ is defined on $[T_1, T_2]$, initially controlled, and barricaded at the left and the right of the productish region. Then $g(t)$ is controlled in the productish region, for all times in $[T_1, T_2]$.
\end{lemma}

In proving the conclusions of Lemma \ref{main_prish_estimates}, we can assume that they hold on the interval $[T_1, T_2)$.  This is because they hold strictly at the initial time by assumption, so we can consider $T_2$ to be the infimum of the times at which the fail.  This extra assumption is usually useful for controlling terms when we don't care about the exact constant involved, because in any case we can choose our constants $u_*$, $\sigma_*$, and $T_*$ so that it is as small as we want  (see e.g. Lemma \ref{lemma:ycontrol_productish}).

With this in mind, Lemma \ref{main_prish_estimates} will be proven by Lemmas \ref{barriers_prish} and \ref{regularity_prish} below, which show items \ref{conc:prish_barrier} and \ref{conc:prish_reg} respectively.  First, in Section \ref{examining_section}, we inspect our approximations $V$ and $W$ more closely.

\subsection{Examining our approximate solution}\label{examining_section}
We are claiming that $v(p, t) \approx V(u(p,t),t)$ where $V$ is the function
\begin{align}
  \label{v_expr}
  \V(u,t) = \frac{u + \ein t}{u}\V_0(u + \ein t) = \left( 1 + \ein \frac{t}{u} \right)\V_0(u + \ein t)
\end{align}
The effectiveness of the barriers defined in \eqref{barriers_def} is dependent on $V$ staying small.  In this section, we prove Lemma \ref{lemma:V_small} which tells us that $V$ does stays small exactly in the productish region $\Omega_{prish}$, and also gives another description of $V$ and $W$.   The proof is elementary, but the reformulation of $V$ is key to how the productish region hooks up with the tip region.

We aim to understand where $V$ stays small.  An apparent scary term in \eqref{v_expr} is $t/u$.  Defining $\rho = u/t$, we can write
$
  \label{eq:79}
  \V = \left( 1 + \ein \rho^{-1} \right) \V_0(u + \ein t)
$.
If we keep in mind that our main assumption on $\V_0$ is that $\V_0(u) = o(1, u \to 0)$, then the following lemma, which says something about where $V$ is small, is immediately apparent.
\begin{lemma}\label{lemma:weaker_V_small}
  Let $\epsilon$ be given.  For any $\rho_*$ there is $u_*(\epsilon)$ and $T_*(\rho_*, \epsilon)$ so that if $t < T_*$, $u < u_*$, and $u/t > \rho*$ then $V < \epsilon$.
\end{lemma}

The discussion is not over: $V$ does not get large if we fix $\rho$  and send $u + \ein t \searrow 0$, as the factor $V_0(u + \ein t)$ helps us.  To understand this factor better, let $\nu(t) = V_0(\ein t)$.
Then by definition,
\begin{align}
  \label{eq:81}
  V_0(u + \ein t) = \nu(t) \frac{V_0((1 + u/(\ein t)) \ein t)}{V_0(\ein t)}.
\end{align}
Now we can use the regularity assumption on $V_0$ \ref{modelpinch_reg} to calculate using the Taylor expansion:
\begin{align}
  V_0(u + \ein t)
  &= V_0(\ein t) + u V_0'(\ein t) + u^2 V_0''((1 + r) \ein t) \\
  &= \nu(t)
    +
    \nu(t)^2
    \frac{u}{t \nu(t)}
    \ein^{-1} \left( t \ein \frac{V_0'(\ein t)}{V_0(\ein t)} \right)
    + u^2 V_0''((1 + r) \ein t).
\end{align}
Here $r \in [0, u/(\ein t)]$ comes from the remainder term in the Taylor expansion.  Now let $\dd1{\nu}(t) = \frac{t \partial_t\nu(t)}{\nu(t)}$ and $\sigma = \frac{u}{t\nu(t)}$, and calculate further:
\begin{align}
  V_0(u + \ein t)
  &= \nu + \nu^2 \sigma \ein^{-1}\dd1{\nu} \\
  &+ u^2 V_0''(\ein t) + u^2 \left( V_0''((1 + r)\ein t) - V_0(\ein t) \right) \\
  &= \nu + \nu^2 \sigma \ein^{-1}\dd1{\nu} \\
  &+ \sigma^2 \nu^3 \frac{t^2 V_0''(\ein t)}{V_0(\ein t)}
    + \sigma^{2+\eta} \nu^{3+\eta}
    \left(
    (t/u)^{\eta}
    \frac{V_0''((1 + r)\ein t) - V_0(\ein t)}{V_0(\ein t)}
    \right)\\
  &=  \nu + \nu^2 \sigma \ein^{-1}\dd1{\nu} + O(\sigma^2 \nu^3).
\end{align}
Here we used more of the regularity assumption \ref{modelpinch_reg}.  Now coming back to our expression for $V(u,t)$ and manipulating it,
\begin{align}
  \label{eq:84}
  V(u, t)
  &= (1 + \ein \rho^{-1})V_0(u + \ein t) \\
  &= \ein \sigma^{-1}\nu^{-1}(1 + \ein^{-1}\sigma\nu)V_0(u + \ein t) \\
  &= \ein \sigma^{-1}\left( 1 + (1 + \dd1 \nu)\ein^{-1}\nu\sigma + O((\nu \sigma)^2) \right). \label{V_expr_smallrho}
\end{align}
This makes it apparent that if we look at where $\sigma > \sigma_*$ for some large $\sigma_*$, $V$ is still small.  We present Lemma \ref{lemma:V_small}.
\begin{lemma}\label{lemma:V_small}
  Let $\epsilon$ be given.  If $\sigma_* > \berr \sigma_*(\epsilon)$ and $u_* < \barr u_*(\epsilon)$, and $T_* < \barr T_*(\sigma_*, u_*, \epsilon)$, then $V < \epsilon$ in the productish region \eqref{omega_prish_def}.
\end{lemma}
\begin{proof}[Proof. (Lemma \ref{lemma:V_small})]
  First, choose $\underline{\sigma}_*$ small enough, and $\barr T_*$ at least small enough, so that $(\sigma^{-1} + \nu) < \epsilon/100$ for all $u, t$ satisfying $\sigma > \underline{\sigma}_*$ and $t < \barr T_*$.  Next, by the expression \eqref{V_expr_smallrho}, we can choose $\rho_*$, and decrease $\barr T_*$, so that for $\sigma > \sigma_*$ and $\rho = \nu \sigma < \rho_*$, we have $V < \epsilon/50$.  Finally, by Lemma \ref{lemma:weaker_V_small} we can chose $\overline{u}_*$ so that $V < \epsilon$ for all $u,t$ satisfying $\rho > \rho_*$ and $u < \overline{u}_*$.
\end{proof}

We also examine the approximate solution for $\hat w$, namely $\hat W = W_0 \circ U_0$ (so $W = W_0 \circ U - \ein_F t$). Similarly to how we handled $V$, we write $\omega(t) = W_0(\ein t)$ and we find  
\begin{align}
  \hat W
  &= \omega(t) \left(
    1 + \mu^{-1}\nu \sigma \dd1{\omega}(t) + O((\nu \sigma)^2)
    \right)\label{Wprish_asymptotic_exp}
\end{align}

\subsection{Trapping between barriers}

Recall the equations of warped product Ricci flow (see section \ref{rf_warped_products}).  The functions $u$ and $\hat w = w + \ein_F t$ satisfy
\begin{align}
  \square_M u &= - \ein + c_{v} v, \\
  \square_M \hat w &= - y = - \frac{|\nabla \hat w|^2}{\hat w - \ein_F t}
\end{align}
First, we find bounds given to us by our regularity \ref{conc:prish_reg}.
\begin{lemma}\label{lemma:ycontrol_productish}
  Suppose we are in the setting of Lemma \ref{main_prish_estimates}.  Assume additionally that items \ref{conc:prish_barrier} and  \ref{conc:prish_reg} hold on $[T_1, T_2)$. If $\sigma_* > \berr \sigma_*(D, C_{reg})$ and $u_* < \berr u_*(D, C_{reg})$ then
  \begin{align}
    \frac{uy}{vw} < C, \quad
    \frac{|\nabla\nabla u|}{v} \leq C
  \end{align}
  in $\Omega_{prish}$, where $C$  depends only on the initial data.
\end{lemma}
\begin{proof}
  First we tackle $\frac{uy}{vw}$.  Note that
  $
    \frac{uy}{vw}
    = \frac{ u^2 |\nabla w|^2}{|\nabla u|^2 w^2} 
    = \left( \frac{u \partial_u w}{w} \right)^2 
    $
    so by item  \ref{conc:prish_reg}
  \begin{align}
    \frac{uy}{vw} \leq \left( \frac{u\partial_u W}{w} + C_{reg}DV\frac{W}{w} \right)^2.\label{bnda}
  \end{align}
  By Lemma \ref{lemma:V_small} we can decrease $u_*$ and increase $\sigma_*$ so that $C_{reg}DV < 1$ and $\frac{W}{W^-} < 2$.  Then since $w$ is between its barriers, we can bound $w$ in \eqref{bnda} in terms of $W$.
  \begin{align}
    \frac{uy}{vw}
    &\leq 4 \left( \frac{u\partial_u W}{W} + 1 \right)^2
      =
      4 \left(
      \frac{u \partial_u W_0(u + \ein t)}{W_0(u + \ein t) - \ein_F t} + 1
      \right)^2\\
    &=
      4 \left(
      \frac{W_0(u + \ein t)}{W_0(u + \ein t) - \ein_F t}
      \dd1{W_0}(u + \ein t)
      + 1
      \right)^2.\label{bndb}
  \end{align}
  By the assumption \ref{w_big} on $W_0$,
  \begin{align}
    \frac{W_0(u + \ein t)}{W_0(u + \ein t) - \ein_F t}
    &= \frac{1}{1 - \frac{\ein_F t}{W_0(u + \ein t)}}
      \leq  \frac{1}{1 - \frac{1}{1 + c}}
  \end{align}
  Therefore \eqref{bndb} is bounded by a constant depending only on the initial data, using also our assumption \ref{modelpinch_reg} that $\dd1 W_0 = \frac{u\partial_u W_0}{W_0}$ is bounded.

  Now we bound the hessian, which requires the geometry of the warped products.  Thinking of the multiply warped product manifold as a family of equidistant hypersurfaces, the norm of the hessian of a function depending only on the hypersurface is given by
  \begin{align}
    |\nabla \nabla f|^2
    &= \on4 |\nabla f|^{-4}\ip{\nabla |\nabla f|^2}{ \nabla f}^2
      + |\nabla f|^2 |A|^2
  \end{align}
  where $|A|^2$ is the norm of the second fundamental form of the hypersurfaces, which in our case is 
  \begin{align}
    |A|^2 = \on4 q u^{-1} v + \on4 \dim(F) w^{-1}y.
  \end{align}
  Therefore we find,
  \begin{align}
    |\nabla \nabla u|^2
    &\leq C \left(
    |\nabla u|^{-4}|\nabla |\nabla u|^2||\nabla u|^2
    + u^{-1}v |\nabla u|^2 + w^{-1}y |\nabla u|^2 
      \right) \\
    &= C \left(
      u^{-1}v^{-1} |\nabla (uv)|^2
      + v^2 + w^{-1}yuv
      \right)\\
    &\leq C\left(
      uv^{-1}|\nabla v|^2 + \left( 1 + \frac{u}{v}\frac{y}{w} \right)v^2
      \right) \\
    &=C\left(
      \left( \frac{u \partial_u v}{v} \right)^2 + \left( 1 + \frac{u}{v}\frac{y}{w} \right)
      \right) v^2.
  \end{align}
  By \ref{conc:prish_reg}, and by the bound on $\dd1{V}$ in Lemma \ref{approx_soln}, we get the desired inequality.
 \end{proof}

 Now we are in the position to prove that \ref{conc:prish_barrier} continues to hold.

\begin{lemma}\label{barriers_prish}
  Suppose we are in the setting of Lemma \ref{main_prish_estimates}, and items \ref{conc:prish_barrier} and \ref{conc:prish_reg} holds on $[T_1, T_2)$.
  If $D> \berr D$, $u_* < \barr u_*(D, C_{reg})$, $\sigma_* > \berr \sigma_*(D, C_{reg})$, $T_* < \barr T_*(D, u_*, \sigma_*)$ then \ref{conc:prish_barrier} holds at $t = T_2$.
\end{lemma}
\begin{proof}
  By the evolution equation for $v$, \eqref{EQ:22}, and by our bound on the hessian from Lemma \ref{lemma:ycontrol_productish}, we have
  $
    |(\square - u^{-1} \ein ) v| \leq C v^2
  $
  for a constant $C$ depending only on the initial data $V_0$ and $W_0$.  Also, by the evolution equation for $w$ and our bound on $y$, we have
  $
    |\square \hat w | \leq C u^{-1} v \hat w
  $.

  Therefore, Lemma \ref{sup_solns} shows that, if we chose $D$ sufficiently large, $V^{\pm}$ and $W^{\pm}$ are sub- and supersolutions for the equations satisfied by $v$ and $w$.  The maximum principle proves the claim.
\end{proof}

\subsection{Regularity}
\begin{lemma}\label{regularity_prish}
  Suppose we are in the setting of Lemma \ref{main_prish_estimates}. We can choose $\barr c_{safe}$, $\berr C_{reg}$, $\berr u_*$, and $\berr T_*$ such that if \ref{conc:prish_barrier} holds for $t \in [T_1, T_2)$ then \ref{conc:prish_reg} holds for $t\in [T_1, T_2]$.
\end{lemma}
\begin{proof}
We prove this theorem by applying parabolic regularity to the equations solved by $v$ and $w$ in terms of $u$.  From \eqref{vevo_basic2} and \eqref{evo_w_in_u2}, we have the equations
\begin{align}
  \partial_{t; u} v - \ein \partial_u v - \ein u^{-1} v
  &= \left( u v \right) \partial^2_u v - \oh u \left( \partial_u v \right)^2 \\
  &+ a_1 v \partial_u v + a_2 u^{-1} v^2 + a_3 \left( \frac{vu}{\hat w - \ein_F t} \right)^2 \left( \partial_u \hat w \right)^2, \\
  \partial_{t; u} \hat w - \ein \partial_u \hat w 
  &= \left( u v \right) \partial_u^2  \hat w + b_1 v \partial_u \hat w + b_2 \left( \frac{vu}{\hat w - \ein_F t} \right) \left( \partial_u \hat w \right)^2,
\end{align}
where $a_1, a_2, a_3$ and $b_1, b_2$ are constants.

We let $\hu = u + \ein t$, $\hat v = \hu^{-1}u v$, and similarly $\hat V = \hu^{-1} u V = V_0(\hu)$.  Calculate,
\begin{align}
  v &= u^{-1} \hu \hat v, \\
  \partial_u v &= - \ein t u^{-2} \hat v + u^{-1} \hu \partial_\hu \hat v, \\
  \partial_u^2 v &=
  2 \ein t u^{-3} \hat v
  - 2 \ein t u^{-2} \partial_\hu \hat v
                   + u^{-1} \hu \partial_\hu^2 \hat v.
\end{align}
Also note that
\begin{align}
  \partial_{t; u} v - \ein \partial_u v - \ein u^{-1} v
  = u^{-1} \hu \partial_{t; \hu} \hat v
    \quad \text{and} \quad
  \partial_{t;u} w - \ein \partial_u  w + \ein_F 
  = \partial_{t; \hu} \hat w.
\end{align}
This lets us derive the following equation for $\hat v$: for some constants
$c_1, c_2, c_3, c_4$,
\begin{align}
  \partial_{t;\hu} \hat v
  &=  \hu\hat v  \partial_\hu ^2 \hat v \\
  &+ c_1 t u^{-2} \hat v^2
    + c_2 t u^{-1} \hat v \partial_\hu \hat v\\
  &+ c_3 \hu (\partial_\hu \hat v)^2
    + c_4 u \hu^{-1}\left( \frac{ v u }{w} \right)^2 (\partial_u \hat w)^2.
\end{align}
We also derive the evolution for $\hat w$:
\begin{align}
  \partial_{t;w} \hat w
  &= (\hu \hat v) \partial_\hu^2 \hat w 
    + b_1 (u^{-1}\hu \hat v) \partial_\hu \hat w
    + b_2 \left( \frac{\hu \hat v}{w} \right) \left( \partial_\hu \hat w\right)^2.
\end{align}

Now, let $u_1, t_1$ be any point in the productish region, let $\hu_1 = u_1 + \mu t_1$, $\hat v_1 = \hat v(u_1, t_1)$, and $\hat w_1 = \hat w(u_1, t_1)$.  Divide through in both equations by $\hu_1 \hat v_1$.  Also divide the equation for $\hat v$ by $\hat V_1 = \hat V(u_1, t_1) = V_0(\hu_1)$ and the equation for $\hat w$ by $\hat W_1 = \hat W(u_1, t_1) = W_0(\hu_1)$.
\begin{align}
  \frac{1}{\hu_1 \hat V_1} \partial_{t;\hu} \left( \frac{\hat v}{\hat V_1} \right)
  &= \left[ \frac{\hu \hat v}{\hu_1 \hat V_1}  \right] \partial_\hu^2 \hat v \\
  &+ c_1 \left[ \frac{t}{\hu_1} \frac{\hat v}{\hat V_1}\frac{u_1^2}{u^2} \right]
    u_1^{-2}  \left( \frac{\hat v}{\hat V_1} \right)
    + c_2 \left[ \frac{t}{\hu_1} \frac{\hat v}{\hat V_1}\frac{u_1}{u} \right]
    u_1^{-1}\partial_\hu \left( \frac{\hat v}{\hat V_1} \right) \\
  &+ c_3 \left[ \frac{\hu}{\hu_1} \right]
    \left(\partial_\hu \left( \frac{\hat v}{\hat V_1}\right)  \right)^2
    + c_4 \left[\frac{v}{v_1} \frac{u^2}{\hu \hu_1}\frac{w_1^2}{w^2}\frac{v}{v_1}\right]
     \left(\partial_u \left( \frac{w}{w_1} \right) \right)^2
\end{align}
\begin{align}
  \frac{1}{\hu_1 \hat V_1}\partial_{t;\hu} \left( \frac{\hat w}{\hat W_1} \right)
  &= \left[ \frac{\hu \hat v}{\hu_1 \hat V_1} \right]
    \partial_\hu^2 \left( \frac{\hat w}{\hat W_1} \right) \\
  &+ b_1 \left[ \frac{\hu \hat v}{\hu_1 \hat V_1} \right]
    u^{-1}\partial_\hu \left( \frac{\hat w}{\hat W_1} \right)
    + b_2 \left[ \frac{\hu}{\hu_1} \frac{\hat v}{ \hat V_1} \frac{\hat w}{w} \frac{\hat w}{\hat W_1} \right] \left( \partial_\hu \left( \frac{\hat w}{\hat W_1} \right) \right)^2
\end{align}
We will apply interior parabolic regularity to these equations, in the region
\begin{align}
  \Xi = \{(\hu,t) : (\hu, t) \in
  [\hu_1 - \oh u_1, \hu_1 + \oh u_1]
  \times
  [t_1 - \max(T_1, t_1 - \oh \hu_1^{-1}v_1^{-1}u_1^2), t_1],
  \}
\end{align}
which is a parabolic ball around $(\hu_1, t_1)$ of radius $\oh u_1$, if we were to scale time to $\hat t = \hu_1 \hat V_1 t$.  We have written the equation so that the factors in square brackets are smooth functions of $u$, $t$,  $\frac{\hat v}{\hat V_1}$, and  $\frac{\hat w}{\hat W_1}$ in this parabolic ball- this requires the knowledge that $v$ and $w$ are trapped between our barriers, so for example $\frac{w}{w_1}$ is not too far from 1 within $\Xi$.  The important thing about this smoothness is that we have bounds on relevant quantities (the $C^{2, \eta}$ norms of the functions) are not dependent on $u_1$ or $t_1$.

All in all, we can apply regularity to bound the $\hu$ derivatives of the functions $\frac{\hat v}{\hat V_1} - \frac{\hat V}{\hat V_1}$ and  $\frac{\hat w}{\hat W_1} - \frac{\hat W}{\hat W_1}$.  Our barriers tell us that the $C^0$ norm for both of these, in $\Xi$, is bounded by $C D V(u_1, t_1)$, where $C$ depends on on the initial functions only.   The regularity theory implies, for some bigger constant $C$ we have,
\begin{align}
  \frac{\twoeta{ \hat v - \hat V }{r_0u}{(du)^2}}{\hat V}
  +
  \frac{\twoeta{ \hat w - \hat W }{r_0u}{(du)^2}}{\hat W}
  \leq
  CDV
\end{align}
  Now we convert this back to a statement in terms of the hatless functions $v$, and $w$.  For instance, just using the definition of the quantities, calculate
 \begin{align}
   \frac{u}{V}|\partial_u (v-V)|
  &= \frac{u}{V}|\partial_u \left( \frac{ u + \ein t}{u} (\hat v - \hat V) \right)| \\
  &\leq \frac{u}{V} \ein t u^{-2} |\hat v - \hat V|
    + \frac{u}{V}|\partial_u (\hat v - \hat V)| \\
  &=
    \ein \frac{t}{\hu}\frac{1}{V} | v -  V|
    +
    \frac{u}{\hu}\frac{u}{\hat V}|\partial_u (\hat v - \hat V) |.
 \end{align}
 Now using our barriers for the first term, and using the bound for regularity on the second term, as well as $\frac{t}{\hu} \leq 1$ and $\frac{u}{\hu} < 1$,
 \begin{align}
   \frac{u}{V}|\partial_u (v-V)|
   &\leq  CDV.
 \end{align}
Performing similar calculations, we can make the bounds we need.
\end{proof}
\subsection{Corollaries of control}\label{corollaries_productish}
The following corollaries state some precise results which hold for a metric satisfying the conclusions of Lemma \ref{main_prish_estimates}.
The corollaries above are just a matter of checking various derivatives and bounds.  For Corollary \ref{prish_curvature_control} one can use the calculations of the curvatures for warped products in Appendix \ref{curvatures_our_coordinates}.

First we rephrase our results in terms of how close the metric is to a cylinder.
\begin{corollary}\label{prish_asymptotic_cylinder}
  There is $C>0$ depending on the initial data and the parameters of control such that the following holds.
  Suppose that $g(t)$ is controlled in the productish region at time $t = t_\#$.  For $u_\#$ such that $(u_\#, t_\#)$ is in the productish region, let
  $$
  g_{cyl} = dx^2 + g_{S^q} + \frac{W_{prish}(u_\#, t_\#)}{u_\#} g_{F}.
  $$
  Let $L$ be given such that $\epsilon = L\sqrt{V_{prish}(u_\#, t_\#)} < 1$.  There is a map $\Phi: [-L, L] \times S^q \times F \to M$ which is the identity on the second two factors such that $u(\Phi(0, \cdot, \cdot), t_\#) = u_\#$ and
  \begin{align}
    \left| g_{cyl} - \Phi^* \left(u_{\#}^{(-1)} g(t_\#) \right) \right|_{C^2([-L, L] \times S^q \times F)} \leq \epsilon C
  \end{align}
\end{corollary}
We also state a result in terms of the curvature of the metrics.
\begin{corollary}\label{prish_curvature_control}
  Suppose $g(t)$ is controlled in the productish region.
  Then there is a constant $C$ depending on the initial data and the parameters of control, such that for all points in the productish region the curvature of $g(t)$ satisfies
  \begin{align}
    \Rm
    &= u^{-1} \left( u g_{S^q} \KN u g_{S^q} \right) + w \Rm_{g_F} + \Rm_{warp} \\
    &= u \Rm_{g_{S^q}} + w \Rm_{g_F} + \Rm_{warp}
  \end{align}
  where $|\Rm_{warp}| \leq C u^{-1} v$.
\end{corollary}
One further basic statement about the curvature is the following.
\begin{corollary}\label{prish_curvature_control_time}
  Suppose $g(t)$ is controlled in the productish region.  There is a constant $C$ depending on the initial data, such that the following holds for all points in the productish region.
  \begin{itemize}
  \item 
    If $\ein_F = 0$, suppose additionally that $(F, g_F)$ is flat.  Then $|\Rm| \leq C \sigma_*^{-1}(t \nu(t))^{-1}$.
  \item 
    If $\ein_F \leq 0$, suppose additionally that $(F, g_F)$ is flat or $W_0(u) \geq c' u$ for some $c' > 0$.  Then $|\Rm| \leq Cu^{-1}$.
  \end{itemize}
\end{corollary}
\begin{proof}
In Corollary \ref{prish_curvature_control}, since $V$ is uniformly bounded in the productish region, we get $|\Rm_{warp}| \leq C u^{-1}$.  We also have $|u \Rm_{g_{S^q}}| \leq C u^{-1}$.  (In fact, it is exactly $C_q u^{-1}$ for some constant $C_q$ depending on $q$.)  Since in the productish region, we have $u \geq t \nu(t)\sigma_*$, this proves that the terms $|u \Rm_{g_{S^q}}|$ and $|\Rm_{warp}|$ from Corollary \ref{prish_curvature_control} satisfy both of the desired bounds.

We have $|w \Rm_{g_F}| = C_F w^{-1}$.  If $\ein_F < 0$, then $W_0(u + \ein t) - \ein_F t \geq (-\ein_F) t$ and so we get the first conclusion.  If $\ein_F = 0$ and $(F, g_F)$ is flat, then $C_F = 0$ so we get both conclusions.  If $\ein_F \leq 0$ and $W_0(u + \ein t) \geq c' u$, then $W_0(u + \ein t) - \ein_F t \geq c' u$ so we get the second conclusion.  If $\ein_F > 0$ then the assumption \ref{w_big} tells us $W_0(u + \ein t) - \ein_F t \geq (1 + c) \ein^{-1}\ein_F u$, so we get both conclusions.
\end{proof}




\section{Control in the tip region}\label{section:tip}
We are still considering a Ricci flow of model pinches.
Recall $\nu(t) = V_0(\ein t)$ and $\omega(t) = W_0(\ein t)$.  Section \ref{examining_section} shows that our approximate solutions in the productish region work up to where $\sigma = \frac{u}{t\nu(t)}$ stays very large.  In order to examine the solution where $\sigma$ is bounded, we will rescale the metric $g$ by $\alpha = t \nu(t)$: set $\tilde g = \alpha^{-1} g$.  Instead of scaling $w$ by $\alpha(t)$ as well, we will work with the function $\bar w = \omega^{-1} (w + \ein_F t) = \omega^{-1}\left(\alpha(t)\tilde w + \ein_F t\right) = \omega^{-1}\hat w$. We also introduce a rescaled time derivative $\partial_{\theta} = \alpha \partial_t$.

The tip region will be, for a constant $\zeta_*$ to be determined,
\begin{align}
  \Omega_{tip} = \left\{
  (u, t) : \frac{u}{t\nu(t)} < \frac{\zeta_*}{\nu^{1/2}}
  \right\}.
\end{align}
In this section, we find the approximate solutions for $v$ and $\bar w$ in the tip region
\footnote{We use the same notation $V$ and $V^\pm$ here for different functions than the barriers in Section \ref{section:productish}.  In the following section, where we need to refer to both the functions defined here and the functions from Section \ref{section:productish}, we will use e.g. $V_{tip}$ for the function defined here and $V_{prish}$ for the function defined there.}
:
\begin{align}
  V \defeq V_{\Bry}(\sigma) + \beta V_{\Pert}(\sigma), \\
  \bar W \defeq 1 + (\log \omega)_{\theta} W_{\Pert}(\sigma),
\end{align}
where $\beta = \alpha'$, $(\log \omega)_{\theta} = \partial_{\theta}(\log \omega)$, and $V_{\Bry}$, $V_{\Pert}$, and $W_{\Pert}$ are functions which are to be defined.  In Lemmas \ref{vsupsoln} and \ref{wsupsoln} we define functions $V^{\pm}$ and $W^{\pm}$, which satisfy $V^- < V < V^+$ and $\bar W^- < \bar W < \bar W^+$, and will serve as barriers for $v$ and $\bar w$.  These functions depend on constants $\epsilon_v$, $\epsilon_w$, and $\delta$.  The barriers $V^-$ and $V^+$ are carefully defined so that if $V^- < v < V^+$ then $\tilde L = \sigma^{-1}(1 - \on4 v)$ is bounded near $\sigma = 0$.  

Here we make definitions similar to Definition \ref{productish_barricaded}.  We use the notation $x^{a, b} = x^a(1 + x)^{b-a}$; which is approximately $x^a$ near $x = 0$ and $x^b$ near $x = \infty$.  
\begin{definition}\label{tip_barricaded}
  We say that $g(t)$ is \emph{barricaded (by the tip barriers)} at a point in space-time if it satisfies
  \begin{align}
    V^- < v < V^+, \quad W^- < w < W^+ \label{tip_barrier_ineq}
  \end{align}
  at that point.
  
  We say that $g(t)$ is \emph{initially controlled in the tip region}
  if at $t = T_1$
  for all points satisfying $\sigma \leq 2\zeta_* \nu^{-1/2}(T_1)$
  it is barricaded,
  and
  \begin{align}
    \twoeta{v-V}{1}{ (d\sigma^{1/2, 1})^2}
     &\leq c_{safe}C_{reg}\left( \delta^{-1}\epsilon_{v} \right)\nu^{1/2}\sigma^{1, -1},
      \label{v_reg}\\
    \twoeta{w-W}{1}{ (d\sigma^{1/2, 1})^2}
    &\leq c_{safe}C_{reg}\epsilon_{w}\nu^{1/2}.
  \end{align}
  
    We say that $g(t)$ is \emph{barricaded at the right of the tip region} if it is barricaded for all points satisfying $\zeta_* \nu^{-1/2} < \sigma < 2 \zeta_* \nu^{-1/2}$.

    We say that $g(t)$ is \emph{controlled in the tip region} if
  \begin{enumerate}[label=(T\arabic{*}), ref=(T\arabic{*})]
  \item\label{conc:tip_barrier}
    For all points in $\Omega_{tip}$, the solution is barricaded.
  \item\label{conc:tip_reg_noncompact}
    For all points in $\Omega_{tip}$ with $\sigma \geq 1$ and
    \begin{align}
      \twoeta{v-V}{1}{ (d\sigma^{1/2, 1})^2}
      &\leq C_{reg}\left( \delta^{-1}\epsilon_{v} \right)\nu^{1/2}\sigma^{1, -1},\\
      \twoeta{w-W}{1}{ (d\sigma^{1/2, 1})^2}
      &\leq C_{reg}\epsilon_{w}\nu^{1/2}.
    \end{align}
  \end{enumerate}  
\end{definition}
\begin{remark}
  $V$ satisfies $(1/C)\sigma^{0,-1} < V < C\sigma^{0,-1}$ and $V^+ - V^-$ satisfies $V^+ - V^- \leq C \sigma^{-1, 0}$.  This is the reason for the factor $\sigma^{1, -1}$ in \eqref{v_reg}.  To understand $(d \sigma^{1/2,1})^2$, know that for the Bryant soliton, $d(\sigma^{1/2, 1})$ is uniformly comparable to the arc length element- the radius $\sqrt{\sigma}$ grows like distance near $\sigma = 0$ and grows like the square root of distance as $\sigma \to \infty$.  So, this metric is comparable to the Bryant soliton metric.  Furthermore, $(\delta^{-1} \epsilon_v)$ controls the separation between the barriers for $v$, whereas $\epsilon_w$ controls the separation between the barriers for $w$- this explains is the reason for the appearance of those constants.  \end{remark}

The following is the main result of this section.

\begin{lemma}\label{main_tip_estimates}
  There is a $c_{safe}$ such that if we let $C_{reg} > \berr C_{reg}$, $\epsilon_v$, $\epsilon_w < \barr \epsilon_w(\epsilon_v)$, $\zeta_*$, and $\delta < \barr \delta(\zeta_*)$, there is a $T_*$ depending on all other constants with the following property.

  Suppose $0 < T_1 < T_2 < T_*$ and $g(t)$ is defined on $[T_1, T_2]$, initially controlled in the tip region, and barricaded at the right of the tip region.  Then $g(t)$ is controlled in the tip region.
\end{lemma}

\subsection{A summary of functions}\label{function_summary}
We will be introducing many functions of $\sigma$.  Here, we provide the reader with a little cheat sheet to recall the asymptotics of the functions.  This makes us feel better about possibly using the asymptotics without warning.

We use the notation $\sigma^{a,b} = \sigma^a (1 + \sigma)^{b-a}$ and $|F|_3 = F + \sigma \partial_\sigma F + \sigma^2 \partial_\sigma^2 F + \sigma^3 \partial_\sigma^3 F$.  As usual, $c < C$ are constants depending only on the given model pinch.  We use $F \sim G$ if $c G < F < CG$.

We have
\begin{align}
  V_{\Bry} \sim \sigma^{0, -1},
  &\quad V_{\Pert} \sim \sigma^{1, 0},
  \quad W_{\Pert} \sim \sigma^{0, 1}, \\
  |V_{\Bry}|_3 \leq C\sigma^{0, -1},
  &\quad |V_{\Pert}|_3 \leq C\sigma^{1, 0},
  \quad |W_{\Pert}|_3 \leq C\sigma^{0, 1}.
\end{align}
More precisely as $\sigma \to \infty$ we have
\begin{align}
  V_{\Bry} = \ein \sigma^{-1} + O(\sigma^{-2}),
  \quad V_{\Pert} = \oh + O(\sigma^{-1}),
  \quad W_{\Pert} = \oh \ein \sigma + O(\log \sigma).
\end{align}
Our approximate solutions are $V = V_{\Bry} + \beta V_{\Pert}$ and $W = 1 + (\log \omega)_{\theta} W_{\Pert}$.  Here are crude bounds on our barriers: for $\nu^{1/2}\sigma < \zeta_*$
\begin{align}
  \oh V_{\Bry} &< V^- < V < V^+ < 2 V_{\Bry} \\
  \oh &< \bar W^- < \bar W < \bar W^+ < 2.
\end{align}
More precise bounds are given in Lemma \ref{barrier_order_lemma}:
\begin{align}
  V^+ - V \sim 
  \left(\delta^{-1}\epsilon_v\right) \nu^{1/2}\sigma^{1,-1},
  \quad
  W^+ - W \sim   \epsilon_{w} \nu^{1/2}, 
\end{align}
and similarly for $V-V^-$ and $W-W^-$.

We also have the following facts about the functions of time $\nu(t) = V_0(\ein t)$, $\omega(t) = W_0(\ein t)$, and $\alpha(t) = t \nu(t)$.  We define $\dd{k}{f}(t) = t^k \partial_t^k f/f$, which is bounded for $f = \nu, \omega$ and $k = 1, 2$ by the model pinch assumption \ref{modelpinch_reg}.  The following are straightforward calculations.
\begin{align}
  \beta \defeq \partial_t \alpha = (1 + \dd1 \nu)\nu,
  &\quad
    \partial_{\theta} \beta = O(\nu^2), \label{beta_bounds}\\
  \partial_{\theta} \log \omega = \dd1 \omega \nu,
  &\quad
  \partial_{\theta}^2 \log \omega = O(\nu^2).\label{omega_bounds}
\end{align}
\subsection{Type-II rescaling}
Note that $v = u^{-1}|\nabla u|^2_{g} = \sigma^{-1}|\nabla \sigma|^2_{\tilde g}$. Also,
$\tilde \grad \sigma = \grad u$ and $|\nabla \sigma|_{\tilde g}^2 = \alpha |\nabla u|_{g}^2$ so $\partial_u = \alpha^{-1}
\partial_\sigma$. So,
\begin{align}
  \label{eq:85}
  \partial_{t;u}
  &= \partial_t - (\partial_t u)\partial_u \\
  &= \alpha^{-1} \partial_\theta - (\partial_t (\alpha \sigma))(\alpha^{-1}\partial_\sigma) \\
  &= \alpha^{-1} \left( \partial_{\theta;\sigma} - \beta \sigma \partial_\sigma \right). \label{partialt_rewrite}
\end{align}

We define $\qop_{\sigma}$ and $\lop_{\sigma}$ to be $\qop$ and $\lop$, from \eqref{V_evo}, with $\partial_u$ replaced with $\partial_\sigma$.
Then using our equation for $\partial_{t;u}v$, \eqref{V_evo}, we find
\begin{align}
  \label{v_evo_theta}
  \partial_{\theta;\sigma}v  &= \sigma^{-1}\qop_{\sigma}[v,v] + \sigma^{-1}\lop_{\sigma}[v] + \beta \sigma \partial_\sigma v - 2 \tilde \kappa^2v 
\end{align}
where $\tilde \kappa = \of \dim(F) \tilde w^{-1}y$.
Let
\begin{align}
  \label{eq:125}
  \fop_{\sigma}[v, \tilde \kappa] = \left( \sigma^{-1}\qop_\sigma[v,v] + \sigma^{-1}\lop_\sigma[v] - 2 \tilde \kappa^2 v \right).
\end{align}
So \eqref{v_evo_theta} is
$\partial_{\theta;\sigma}v - \fop_{\sigma}[v, \tilde \kappa] - \beta \sigma \partial_\sigma v= 0$.

\subsection{The Bryant Soliton}
The Bryant soliton $(\Bry, g_{\Bry}, f_{\Bry})$ is a steady gradient Ricci soliton on the topology $\Bry = \real^{q+1}$. The metric may be written as
\begin{align}
  \label{eq:58}
  g_{\Bry} = \frac{d\sigma^2}{\sigma V_{Bry}(\sigma)} + \sigma g_{S^q}.
\end{align}
As a steady soliton, under Ricci flow it moves only by diffeomorphisms, which fix the warped product structure.  The value of $v = \sigma^{-1}|\nabla \sigma|^2$ at a point where $\sigma = \sigma_*$ is a diffeomorphism-invariant property, so $\partial_{t;\sigma}v_{\Bry} = 0$.  
Therefore,
\begin{align}
  \label{eq:60}
  \qop[V_\Bry, V_\Bry] + \lop[V_\Bry] = 0.
\end{align}

The Bryant soliton has strictly positive sectional curvature, and its scalar curvature has a maximum at $\sigma = 0$.  The soliton is defined up to scaling and diffeomorphism, so let's say we have chosen the scaling with maximum scalar curvature $\ein$.  As $\sigma \to \infty$, $V_{\Bry}$ has the asymptotics
\begin{align}\label{vbry_infty_asymptotics}
  V_{\Bry}(\sigma) = (1 + O(\sigma^{-1})) \ein \sigma^{-1}
\end{align}
and as $\sigma \to 0$, $V_{\Bry}$ has the asymptotics
\begin{align}\label{vbry_zero_asymptotics}
  V_{\Bry}(\sigma) = 4 \left( 1 - \frac{\ein}{q(q-1)}\sigma + o(\sigma) \right). 
\end{align}

For any $k > 0 $ we may scale the metric by $k^{-1}$, resulting in the Bryant soliton with maximum scalar curvature $k \ein$.  The corresponding function $V_{k \Bry}$ is related by
$V_{k\Bry}(\sigma) = V_{\Bry}(k \sigma)$.

\subsection{Approximation for $v$}

Suppose that $v$ satisfies \eqref{v_evo_theta}, and also converges sufficiently smoothly to a limit $v_0$ as $\theta \searrow -\infty$.  Suppose also that $\tilde \kappa^2$ converges to zero as $\theta \searrow 0$.  Then we learn,
$
  \qop_{\sigma}[v_0, v_0] + \lop_\sigma[v_0] = 0.
$
That is, $v_0$ describes a steady soliton.
If the limit metric has $\sigma \in [0, \infty)$, and has nonzero curvature, then we learn that as a function of $u$, $v_0 = V_{k \Bry}(u)$ for some scaling factor $\mr$. Comparing the asymptotics \eqref{vbry_infty_asymptotics} with our approximate solution in the parabolic region \eqref{V_expr_smallrho}, we choose $k = 1$.

Now we address the term $\beta \sigma v_\sigma$.  This term suggests that our approximation $v \approx v_0$ for small $\theta$ is off by a term of order $\beta$.  Write $\tilde v(\sigma,\theta) = v_0(\sigma) + \beta v_1(\sigma)$, and plug into $\partial_{\theta;\sigma} v - \fop_\sigma[v, \kappa]-\beta \sigma \partial_\sigma v = 0$.   This gives us,
\begin{align}
  \label{eq:73}
  \partial_{\theta;\sigma} v - \fop_\sigma[v, \kappa]-\beta \sigma \partial_\sigma v 
  &= \beta_{\theta} v_1 \\
  &- \beta \left( 2 \sigma^{-1}\qop[v_0, v_1] + \sigma^{-1}\lop[v_1] +  \sigma \partial_\sigma v_0 \right) \\
  &+ \left( \dots \right).
\end{align}
Here the term $(\dots )$ is bounded by
\begin{align}
  \label{eq:104}
   |\dots|
  &\leq C\beta^2 \left( \sigma^{-1}|v_1|^2_2  + \beta^{-2}\tilde \kappa^2 |v_0| + \beta^{-1}\tilde \kappa^2 |v_1| \right).
\end{align}
We also have $\beta_{\theta} = O(\nu^2)$. (See the end of Section \ref{function_summary}.) So,
\begin{align}
  \label{fop_calcone}
  \partial_{\theta;\sigma} v - \fop_\sigma[v, \kappa]-\beta \sigma \partial_\sigma v 
  &=- \beta \left( 2 \sigma^{-1}\qop[v_0, v_1] + \sigma^{-1}\lop[v_1] +  \sigma \partial_\sigma v_0 \right) \\
  &+ \beta^2 E
\end{align}
where
$
  E \leq C \left( |v_1| + \sigma^{-1}|v_1|^2_3 + \beta^{-2}\tilde \kappa^2 |v_0| + \beta^{-1}\tilde \kappa^2 |v_1| \right)
$.
Concerning the equation approximately satisfied by $v_{1}$, we have the following lemma, which is Lemma 4 of \cite{ACK}.  Recall $\sigma^{a,b} = \sigma^{a}(1 + \sigma)^{b-a}$.
\begin{lemma}\label{pert_fun}
  There is a solution $V_{Pert}$ to
  \begin{align}
    \label{eq:94}
    2 \sigma^{-1}\qop_\sigma[V_{\Bry}, V_{Pert}] + \sigma^{-1}\lop[V_{Pert}] = -\sigma \partial_\sigma V_{\Bry}.
  \end{align}
  on $[0, \infty)$, which extends to a smooth even function on $(-\infty, \infty)$.

  The function $V_{\mr Pert}(\sigma) = \mr^{-1}V_{Pert}(\mr \sigma)$ is a solution to
  \begin{align}
    \label{eqn_V_pert}
    2 \sigma^{-1}\qop_\sigma[V_{\mr\Bry}, V_{\mr\Pert}] + \sigma^{-1}\lop[V_{\mr\Pert}] = -\sigma \partial_\sigma V_{\mr \Bry}.
  \end{align}
  As $\sigma \to \infty$, $V_{\mr \Pert}$ has the asymptotics
  \begin{align}
    \label{Vpert_asymptotics}
    V_{\mr \Pert} = (1 + O(\sigma^{-1})) k^{-1}.
  \end{align}
  There is a $C > 0$ depending on the dimension such that
  \begin{align}
    \label{vpert_bnd}
    |V_{\Pert}|_2 < C \sigma^{1,0}
  \end{align}
\end{lemma}

This invites the choice of approximate solution
\begin{align}
  V = V_{Bry} + \beta V_{Pert}.
\end{align}

\subsection{Approximation for $w$}\label{rescale_w}
The expression for our approximation in the productish region \eqref{Wprish_asymptotic_exp} suggests that, in the tip region, $\bar w = \omega^{-1}(w + \ein_F t)$ is approximately constantly $1$.  This is a complete solution to the equation satisfied by $\bar w$ (the term $\ein_F t$ takes care of the reaction part of the equation).  We derive an equation for $\bar w$, to find the next order term.  We can come from the evolution of $\hat w = w + \ein_F t$ in terms of $u$, \eqref{evo_w_in_u}:
\begin{align}
  \partial_{t;u} \hat w
    &= u v \partial_u^2 \hat w - y + \ein \partial_u \hat w - \ein/2 v\partial_u \hat w.
\end{align}
Multiplying by $\omega^{-1}\alpha$, we have
\begin{align}
  \alpha \partial_{t;u} \bar w
  &= \sigma^{-1}\rop[\bar w, v] - (\alpha \omega^{-1})y - (\log \omega)_{\theta} \bar w,
\end{align}
where
\begin{align}
  \rop[z,v]
  = \sigma^2 v \partial_\sigma^2 z + (\mu - (c_v-\oh q)v)\sigma \partial_\sigma z.
\end{align}
Then using \eqref{partialt_rewrite},
\begin{align}
  \partial_{\theta; \sigma} \bar w
  &= \sigma^{-1}\rop[\bar w, v] - (\alpha \omega^{-1})y - (\log \omega)_{\theta} \bar w - \beta \sigma \partial_\sigma \bar w \\
  &= \sigma^{-1}\rop[\bar w, v]
    - \frac{1}{\bar w - \ein_F \frac{t}{\omega}}v \sigma (\partial_\sigma \bar w)^2 
  - (\log \omega)_{\theta} \bar w - \beta \sigma \partial_\sigma \bar w. \label{evo_w_in_sigma}
\end{align}

Concerning the operator $\rop$, we have the following Lemma.
\begin{lemma}\label{wpert_exist}
  There is a solution $W_{\Pert}(\sigma)$ to
  \begin{align}
    \sigma^{-1}\rop[W_{\Pert}, V_{\Bry}] = 1
  \end{align}
  which extends to a smooth even function on $(-\infty, \infty)$.

  The function $W_{k\Pert}(\sigma) = W_{\Pert}(k \sigma)$ is a solution to
  \begin{align}
    \sigma^{-1}\rop[W_{k\Pert}, V_{k\Bry}] = 1
  \end{align}
  As $\sigma \to \infty$, $W_{k\Pert}$ has the asymptotics
  \begin{align}\label{Wpert_asymptotics}
    W_{kPert} = (1 + o(1))\oh \ein k\sigma
  \end{align}
\end{lemma}
\begin{proof}
  The main idea is that $W_{\Pert}$ is just a scaling of the gradient potential function $f$.  On any steady soliton, the gradient potential function satisfies
$\lap_X f = 1 $
where $\lap_X$ is the drift laplacian. The operator $\sigma^{-1}\rop$ is a recasting of the laplacian in these coordinates.  For the derivation see page \pageref{wpert_exist_continuation}.
\end{proof}

This suggests the approximate solution $\bar W = 1 + (\log \omega)_{\theta}W_{Pert}$.  To find this, plug in $\bar w(\sigma, t) = 1 + \bar w_1(\sigma, t)$ as an initial approximation and assume that $\bar w_1$ goes to zero as $t \searrow 0$. Then taking the highest order terms in the limit $t \searrow 0$ we are left with the equation
\begin{align}
  - \left( \sigma^{-1} \rop[\bar w_1, V_{Bry}] - (\log \omega)_{\theta}\cdot 1 \right)
\end{align}
for $\bar w_1$.  Note $\log (\omega)_{\theta} = O(\nu)$ (see the end of Section \ref{function_summary}).

\subsection{$y$ control}
One tricky term which appears in the evolution of $v$, \eqref{v_evo_theta}, is $\tilde \kappa^2 = \on4 \dim(F) w^{-1} y$.  This cannot be controlled with simple barrier arguments: at a point where $w$ is trapped between barriers for $w$, and $v$ touches barriers for $v$, we only know that the derivative $v$ matches the derivative of the barrier for $v$, but we do not get a free bound on the derivative of $w$.  For this reason we need to use regularity.  

\begin{lemma}\label{lemma:y_control_tip}
  Suppose we are in the setting of Lemma \ref{main_tip_estimates}, with \ref{conc:tip_barrier} and \ref{conc:tip_reg_noncompact} for $t \in [T_1, T_2)$.  Then, if $T_*$ is sufficiently small,
  \begin{align}
    \tilde \kappa^2 \leq C C_{reg}^2 \epsilon_{w}^2 \sigma^{1,0} \nu
  \end{align}
\end{lemma}
\begin{proof}
We rewrite $\tilde \kappa^2$ as,
\begin{align}
  \tilde \kappa^2
  &= C \tilde w^{-1} y
  = C v \frac{|\nabla \tilde w|^2_{\tilde g}}{\tilde w^2}\frac{1}{|\nabla \tilde \phi|_{\tilde g}^2} = C v\frac{1}{\tilde w^2} \left( \partial_{\tilde \phi} \tilde w \right)^2
\label{kappa_rewrite}
\end{align}
Here, we used that $y$ and $v$ are scale-invariant, and that $v = \on4 |\nabla \tilde \phi|^2_{\tilde g}$.  
Using the assumption \ref{w_big} on $W_0$, $\omega(t) > (1 + c)\ein t$ (whether $\ein$ is positive or not), so we have
\begin{align}
  \tilde \kappa^2
  &= C v\left( \frac{1}{\bar w - \ein_F \frac{t}{\omega}} \right)^2 \left( \partial_{\tilde \phi} \bar w \right)^2
  \leq C v\left( \frac{1}{\bar w - \frac{1}{1+c}} \right)^2 \left( \partial_{\tilde \phi} \bar w \right)^2.
\end{align}
In the region under consideration, we can take $T_*$ small enough so that
$\bar W^- > 1 - \frac{1}{2}\frac{c}{1 + c}$.
Therefore since $\bar w > \bar W^-$, increasing $C$,
\begin{align}
  \tilde \kappa^2
  &\leq C v \left( \partial_{\tilde \phi} \bar w \right)^2.\label{basic_tildekappa_ineq}
\end{align}

To control $\kappa^2$ in $\{\sigma < 1\}$,  we need to use that $\partial_{\tilde \phi} \bar w = 0$ at $\tilde \phi = 0$.  Copying \ref{conc:tip_reg_noncompact} for $k = 2$, and writing it in terms of $\tilde \phi$, 
\begin{align}
  \partial_{\tilde \phi}^2 \bar w < \partial_{\tilde \phi}^2 \bar W + 2C_{reg} \epsilon_{w} \nu^{1/2}.
\end{align}
Since $\partial_{\tilde \phi}^2 \bar W \leq C \nu$, for sufficiently small times the first term dominates, and $\partial_{\tilde \phi}^2 \bar w \leq 4 C_{reg} \epsilon_w \nu^{1/2}$.  We can integrate this from $\phi = 0$ to find 
$
  \partial_{\tilde \phi} \bar w < 4C_{reg} \epsilon_{w} \nu^{1/2}\tilde \phi.
$
Then \eqref{basic_tildekappa_ineq} proves the claim for $\sigma < 1$.

To control $\kappa^2$ in $\{\sigma > 1\}$, we use \ref{conc:tip_reg_noncompact}:
\begin{align}
  \partial_{\sigma} \bar w \leq \partial_\sigma \bar W + C_{reg} \epsilon_{w}\nu^{1/2}.
\end{align}
We have $\partial_\sigma \bar W \leq C \nu$ since $\partial_\sigma W_{pert}$ is bounded.  So for small times $\partial_\sigma \bar w \leq 2C_{reg} \epsilon_w \nu^{1/2}$.  In terms of $\phi$ this says $\partial_\phi \bar w \leq 4C_{reg} \epsilon_w \phi \nu^{1/2}$.  Since $v \leq V^+ \leq C \sigma^{0, -1}$, \eqref{basic_tildekappa_ineq} gives us $\tilde \kappa^2 \leq C C_{reg}^2 \epsilon_w^2 \nu$.
\end{proof}

\subsection{Barriers}\label{tip_barriers_section}
In this section we define the barriers $V^{\pm}$ and $\bar W^{\pm}$ and prove that item \ref{conc:tip_barrier} continues to hold.  The barriers are defined as follows.  Let $k(t)^{\pm} = 1 \mp \delta^{-1}\epsilon_{v} \nu^{1/2}$ and then set
\begin{align}
  V^{\pm} &= V_{k^{\pm}(t) \Bry}
  + \left(\beta \mp \epsilon_v \nu\right)  V_{k^{\pm}(t)\Pert} \label{v_barriers_def}\\
  \bar W^{\pm} &= 1 \pm \epsilon_{w}\nu^{1/2}
  + \left( (\log \omega)_{\theta} \mp \delta \epsilon_w \nu  \right)W_{\Pert} .\label{w_barriers_def}
\end{align}
We will prove that these are sub- and supersolutions to the equations satisfied by $\bar w$ and $v$.  The power $\nu^{1/2}$ is a bit mysterious here, but it is the best possible for barriers of this form.  We discuss its derivation after Lemma \ref{buckling_lemma}.  It is helpful to remember that $\beta \sim \nu$, and $(\log \omega)_{\theta} \lesssim \nu$ (see the end of Section \ref{function_summary}).

The terms $- \epsilon_v \nu V_{k^{\pm}\Pert}$ and $- \nu \delta \epsilon_w W_{\Pert}$ are the terms which will give us that $V^+$ and $W^+$ are strict supersolutions to their equations.  They are chosen by taking the approximate solution, which is found by starting from a limit at $t = 0$ and adding a perturbation which solves an elliptic equation, and then fiddling with the size of the perturbation.

Because the extra amount of the perturbation needed for a supersolution comes with a negative sign in both cases, we need to add something else to ensure that the supersolution lies above the intended approximate solution.  This is the role of $k^{\pm}(t)$ and of $\pm\epsilon_w \nu^{1/2}$.  (If it's not clear what's going on with $k^+$, recall that $V_{\Bry}$ is decreasing so $V_{k^+\Bry}(\sigma) = V_{\Bry}(k^+\sigma) > V_{\Bry}(\sigma)$.)  The role of $\delta$ in both equations is to control the ratio of the extra positive term used to make the supersolution bigger than the approximate solution, to the extra negative term used to make the supersolution a supersolution to the equation.

Lemma \ref{barrier_order_lemma} clarifies the role of $\delta$.  Recall $\sigma^{a,b} = \sigma^a (1 + \sigma)^{b-a}$.  The significance of the factor $\sigma^{1, -1}$ in the inequalities for $V$ in this lemma is the following. At infinity, $V \sim \sigma$ so this is a normalization.  At $0$, $V^+ - V^- \sim \sigma$ is necessary to ensure smoothness of a solution with $V$ trapped between $V^-$ and $V^+$.  On the other hand $\bar W \sim 1$ everywhere so $\bar W$ requires no normalization.
\begin{lemma}\label{barrier_order_lemma}
  There are constants $c < C$ depending only on the dimension such that the following holds.
  
  We have, for $V_{diff} = V^+ - V$ or $V_{diff} = V - V^-$,
  \begin{align}
    c \delta^{-1}\epsilon_{v} \nu^{1/2} \sigma^{1, -1} \left(1 - C\delta \nu^{1/2}\sigma^{0,1} \right)
    \leq V_{diff}
    \leq C \delta^{-1}\epsilon_{v} \nu^{1/2} \sigma^{1,-1}.
  \end{align}
  Similarly, for $W_{diff} = \bar W^+ - \bar W$ or $W_{diff} = W - W^-$,
  \begin{align}
    c \epsilon_{w} \nu^{1/2}\left( 1 - C\delta \nu^{1/2}\sigma^{0, 1} \right)
    \leq W_{diff}
    \leq C \epsilon_{w} \nu^{1/2}.
  \end{align}
  In particular, if we choose $\delta < \frac{1}{2C}\zeta_*^{-1}$ then, renaming $c$, for all $\sigma < \zeta_* \nu^{-1/2}$
  \begin{align}
    c \delta^{-1}\epsilon_{v} \nu^{1/2} \sigma^{1, -1} 
    &\leq V_{diff}
    \leq C \delta^{-1}\epsilon_{v} \nu^{1/2} \sigma^{1,-1},\\
    c \epsilon_{w} \nu^{1/2}
    &\leq \bar W_{diff}
    \leq C \epsilon_{w} \nu^{1/2}.
  \end{align}
\end{lemma}
\begin{proof}
  The asymptotics of the $V_{Bry}$ are given in \eqref{vbry_infty_asymptotics} and \eqref{vbry_zero_asymptotics}.  Also recall that
  $V_{kBry}(\sigma) = V_{Bry}(k\sigma)$.
  Using these asymptotics, for small enough $\sigma$,
  \begin{align}
    c \delta^{-1}\epsilon_v \nu^{1/2} \sigma
    < V_{k^+Bry}(\sigma) - V_{Bry}(\sigma)
    < C \delta^{-1}\epsilon_v \nu^{1/2}\sigma,
  \end{align}
  and for large enough $\sigma$,
  \begin{align}
    c \delta^{-1}\epsilon_v \nu^{1/2}\sigma^{-1}
    < V_{k^+Bry}(\sigma) - V_{Bry}(\sigma)
    < C \delta^{-1}\epsilon_v \nu^{1/2}\sigma^{-1}.
  \end{align}
  Furthermore, since $V_{Bry}$ is strictly decreasing in any compact set away from the origin, for any $\sigma_1 < \sigma_2$ there are constants $c_{\sigma_1, \sigma_2}$ and $C_{\sigma_1, \sigma_2}$ such that
  \begin{align}
    c_{\sigma_1, \sigma_2} \delta^{-1}\epsilon_v \nu^{1/2}
    < V_{k^+Bry}(\sigma) - V_{Bry}(\sigma)
    < C_{\sigma_1, \sigma_2} \delta^{-1}\epsilon_v \nu^{1/2}.
  \end{align}
  These three sets of bounds prove that for some $c < C$,
  \begin{align}
    c \delta^{-1}\epsilon_v \nu^{1/2} \sigma^{1, -1}
    < V_{k^+Bry}(\sigma) - V_{Bry}(\sigma)
    < C \delta^{-1}\epsilon_v \nu^{1/2} \sigma^{1,-1}.
  \end{align}

  Putting this together with the bound on $V_{Pert}$, \eqref{vpert_bnd}, which says
  \begin{align}
    - \epsilon_v \nu V_{pert}(\sigma) > -C \epsilon_v\nu \sigma^{1,0},
  \end{align}
  and using $\beta \leq C \nu$ (from \eqref{beta_bounds}), we have the claim for $V$.

  The proof for $\bar W$ is similar but more straightforward.  One needs to use the properties from Lemma \ref{wpert_exist}.
\end{proof}

We now prove that $V^{\pm}$ and $\bar W^{\pm}$ are sub- and supersolutions to the equations satisfied by $v$ and $\bar w$.  In Lemma \ref{lemma:barrier_wrapup} we will summarize by saying that item \ref{conc:tip_barrier} continues to hold.

\begin{lemma}\label{vsupsoln}
  Let
    $\zeta_* > 0$, $\epsilon_v > 0$, and $\delta > 0$
  be given.  Let $V^{\pm}$ be the functions defined in \eqref{v_barriers_def}.
  
    Suppose $\tilde \kappa = \frac{y}{\tilde w} = \alpha \frac{y}{w}$ satisfies $\tilde \kappa^2(\sigma, t) \leq c_{ytip}\epsilon_v \nu \sigma^{1,0}$ where $c_{ytip}$ is a constant (chosen in the proof) depending only on dimensions.
  
   Then there is a $T_*$ depending on all parameters so that for $t < T_*$ and $\sigma < \zeta_* \nu^{-\oh}$ we have, for a constant $c$ depending only on the dimensions,
  \begin{align}
    \partial_{\theta; \sigma}V^+ - \fop_{\sigma}[ V^+, \tilde \kappa] - \beta \sigma \partial_\sigma V^+ \geq c \epsilon_v \nu \sigma^{1,-1} \label{ineq_vplus}
  \end{align}
  and
  \begin{align}
    \label{eq:119}
    \partial_{\theta; \sigma}V^- - \fop_{\sigma}[ V^-, \tilde \kappa] - \beta \sigma \partial_\sigma V^- \leq - c \epsilon_v \nu \sigma^{1,-1}.
  \end{align}
\end{lemma}

\begin{proof}
  Let us first demonstrate the main calculation, implicitly defining error terms $E_1$ and $E_2$.  Calculate
  \begin{align}
    - \fop_{\sigma}[V^+, 0]
    &= - \left( \sigma^{-1} \qop[V_{k^+\Bry}, V_{k^+\Bry} ] + \sigma^{-1}\lop[V_{k^+\Bry}] \right)\\
    &- (\beta - \epsilon_v \nu)
      \left(
      2\sigma^{-1}\qop[V_{k^+\Bry}, V_{k^+\Pert}] + \sigma^{-1} \lop[V_{k^+\Bry}]
      \right) \\
    &- (\beta - \epsilon_v \nu)^2 \sigma^{-1}\qop[V_{k^+\Pert}, V_{k^+\Pert}].
  \end{align}
  The first line vanishes, and the second line can be computed from the equation solved by $V_{k^+\Pert}$.  The last line is error.
  \begin{align}
    -\fop_{\sigma}[V^+,0]
      = +(\beta - \epsilon_v\nu) \sigma \partial_\sigma V_{k_+\Bry} + E_1
  \end{align}
  Also calculate,
  \begin{align}
    - \beta \sigma \partial_\sigma V^+
    &= - \beta \sigma \partial_\sigma V_{k^+\Bry}
      - (\beta -\epsilon_v\nu)\beta \sigma \partial_\sigma V_{k^+\Pert} \\
    &= - \beta \sigma \partial_\sigma V_{k^+\Bry} + E_2
  \end{align}
  Putting these together,
  \begin{align}
    - \fop_{\sigma}[V^+, 0] - \beta \sigma \partial_\sigma V^+
    &=
      -\epsilon_v \beta \sigma \partial_\sigma V_{k^+\Bry} + C (\beta \sigma^{1,-1}) \left( \beta \sigma^{0,1}\right) \\
    &\geq c \epsilon_v \nu \sigma^{1,-1} + E_1 + E_2
  \end{align}
  where we used that $\sigma \partial_\sigma V_{k^+ \Bry} \leq -c\sigma^{1,-1}$ for some $c$.  Therefore it remains to bound $E_1$ and $E_2$, as well as the other terms in \eqref{ineq_vplus}, namely
  \begin{align}
    \partial_{\theta; \sigma}V^+ \quad\text{and}\quad \fop_\sigma[V^+, \tilde \kappa] - \fop_\sigma[V^+,0] = \tilde \kappa^2 V^+.
  \end{align}

  For the following, note we can assume that $k(t)$ is in $[1/2,2]$. Using $\beta \sim \nu$, and using the notation $|f|_2 = |f| + \sigma |\partial_\sigma f| + \sigma^2 |\partial_\sigma^2 f|$, we have the bound on $E_1$ and $E_2$,
  \begin{align}
    |E_1| + |E_2|
    &= \left|(\beta-\epsilon\nu)^2 \sigma^{-1}\qop [V_{k^+\Pert}, V_{k^+\Pert}]\right|
      +
      \left|
      (\beta - \epsilon \nu)
      \beta \sigma \partial_\sigma  V_{k^+\Pert}
      \right|
    \\
    &\leq C\nu^2 \left( \sigma^{-1}|V_{\Pert}|_2^2 + |V_{\Pert}|_2 \right)\\
    &\leq C\nu^2 \left( \sigma^{-1}\left( \sigma^{1,0} \right)^2 + \sigma^{1,0} \right) \leq C \nu^2  \sigma^{1,0} = C \left( \nu \sigma^{1,-1} \right) \left( \nu \sigma^{0,1} \right)
  \end{align}

Now we bound the time term, using \eqref{beta_bounds}.
  \begin{align}
    \partial_{\theta;\sigma}V^+
    &= \partial_{\theta; \sigma}
      \left(
      V_{\Bry}(k(t)\sigma)
      +
      (1-\epsilon)\frac{\beta}{k(t)}V_{\Pert}(k(t)\sigma)
      \right) \\
    &\leq C
      \left(
      \sigma \partial_\sigma V_{\Bry} \partial_\theta k
      + (\partial_\theta \beta + \beta \partial_\theta k)V_{\Pert}
      + \beta \sigma \partial_\sigma V_{\Pert} \partial_\theta k
      \right) \\
    &\leq C
      \left(
      \sigma^{1,-1}\nu^{1+1/2}
      + (\nu^2 + \beta \nu^{1+1/2})\sigma^{1,0}
      + \beta \sigma^{1,0}\nu^{1+1/2}
      \right) \\
    &\leq C \nu \sigma^{1,-1}
      \left(
      \nu^{1/2} + \nu \sigma^{0,1}
      \right)
  \end{align}
  Finally, we use our assumption on $\tilde \kappa$ to bound the term $\tilde\kappa^2 V^+$ by
  \begin{align}
    \tilde \kappa^2 V^+
    &\leq C
      \left(
      c_{ytip}\epsilon_v \nu \sigma^{1,0}
      \right)
      \sigma^{0,-1}
      \\
    &\leq
      C\nu \sigma^{1,-1}
      \left( c_{ytip}\epsilon_v \right)
  \end{align}
  
  All in all, we find
  \begin{align}
    \partial_{\theta;\sigma} - \fop[V^+, \tilde \kappa ] - \beta \sigma \partial_\sigma V^+
    &\geq \nu \sigma^{1,-1}(c\epsilon_v - Cc_{ytip}\epsilon_v - o(1)) 
  \end{align}
  Here the term $o(1)$ goes to zero as $t \searrow 0$, in any region where $\sigma < \zeta_* \nu^{-1/2}$.  The lemma follows by choosing the $c$ in the statement to be one half of the $c$ above, choosing $c_{ytip}$ to be sufficiently small, and choosing $T_*$ to be small enough so that the $o(1)$ term is sufficiently small.
\end{proof}

For arbitrary functions $\bar w$ and $v$ we define
\begin{align}
  \dop(\bar w, v)
  \defeq
   \partial_{\theta; \sigma} \bar w
    -\left(
    \sigma^{-1}\rop[\bar w, v] - \frac{1}{\bar w - \ein_F \frac{t}{\omega}}v \sigma (\partial_\sigma \bar w)^2 - \beta \sigma \partial_\sigma \bar w - (\log \omega)_{\theta} \bar w
    \right).
\end{align}
The equation solved by $\bar w$ \eqref{evo_w_in_sigma} is therefore $\dop(\bar w, v) = 0$.

\begin{lemma}\label{wsupsoln}
  Let $\zeta_*>0$, $\epsilon_{v}>0$, $\epsilon_w>0$, and $\delta>0$ be given.  Let $V^{\pm}$ and $W^{\pm}$ be the barriers defined in \eqref{v_barriers_def} and \eqref{w_barriers_def}.
  
 There is a $T_*$ depending on all parameters such that for all $t < T_*$ and $\sigma < \zeta_* \nu^{-1/2}$ we have
  \begin{align}
    \dop(\barr W^+,v)
    > \oh \delta \epsilon_w \nu
  \end{align}
  and 
  \begin{align}
    \dop(\barr W^-, v)
    < - \oh \delta \epsilon_w \nu
  \end{align}
\end{lemma}

\begin{proof}
  The main idea is that 
  \begin{align}
    (\log \omega)_{\theta} \barr W^+ - \sigma^{-1}\rop[\barr W^+, v]
    &= (1 + \epsilon_w\nu^{1/2})(\log \omega)_{\theta} \\
    &+ (\log \omega)_{\theta}\left( (\log \omega)_{\theta} - \delta \epsilon_w \nu \right) W_{Pert}\\
    &- (\log \omega)_{\theta}\sigma^{-1}\rop( W_{pert}, V_{Bry}) +  \delta \epsilon_w \nu \sigma^{-1}\rop(W_{Pert}, V_{Bry}) \\
    &+ ((\log \omega)_{\theta} - \delta \epsilon_w \nu)
      \cdot (\sigma^{-1}\rop(W_{pert}, V_{Bry}) - \sigma^{-1}\rop(W_{Pert}, v))
  \end{align}
  We can simplify the first and third lines to find
  \begin{align}
    (\log \omega)_{\theta} \barr W^+ - \sigma^{-1}\rop[\barr W^+, v]
    &\geq \epsilon_w (\log \omega)_{\theta} \nu^{1/2} + \delta \epsilon_w \nu \\
    &+ (\log \omega)_{\theta}\left( (\log \omega)_{\theta} - \epsilon_w \nu \right) W_{Pert}\\
    &+ ((\log \omega)_{\theta} - \epsilon_w \nu)\cdot (\sigma^{-1}\rop(W_{pert}, V_{Bry}) - \sigma^{-1}\rop(W_{Pert}, v))
  \end{align}
  The first line has the correct sign, we will use it to bound the other lines and the rest of the terms.  First, let's bound the other lines above:
  \begin{align}
    (\log \omega)_{\theta} \barr W^+ - \sigma^{-1}\rop[\barr W^+, v]
    &\geq \epsilon_w \nu \\
    &- C\nu^2 \sigma^{0, 1} \\
    &- C \nu (\delta^{-1}\epsilon_v)\nu^{1/2} \sigma^{1,0}
  \end{align}
  Here we used the bound $|V_{\Bry} - v| < c \delta^{-1}\epsilon_v\nu^{p}\sigma^{1,-1}$ together with $|\sigma \partial_\sigma W_{Pert}| + |\sigma^2 \partial_\sigma^2W_{pert}| \leq  \sigma^{0,1}$.   In the second inequality we also used $(\log \omega)_{\theta} = \nu \dd1{\omega} \leq C \nu$.

  Next we find the term $\partial_{\theta;\sigma}\bar W^+$.  The term $\partial_{\theta; \sigma}\epsilon_w\nu^{1/2}$ has the correct sign, so we ignore it.  For the other time derivatives, we can use $|\partial_{\theta}^2 (\log \omega)| + |\partial_{\theta} \nu| \leq C \nu^2$ (from \eqref{beta_bounds}, \eqref{omega_bounds}):
  \begin{align}
    \left| \partial_{\theta}\left( (\log \omega)_{\theta} - \epsilon_w\nu  \right) W_{Pert} \right|
    \leq C \nu^2 W_{pert} 
    \leq C \nu^{2} \sigma^{0, 1}.
  \end{align}
  To bound the remaining terms, note
  $
    |\sigma \partial_\sigma \bar W^+|
    \leq \nu \sigma^{1,1}
  $
  and $v \leq C\sigma^{0, -1}$.  Also, as in the proof of Lemma \ref{lemma:y_control_tip} we can bound
  $
    \frac{1}{\bar W^\pm - \ein_F t} \leq C
    $.
  So
  \begin{align}
    \left|
    \frac{1}{\bar W^+ - \ein_F \frac{t}{\omega}}v \sigma (\partial_\sigma \bar W^+)^2 + \beta \sigma \partial_\sigma \bar W^+
    \right|
 &\leq C \left(
   \sigma^{-1}v |\sigma \partial_\sigma \bar W^+|^2 +
   \nu |\sigma \partial_\sigma \bar W^+|
   \right) \\
 &\leq C \left( \nu^2 \sigma^{1,0} + \nu^2 \sigma^{1,1} \right)
   \leq C \nu \left( \nu \sigma^{1,1} \right)
  \end{align}

  Putting together all of the inequalities, we have
  \begin{align}
    \dop(\bar W^+, v)
    &\geq \nu 
      \left(
      \delta \epsilon_w
      - C \nu \sigma^{1,1}-C (\delta^{-1}\epsilon_v)\nu^{1/2} \sigma^{1,0}
      \right).
  \end{align}
  In the space-time region under consideration,
  \begin{align}
    \dop(\bar W^+, v)
    &\geq \nu 
      \left(
      \delta \epsilon_w
      - C \nu^{1/2}\zeta_*-C (\delta^{-1}\epsilon_v)\nu^{1/2}
      \right)
      = \nu
      \left( \delta \epsilon_w - C(\zeta_* + \delta^{-1}\epsilon_v) \nu^{1/2} \right).
  \end{align}
  For small enough $T_*$, the positive term dominates.
\end{proof}

\begin{lemma}\label{lemma:barrier_wrapup}
  Suppose we are in the setting of Lemma \ref{main_tip_estimates}.  Suppose $\epsilon_{w} \leq \barr \epsilon_{w}(\epsilon_{v}, C_{reg})$.  There is a $T_*$ depending on all parameters such that the following holds.
  
  If items \ref{conc:tip_barrier} and \ref{conc:tip_reg_noncompact} hold for $t \in [T_1, T_2)$,  then item \ref{conc:tip_barrier} holds for $t \in [T_1, T_2]$.
\end{lemma}
\begin{proof}
  Choose $\barr \epsilon_{w}$ small enough (i.e. $\lesssim \sqrt{\epsilon_v}$) so that Lemma \ref{lemma:y_control_tip} implies that we have the desired inequality $\tilde \kappa^2 \leq c_{ytip} \epsilon_v \beta \sigma^{1, 0}$
  needed to apply Lemma \ref{vsupsoln}.

  Now, suppose that $v$ or $w$ touches one of its barriers at time $t = T_2$.  By Lemma \ref{vsupsoln} or \ref{wsupsoln}, we get a contradiction to the maximum principle since these lemmas say that $V^\pm$ and $W^{\pm}$ are strict sub- and supersolutions to the corresponding equations.
\end{proof}

\subsection{Regularity}
We prove the regularity \ref{conc:tip_reg_noncompact} separately for $\sigma \geq 1$ and $\sigma \leq 1$.  

\begin{lemma}\label{lemma:tip_reg_noncompact}
  Suppose we are in the setting of Lemma \ref{main_tip_estimates}. Suppose $\delta < \barr \delta(\zeta_*)$ so that the conclusion of Lemma \ref{barrier_order_lemma} holds.  We can choose $c_{safe}$ and $\berr C_{reg}$ depending only on the dimensions such that the following holds.  Suppose item \ref{conc:tip_barrier} holds for $t \in [T_1, T_2)$.   Then item \ref{conc:tip_reg_noncompact} holds for $t \in [T_1, T_2]$ and for $\sigma \geq 1.$
\end{lemma}
\begin{proof}
We copy equation \eqref{v_evo_theta}, using the expression \eqref{kappa_rewrite} for $\tilde \kappa^2$:
\begin{align}
  \partial_{\theta;\sigma}v
  &= \sigma v \partial_\sigma^2 v + c_1 \sigma^{-1}v + c_2 \partial_\sigma v + c_3 \sigma^{-1}v^2 + c_4 \sigma (\partial_\sigma v)^2 \\
  &+ \beta \sigma \partial_\sigma v
    + c_5 \sigma v
    \left( \frac{1}{\bar w - \ein_F t/\omega } \right)^2
    \left(\partial_\sigma \bar w \right)^2 v.
\end{align}
For $\sigma_1$ arbitrary, we multiply this by $\sigma_1$ to find,
\begin{align}
  \partial_{\theta;\sigma}\left( \sigma_1 v\right)
  &= \left[\sigma v \right]
     \partial_\sigma^2 \left(\sigma_1 v\right)
    + c_1 \left[\frac{\sigma_1}{\sigma} \right] \sigma_1^{-1} \left(\sigma_1v\right)
    + c_2 \partial_\sigma \left(\sigma_1 v \right)\\
  &+ c_3 \left[ \frac{\sigma_1}{\sigma} \right] \sigma_1^{-2} \left(\sigma_1v\right)^2
    + c_4 \left[ \frac{\sigma}{\sigma_1} \right] \sigma_1^{-1}(\partial_\sigma \left(\sigma_1 v \right))^2 \\
   &+ \left[ \beta \sigma \right]\partial_\sigma \left( \sigma_1 v \right)
     +c_5 
     \left[
     \frac{\sigma}{\sigma_1} (\sigma_1 v)^2
     \left( \frac{1}{\bar w - \ein_F t/\omega } \right)^2
     \right]
     \left(\partial_\sigma \bar w \right)^2 .
\end{align}
We also have the equation, from \eqref{evo_w_in_sigma},
\begin{align}
  \partial_{\theta; \sigma}\bar w
  &= \left[ \sigma v \right] \partial_\sigma^2 w
    + \left[ c_6 - c_8v \right]\partial_\sigma \bar w\\
  &- \left[ \frac{1}{\bar w - \ein_F \frac{t}{\omega}}v \right]\sigma (\partial_\sigma \bar w)^2 \\
  &- (\log \omega)_{\theta} \bar w - \left[ \beta \sigma  \right]\partial_\sigma \bar w.
\end{align}

For $\sigma_1$ and $t_1$ arbitrary but satisfying
\begin{align}
  1 < \sigma_1 < \zeta_* \nu^{-1/2}
\end{align}
we will apply parabolic regularity to $\sigma_1 v$ and $w$ in the region
\begin{align}
\Xi = (\sigma, \theta) \in [\sigma_1 - 1/2, \sigma_1 + 1/2] \times [\max(\theta(t_1) - 1/2, \theta(T_1)), \theta(t_1)]. \label{xi_def}
\end{align}
By Lemma \ref{barrier_order_lemma}, for $\oh < \sigma < \zeta_* \beta^{-1/2}$ we have
\begin{align}
  \sigma V^+ - \sigma V^- &< C \delta^{-1}\epsilon_v \nu^{1/2}, \\
  \bar W^+ - \bar W^- &< C \epsilon_w \nu^{1/2}.
\end{align}
Also, using that the solution is barricaded, the terms we have written in square brackets are smooth functions of $\sigma$, $\sigma_1 v$, and $w$ within \eqref{xi_def}, independently of the choice of $\sigma_1$ and $t_1$.

Therefore, we may apply regularity to $\sigma_1 (v -  V)$ and $\bar w - \bar W$ to find \ref{conc:tip_reg_noncompact} at $(\sigma_1, t_1)$.
\end{proof}

The control for $\sigma \leq 1$ requires more delicacy.  This requires the knowledge that at $\phi = 0$, we have $v = \sigma^{-1}|\nabla \sigma|_{\tilde g}^2 = 4|\nabla \tilde \phi|_{\tilde g}^2 = 4$  so that as $\tilde \phi$ goes to zero the metric closes off with $\tilde \phi$ behaves likes the radius of polar coordinates near the origin.  Here we pay for our choice of using the length-squared warping function (which we chose to minimize the number of square roots), it is much easier to see the equations in terms of $\tilde \phi = \sqrt{\sigma}$.  Furthermore, instead of controlling $v$ it is easier to understand the evolution of $\tilde L = (1 - \on4 v)/\sigma = (1-|\nabla \tilde \phi|_{\tilde g}^2)/\tilde \phi^2$ which is a smooth function on the warped product.  This is because $v$ naturally satisfies both a Neumann and Dirichlet condition at $\tilde \phi = 0$, whereas $\tilde L$ only satisfies the Neumann condition $\partial_{\tilde \phi}\tilde L = 0$ which comes from it being a rotationally symmetric function.

\begin{lemma}\label{lemma:tip_reg}
  Assume that we are in the setting of Lemma \ref{main_tip_estimates}. We can choose $c_{safe}$ and $\berr C_{reg}$ depending only on the dimensions such that the following holds.  Suppose additionally that item \ref{conc:tip_barrier} and \ref{conc:tip_reg_noncompact} hold for $t \in [T_1, T_2)$.  Then item \ref{conc:tip_reg_noncompact} holds for $t \in [T_1, T_2]$ and for $\sigma \leq 1$.
\end{lemma}
\begin{proof}
We can derive the evolution for $\tilde L$ from \eqref{evo_L_in_phi}.
  \begin{align}
    \partial_{\theta;\phi} \tilde L
    &= \left(1 - \tilde\phi^2 \tilde L\right)\partial^2_{\tilde\phi} \tilde L + \oh \tilde\phi^2 (\partial_{\tilde\phi} \tilde L)^2\\
    &+ \tilde\phi^{-1}(\oh \ein + 5 - \tilde\phi^2\tilde L)\partial_{\tilde\phi} \tilde L + (\ein + 2)\tilde L^2 \\
    &+ c \tilde \phi^{-2} \frac{\alpha}{\omega} v
      (\bar w - \ein_F t/\omega)^{-2} \left( \partial_{\tilde \phi} \bar w \right)^2\\
    &+ \beta \tilde L + \on2 \beta \phi \partial_\phi \tilde L.
  \end{align}
  We can also derive the equation for $\bar w$ in terms of $\tilde \phi$:
  \begin{align}
  \partial_{\theta;\tilde \phi}\bar w
    &= v \partial_{\tilde \phi}^2 \bar w - \frac{\alpha}{\omega}y
      + (\on2\ein - (\on4\ein - 1) v) \tilde \phi^{-1}\partial_{\tilde \phi} \bar w \\
    &+ (\log \omega)_{\theta} \bar w + \on2 \beta \phi \partial_\phi \bar w \\
    &= v \partial_{\tilde \phi}^2 \bar w
      -\on4 v\frac{1}{( \bar w - \ein_F t/\omega)}(\partial_{\tilde \phi}\bar w)^2
      + (\on2\ein - (\on4\ein - 1) v) \tilde \phi^{-1}\partial_{\tilde \phi} \bar w \\
    &+ (\log \omega)_{\theta} \bar w + \on2 \beta \phi \partial_\phi \bar w \\
  \end{align}

  Let $\tilde L_{approx} = \sigma^{-1}( 1 - \on4 V) = \tilde \phi^{-2}(1 - \on4 V)$ which is the approximation for $\tilde L$ given by the approximate solution for $V$.  Our barriers tell us that, for $\sigma < 1$, we have
  \begin{align}
    |\tilde L- \tilde L_{approx}| < c \delta^{-1}\epsilon_{v} \nu^{1/2}, \quad |\bar w - \bar W| < c \epsilon_{w}\nu^{1/2}.
  \end{align}

  The terms $\tilde \phi^{-1}\partial_{\tilde \phi} \tilde L$ and $\tilde \phi^{-1} \partial_{\tilde \phi} \bar w$ appear with integer coefficients, these are not a problem if we consider $\tilde \phi$ as a radial coordinate from $\phi = 0$, then working with the second derivative $\tilde \phi^2$ they make a laplacian.  Furthermore, the term $\tilde \phi^{-2} \left( \partial_{\tilde \phi} \bar w \right)^2$, which appears in the evolution of $\tilde L$, may be controlled as follows.  The regularity up to time $T_2$ gives us control on the $C^{0, \eta}$ norm of this term- specifically, that $|\partial_{\tilde \phi} w|_{0, \eta}^2 \leq C_{reg}\epsilon_w \nu$.  We have to be careful not to have a circular argument here: since this term with $C_{reg}$ appearing is multiplied by something that goes to zero as $t \searrow 0$, we can restrict $T_*$ depending on $C_{reg}$ and thereby bound this term independently of $C_{reg}$.
  
  So, we can apply regularity to $\tilde L - \tilde L_{approx}$ and $\bar w - \bar W$.  Rewriting $\tilde L - \tilde L_{approx} = -\on4 \tilde \phi^{-2}(v - V)$ proves the claim.  
\end{proof}

\subsection{Corollaries of control}\label{corollaries_tip}

The following corollary follows quickly from the control we have, by checking the curvatures of warped products.
\begin{corollary}\label{tip_curvature_control}
  Suppose $g_{wp}(t)$ is controlled in the tip region.
  If $\ein_F = 0$, suppose $F$ has constant curvature.  Then for some $C$, in the tip region,
  \begin{align}
    |\Rm| \leq \frac{C}{t\nu(t)}
  \end{align}
\end{corollary}

We now give a specific result about the convergence in tip region as $t \searrow 0$.  We assume that $g(t)$ is controlled in the tip region for $t \in (0, T_2)$.  For each time, the scaled warping function $\sigma = \frac{u}{t \nu(t)}$ is a function $\sigma : I \to (0, \infty)$ which we extend by the identity to a map $\sigma : M = I \times S^q \times F \to (0, \infty) \times S^q \times F$.  For each $t$, $\sigma$ is a bijection if we restrict to some subset of $I$, i.e. we have an inverse
\begin{align}
  \sigma^{-1}: (0, \sigma_{max}(t)) \times S^q \times F \to I \times S^q \times F.
\end{align}
By our bounds on $v$, specifically since we keep it positive, $\sigma_{max}(t) \to \infty$ as $t \searrow 0$.  We may define
\begin{align}
  G(t) = \frac{1}{\alpha(t)} \left( \sigma^{-1} \right)^* g(t).
\end{align}
As $t \searrow 0$ the domain of definition of $G$ exhausts $(0, \infty) \times S^q \times F$.  Essentially, we can use $\sigma$ to find the diffeomorphisms such that neighborhoods of the tip converge to the Bryant soliton times a Euclidean factor.
\begin{corollary}\label{lemma:barricaded_convergence}
  Suppose that $g(t)$ is controlled in the tip region.
  
  The (for each $t$ partially defined) metric $G(t)$, restricted to $(0,\infty) \times S^q$, converges in $C^{\infty}$ as $t \searrow 0$ to the Bryant soliton metric
  \begin{align}
    \frac{d\sigma_{Bry}^2}{\on4 \sigma_{Bry} v_{Bry}} + \sigma_{Bry} g_{S^q}.
  \end{align}
  The pullback of the vector field $(\partial_\theta \sigma)\partial_\sigma$,
  \begin{align}
    X(t) = \left(\sigma^{-1}\right)^* \left( (\partial_\theta \sigma )\partial_\sigma \right)
  \end{align}
  converges to the soliton vector field for the Bryant soliton.
  
  Put $p = \dim(F)$.  Suppose additionally that $g_{mp}$ is $\Rm$-permissible (Definition \ref{reasonable_def}). For any point $P \in (0, \infty) \times S^q \times F$ the pointed manifolds $((0,\infty) \times S^q \times F, G(t), P)$ converge, as $t \searrow 0$, to
  \begin{align}
    \left(
    (0, \infty) \times S^q \times \real^p
    ,
    \frac{d\sigma^2}{\sigma v_{Bry}} + \sigma^2 g_{S^q} + g_{\real^p}
    ,
    \star
    \right).
  \end{align}
 The target point $\star$ doesn't matter since the target manifold is homogeneous.  The convergence is in the sense of pointed $C^{\infty}$ Riemannian manifolds, which allows a pullback by a time-dependent diffeomorphism.
\end{corollary}
\begin{proof}
  The convergence to the Bryant soliton in terms of $\sigma$ happens up to some number of derivatives just because of the consequences of Lemma \ref{main_tip_estimates}.  To get $C^{\infty}$ convergence, we need extra regularity, i.e. item \ref{conc:tip_reg_noncompact} for larger $k$.  To get this, we use interior parabolic regularity in the same way as Lemmas \ref{lemma:tip_reg_noncompact} and \ref{lemma:tip_reg}.  In this situation, we no longer need estimates on the initial data.  This is because the time variable $\theta$ goes to $-\infty$ as $t \searrow 0$, so the parabolic ball $\Xi$ in \eqref{xi_def} never touches $t=0$, the initial time for $g(t)$.

  Note that $\tilde g(t) = \alpha^{-1}g(t)$ satisfies
  \begin{align}
   \partial_{\theta} \tilde g
    = -2 \Rc[\tilde g] - \beta \tilde g.
  \end{align}
  So $G(t)$ satisfies
  \begin{align}
    \partial_{\theta} G
    =
    -2 \Rc[G]
    - \lie_{\left( \partial_{\theta}\sigma \right) \partial\sigma} G
    - \beta \tilde g.
  \end{align}
  As $t \searrow 0$, we have $\beta \searrow 0$, $G \to G_{Bry}$, and $\partial_{\theta}G \to 0$.  This shows the convergence of $\partial_{\theta}\sigma$ to the soliton vector field.

  To get the final convergence of the $wg_F$ factor to $g_{\real^{p}}$, note that we have
  \begin{align}
    w \sim \omega - \ein_F t
  \end{align}
  so $\alpha^{-1}w \sim \alpha^{-1}\omega \left( 1 - \ein_F t/\omega\right)$.  In the case $\ein_F < 0$, this goes to $\infty$ at least as fast as $\frac{t}{\alpha} = \frac{1}{\nu}$ goes to infinity.  In the case $\Lambda_F > 0$ or $\ein_F = 0$, this goes to infinity by the assumption that $g_{mp}$ is $\Rm$-permissible.
\end{proof}


\section{Full flows of mollified metrics}\label{section:full_flow}


In Sections \ref{section:productish} and \ref{section:tip}, we studied the flow in two regions- the productish region and the tip region.  We now want to start from one of our model pinches and create mollified initial metrics.  The mollified metrics will exist for a uniform amount of time and satisfy the estimates from Lemmas \ref{main_prish_estimates} and \ref{main_tip_estimates}.  We will then take a limit of the mollified flows to construct a forward evolution from the model pinch.

In the previous two sections we constructed functions, which depend on $u$ and time, and serve as barriers of the flow around approximate solutions.  Let $V_{prish}$ and $W_{prish}$ be the approximate solutions constructed in Section \ref{section:productish}, and let $V^+_{prish}$, $V^-_{prish}$, $W^+_{prish}$, $W^-_{prish}$ be the functions constructed in Lemma \ref{barriers_prish}.  Let $V_{tip}$ and $W_{tip}$ be the approximate solutions constructed in Section \ref{section:tip}, and let $V^+_{tip}$, $V^-_{tip}$, $W^+_{tip}$, $W^-_{tip}$ be the functions constructed in Section \ref{tip_barriers_section}.  We remind the reader that $\bar w = \omega^{-1}(w - \ein_F t)$, and we decorate the barriers and approximate solutions with a bar analogously.

As a first step, the following lemma tells us how close the approximate solutions are to each other.  Here, $|f|_{2,\eta}$ is the $C^{2, \eta}$ norm for the metric $(d\sigma)^2$, in the ball of radius $1$ around a given point.
\begin{lemma}\label{basic_deriv_closeness}
  For $\sigma < \epsilon \rho_* \nu^{-1}$,
  \begin{align}
    \sigma |V_{prish} - V_{tip}|_{2,\eta} &\leq C(\rho_*)
    \left( \nu^2 \sigma^2 + \sigma^{-1} + \nu  \right) \label{V_approx_diff}\\
    |\bar W_{prish} - \bar W_{tip}|_{2,\eta} &\leq C(\rho_*)
    \left( \nu^2 \sigma^2 + \nu \log \sigma \right)
  \end{align}
\end{lemma}
\begin{proof}
  We claim we can use $V_{common} = \ein \sigma^{-1}\left( 1 + (1 + \dd1{\nu})\ein^{-1}\nu \sigma \right)$ as an approximation for both $V_{prish}$ and $V_{tip}$, and similarly that we can use $\bar W_{common} =  (1 + \ein \dd1{\omega} \nu \sigma)$ as an approximation for both $\bar W_{prish}$ and $\bar W_{tip}$.

  For $V_{prish}$ and $W_{prish}$, the zeroth order statement follows from the approximations \eqref{V_expr_smallrho} and \eqref{Wprish_asymptotic_exp}.  The higher order statements can be found similarly to how we found \eqref{V_expr_smallrho} and \eqref{Wprish_asymptotic_exp}, by estimating the Taylor expansion of the derivatives.  We have,
  \begin{align}
    \sigma |V_{prish}-V_{common}| + |\bar W_{prish} - \bar W_{common}| \leq C \nu^2 \sigma^2.
  \end{align}
  
  For $V_{tip}$, the zeroth order statement follows from the asymptotics \eqref{vbry_infty_asymptotics} and \eqref{Vpert_asymptotics} for $V_{\Bry}$ and $V_{\Pert}$, and the fact that $\beta = (1 + \dd1{\nu})\nu$.  For $W_{tip}$, it follows from the asymptotics \eqref{Wpert_asymptotics} for $W_{\Pert}$.  To get the higher order statements, one needs to use the analyticity of the involved functions.  We have,
  \begin{align}
    \sigma |V_{tip}-V_{common}|_{2, \eta} \leq C(\sigma^{-1} + \nu).
    |\bar W_{tip} - \bar W_{common}|_{2, \eta} \leq C \nu \log \sigma
  \end{align}
  (The term $\sigma^{-1}$ comes from the error in $ V_{\Bry} \sim \ein \sigma^{-1}$, and the term $\nu$ comes from the error in $\nu V_{pert} \sim \nu$.  The term $\nu \log \sigma$ comes from the error in $\nu W_{pert}\sim \nu \sigma$.)
\end{proof}

\subsection{Buckling barriers}\label{section:buckling}

\begin{figure}[t]
  \centering
  \includegraphics[scale=.8]{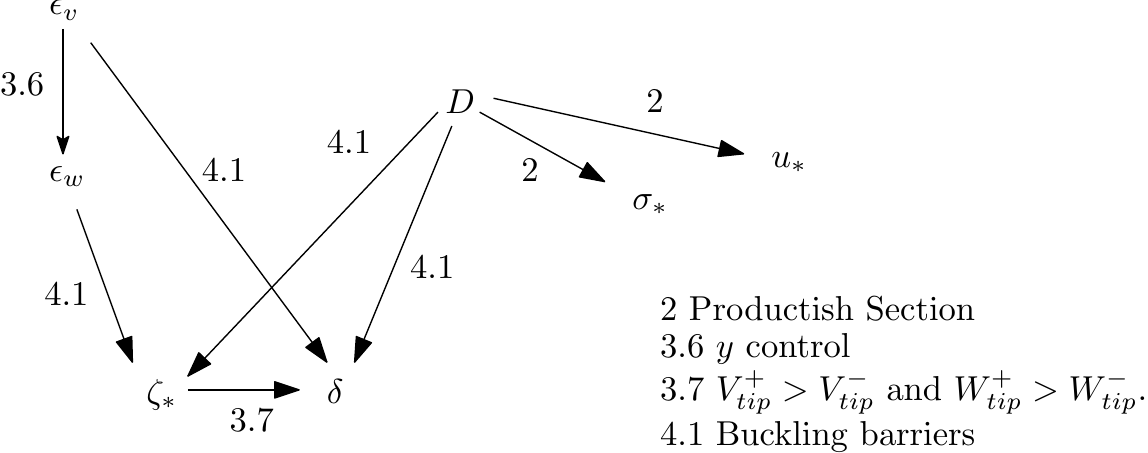}
  \caption[Constant dependency graph]{\textbf{Constant dependency graph.}  All constants only depend on the constants which point to them.  The arrows are marked with the sections where the dependency arises.  $T_*$ is allowed to depend on all constants.}
\end{figure}

In this section, we prove Lemma \ref{buckling_lemma}.  This shows that the barriers are ordered in a specific way: see Figure \ref{figure:barrier_switch}.  The point is that this ordering means that boundary condition for the tip barriers is guaranteed by the productish barriers, and the left-hand boundary condition for the productish barriers is guaranteed by the tip barriers.  We formalize this consequence in  Lemma \ref{buckling_wrapup}.

\begin{lemma}\label{buckling_lemma}
  Let $\epsilon_v$, $\epsilon_w$, and $\sigma_*$ be given.  Assume $D > \berr D$, $\zeta > \berr \zeta_*(D, \epsilon_w)$, $\delta < \barr \delta(\zeta_*, \epsilon_v, D)$, and finally $T_*$ is chosen depending on all other parameters. 
  
  Then we have the following inequalities.  
  For $\zeta_*\nu^{-1/2} \leq \sigma \leq 2\zeta_*\nu^{-1/2}$,
  \begin{align}
    V_{tip}^+ > V_{prish}^+ &\quad V_{tip}^- < V_{prish}^-, \\
    W_{tip}^+ > W_{prish}^+ &\quad W_{tip}^- < W_{prish}^-. 
  \end{align}
  For $\oh \sigma_* \leq \sigma \leq \sigma_*$,
  \begin{align}
    V_{prish}^+ > V_{tip}^+ &\quad V_{prish}^- < V_{tip}^-, \\
    W_{prish}^+ > W_{tip}^+ &\quad W_{prish}^- < W_{tip}^- 
  \end{align}
\end{lemma}
\begin{proof}
  We note the following inequalities:
  \begin{align}
    c D \sigma^{-1} &< \sigma V_{prish}^+ - \sigma V_{prish} < C D \sigma^{-1}, \\
    c D \sigma^{-1} &< \bar W_{prish}^+ - \bar W_{prish} < C D \sigma^{-1}
  \end{align}
  This comes from the definition of the barriers $V_{prish}^{\pm} = (1 \pm D V)V$ and $\bar W_{prish}^{\pm} = (1 \pm D V)\bar W$, together with $V \sim \sigma^{-1}$ and $\bar W \sim 1$.
  Also, provided we take $\delta < c \zeta_*^{-1}$, by Lemma \ref{barrier_order_lemma} we have
  \begin{align}
    c \delta^{-1}\epsilon_{v} \nu^{1/2} 
    &\leq \sigma V_{tip}^+ - \sigma V_{tip}
    \leq C \delta^{-1}\epsilon_v \nu^{1/2} ,\\
    c \epsilon_w \nu^{1/2}
    &\leq \bar W_{tip}^+ - \bar W_{tip}
    \leq C \epsilon_w \nu^{1/2}.
  \end{align}

  We can put all these inequalities, together with \eqref{basic_deriv_closeness}, in terms of $\zeta$:
\begin{align}
  \sigma |V_{prish} - V_{tip}|
  &\leq C(\rho_*)
    \left( \nu \zeta^2 + \zeta^{-1} \nu^{1/2} + \nu  \right), \label{bar_diff}\\
  |\bar W_{prish} - \bar W_{tip}|
  &\leq C(\rho_*)
    \left( \nu \zeta^2 + \nu |\log\nu| + \nu |\log \zeta| \right),
\end{align}
\begin{align}
  c D \zeta^{-1}\nu^{1/2} &< \sigma V_{prish}^+ - \sigma V_{prish} < C D \zeta^{-1}\nu^{1/2}, \label{prish_bar_sep}\\
  c D \zeta^{-1}\nu^{1/2} &< W_{prish}^+ - W_{prish} < C D \zeta^{-1}\nu^{1/2},
\end{align}
\begin{align}
    c \delta^{-1}\epsilon_{v} \nu^{1/2} 
    &\leq \sigma V_{tip}^+ - \sigma V_{tip}
    \leq C \delta^{-1}\epsilon_{v} \nu^{1/2} ,\label{tip_bar_sep}\\
    c \epsilon_w \nu^{1/2}
    &\leq \bar W_{tip}^+ - \bar W_{tip}
    \leq C \epsilon_w \nu^{1/2}.
\end{align}

We now use the inequalities \eqref{bar_diff}, \eqref{prish_bar_sep}, and \eqref{tip_bar_sep} to prove the desired inequalities for the supersolutions.  The desired inequalities for the subsolutions are similar.

First we deal with the inequality at $\sigma_*/2 < \sigma < \sigma_*$, where we wish to show that $V^+_{prish} > V_{tip}^+$.  By applying \eqref{prish_bar_sep}, then \eqref{bar_diff}, then \eqref{tip_bar_sep} we find
\begin{align}
  \sigma V^+_{prish}
  &\geq \sigma V_{prish} + c D \sigma^{-1} \\
  &\geq \sigma V_{tip} + c D \sigma^{-1} \\
  &- C \left( \nu^2 \sigma^2 + \sigma^{-1} + \nu \right) \\
  &\geq \sigma V_{tip}^+ + c D \sigma^{-1}\\
  &- C \left( \nu^2 \sigma^2 + \sigma^{-1} + \nu \right)
   - C \delta^{-1}\epsilon_{v} \nu^{1/2}.
\end{align}
Choosing $D$ such that $c D \geq 2 C$ means that $\delta V^+_{prish} \geq \sigma V^+_{tip}$ at least for short time.  Showing that $W^+_{prish} > W_{tip}^+$ is similar.

Now we deal with the inequalities for $\zeta_*  \leq \zeta \leq 2 \zeta_*$.  First choose $\zeta_* \geq 10\frac{CD}{c \epsilon_w}$, and then chose $\delta \leq \on{10} (CD)^{-1}c \epsilon_v \zeta_*$.  Then we have, using \eqref{tip_bar_sep},
\begin{align}
  \sigma V^+_{tip}
  &\geq \sigma V_{tip} + c \delta^{-1}\epsilon_{v} \nu^{1/2} \\
  &\geq \sigma V_{tip} + 10 C D \zeta^{-1}\nu^{1/2}.
\end{align}
Now using \eqref{bar_diff} and then \eqref{prish_bar_sep}, for $\zeta_* \leq \zeta \leq 2 \zeta_*$,
\begin{align}
  \sigma V^+_{tip}
  &\geq \sigma V_{prish} + 10 C D \zeta^{-1}\nu^{1/2}\\
  &- C \left(  \nu \zeta^2 + \nu \right)
    - C \zeta^{-1}\nu^{1/2} \\
  &\geq \sigma V_{prish}^+ + 10 C D \zeta^{-1}\nu^{1/2}\\
  &- C \left(  \nu \zeta^2 + \nu \right)
    - C \zeta^{-1}\nu^{1/2}
    - CD \zeta^{-1}\nu^{1/2} \\
  &\geq \sigma V_{prish}^+ + 8 CD \zeta^{-1}\nu^{1/2}
\end{align}
with the last line valid for small enough times.  Therefore, for small enough times, $V^+_{tip} \geq V^+_{prish}$ here.  The calculation is similar for $W$; since $\zeta \geq 10 \frac{CD}{c \epsilon_w}$ the upper bound $CD \zeta^{-1} \nu^{1/2}$ on $\bar W^+_{prish} - W_{prish}$ is dominated by the lower bound $c \epsilon_w \nu^{1/2}$ on $\bar W_{tip}^+ - \bar W_{tip}$.

\end{proof}

  We take a moment here to remark on the design of the tip barriers. To understand the term $\nu^{1/2}$ in the barriers' definitions, consider what would happen in Lemma \ref{barrier_order_lemma} if we replaced $\nu^{1/2}$ with some function $f(\nu) \ll \nu^{1/2}$.  We would still have
  \begin{align}
    V_{diff} = V^+ - V \geq c \delta^{-1}\epsilon_{v} f(\nu) \sigma^{1,-1} - C\epsilon_v \nu \sigma^{1,0}
  \end{align}
  and upon pulling out the factor $\delta^{-1}\epsilon_vf(\nu)\sigma^{1,-1}$,
  \begin{align}
    V_{diff} \geq c \delta^{-1}\epsilon_{v} f(\nu) \sigma^{1,-1}\left( 1  - C\delta \frac{\nu}{f(\nu)} \sigma^{0,1}\right).
  \end{align}
  Since $f(\nu) \ll \nu^{1/2}$, $\frac{\nu}{f(\nu)} \gg f(\nu)$, so the region where $V_{diff} > 0$ is not contained in the region $f(\nu) \sigma \leq \zeta_*$ for any $\zeta_*$.

  However, in Lemma \ref{buckling_lemma}, it was important that the region where $V_{diff} > 0$ is contained in the region $f(\nu)\sigma \leq \zeta_*$.  The reason is that, in approximating the first term of $V^+_{tip}$, we use the asymptotics of $V_{Bry}$ to say
  \begin{align}
    V_{k^+\Bry} =
    (\ein + \delta^{-1}\epsilon_v f(\nu))\sigma^{-1}
    + O(\sigma^{-2}).
  \end{align}
    The term $\ein \sigma^{-1}$ matches with the leading order term of the approximation for $V$ coming from the productish region \eqref{V_expr_smallrho}.  The $O(\sigma^{-2})$ term is essentially uncontrollable and falls into the error between $V_{tip}$ and $V_{prish}$ in \eqref{V_approx_diff}. (We could find its sign by studying the Byrant soliton more closely, but that would only help us for either the sub- or supersolution.)  Then we need the left over term $f(\nu)\sigma^{-1}$ to cover $O(\sigma^{-2})$- in other words, we need $f(\nu)\sigma \geq C$ for some $C$.

  Therefore the $\nu^{1/2}$ is somehow optimal, at least for the technique that we are using.

  The point of the inequalities in Lemma \ref{buckling_lemma} is that they immediately imply Lemma \ref{buckling_wrapup} below.  This says that we can remove the assumption in Lemma \ref{main_prish_estimates} which assumed that the solution stays within the productish barriers on the left edge of the productish region, and we can remove the assumption from \ref{main_prish_estimates} which assumed that the solution stays within the tip barriers on the right edge of the tip region.

\begin{figure}[t]
  \centering
  \includegraphics[scale=.4]{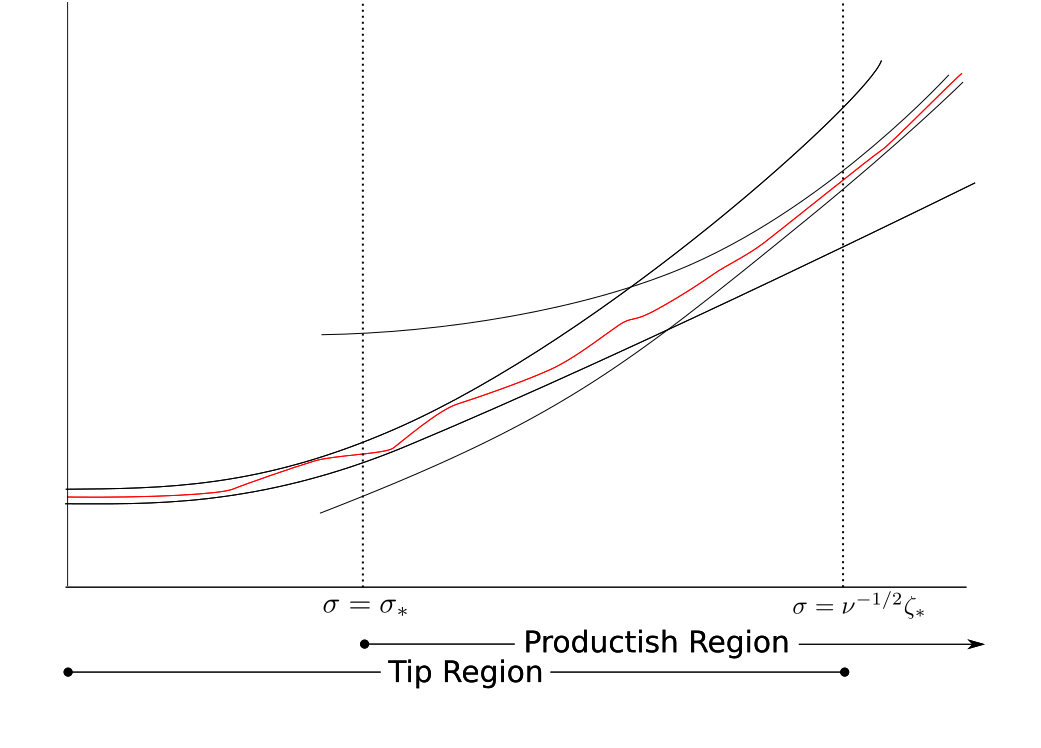}
  \caption
  [Buckling barriers]
{\textbf{Buckling barriers.}  The red solution lies between the productish barriers in the productish region, and the tip barriers in the tip region.  Because of the ordering of the barriers at $\sigma = \sigma_*$, the boundary conditions for the productish barriers are automatically satisfied. Similarly at $\sigma = \beta^{-p}\zeta_*$.\label{figure:barrier_switch}}
\end{figure}

\begin{lemma}\label{buckling_wrapup}
  Let $D > \berr D$, $C_{reg} > \berr C_{reg}$, $u_* < \barr u_*(D, C_{reg})$, $\sigma_* > \berr \sigma_*(D, C_{reg})$, $\epsilon_v$, $\epsilon_w < \berr \epsilon_w(\epsilon_v)$, $\zeta_* \geq \berr \zeta_*(\epsilon_w, D)$, and $\delta < \barr \delta(\epsilon_v, D, \zeta_*)$ be given.  There is a $T_*$ depending on all parameters such that if $T_2 < T_*$ we have the following.

  Let $0 < T_1 < T_2 < T_*$.   Assume that the initial metric is controlled at the initial time in the productish and tip regions, and also controlled at the right of the productish region.  Then we have the conclusions \ref{conc:prish_barrier}, \ref{conc:prish_reg} from Lemma \ref{main_prish_estimates} and \ref{conc:tip_barrier}, \ref{conc:tip_reg_noncompact} from Lemma \ref{main_tip_estimates}.
\end{lemma}
\begin{proof}
  Let $T_{bad} > T_1$ be the maximal time such that all the conclusions hold for $g(t)$ on $[T_1, T_{bad})$.  By Lemma \ref{buckling_lemma}, the assumption that the solution is barricaded on the right edge of the tip region is satisfied on $[T_1, T_{bad}]$, since the productish region barriers are tighter than the tip region barriers there.  Similarly, the assumption that the solution is barricaded on the left edge of the productish region holds on $[T_1, T_{bad}]$.  By the assumptions of our lemma, all other assumptions needed to apply Lemmas \ref{main_prish_estimates} and \ref{main_tip_estimates} hold on $[T_1, T_{bad}]$.  Therefore all the conclusions still hold at time $t = T_{bad}$.
\end{proof}

The assumptions that $g$ is well controlled in the productish and tip regions are all assumptions on the metric at time $T_1$.  The only assumption left after Lemma \ref{buckling_wrapup} that is an a priori assumption on the forward evolution is that the metric is barricaded at the right of the productish region.

From now on we consider the constants $D$, $C_{reg}$, $u_*$, $\sigma_*$, $\epsilon_v$, $\epsilon_w$, $\zeta_*$, and $\delta$ to be fixed and satisfying Lemma \ref{buckling_wrapup}.

\subsection{Mollifying metrics}\label{section:mollifying_section}
In this section we will define mollified metrics, and prove some basic properties.  We introduce a smooth cutoff function $\eta(x):[0, \infty) \to [0, 1]$ which satisfies
\begin{align}
  \begin{cases}
    \eta(x) = 1 & x < 1 \\
    \eta(x) \in [0, 1] & 1 \leq x \leq 2\\
    \eta(x) = 0 & x > 2
  \end{cases}
\end{align}
and define $\eta_{r}(x) = \eta(x/r)$.

Now, for arbitrary sufficiently small $m$, and $T^{(m)}_1$ to be determined, we define
\begin{align}
  V_{init}^{(m)}
  &=
    \begin{cases}
      \eta_{2\zeta_*}(\zeta) V_{tip}(u, T^{(m)}_1)
      + (1 - \eta_{2\zeta_*}(\zeta)) V_{prish}(u, T^{(m)}_1)
      & \zeta_* \nu^{-1/2} \leq \zeta \leq 4 \zeta_* \nu^{-1/2}\\
      V_{prish}(u, T^{(m)}_1)
      & 4\zeta_* t\nu^{1/2} \leq u \leq m\\
      \eta_m(u) V_{prish}(u, T^{(m)}_1) + (1-\eta_m(u))V_0(u) 
      & m \leq u \leq \infty
    \end{cases}
\end{align}
and define $W_{init}^{(m)}$ similarly.  These functions agree with $V_0$ and $W_0$ for $ u > 2m$, agree with the productish approximation (evaluated at time $T_1^{(m)}$) for $4 \zeta_* t \nu^{1/2} < u \leq m$, and agree with the tip approximation (evaluated at time $T_1^{(m)}$) for $\zeta < 2 \zeta_*$.

So far we have just been dealing with the diffeomorphism invariant considerations of $v$ and $w$ as functions of $u$ and $t$.  Now fix a model pinch metric $g_{mp}$ on $M = I \times S^q \times F$, with the corresponding functions $V_0(u)$ and $W_0(u)$.  We write $u_0:M \to \real_+$ to be the initial value of $u$ at a given point in $M$.  Then
\begin{align}
  g_{mp} = \frac{du_0^2}{u_0 V_0(u_0)} + u_0 g_{S^q} + W_0(u_0)g_F.
\end{align}
We also define $M_{[u_1, u_2]} = \{p \in M : u_0(p) \in [u_1, u_2]\}$.  Now we define mollifications $g_{init}^{(m)}(t)$.  We define them as
\begin{align}
  g^{(m)}_{init} = \frac{du_0^2}{u_0 V^{(m)}_{init}(u_0)} + u_0^2 g_{S^q} + W^{(m)}_{init}(u_0, t) g_F.
\end{align}
Note that $g^{(m)}_{init}$ is equal to $g_{mp}$ in $\bar M_{[2m, \infty)}$, and is smooth.  It may seem that we have repeated ourselves, since we have already chosen $V^{(m)}_{init}$ and $W^{(m)}_{init}$.  The point here is we are also fixing the coordinate of the interval factor.

The following Lemma says that $g_{init}^{(m)}$ satisfies all of the conditions on the initial metric required by Lemmas \ref{main_prish_estimates} and \ref{main_tip_estimates}.
\begin{lemma}\label{gm_good_init}
  Let $m < \berr m$ and suppose $T^{(m)}_1 < \berr T^{(m)}_1(m) < m$.  Let $g_{init}^{(m)} = g^{(m)}(T^{(m)}_1).$  Then for $T_1 = T^{(m)}_1$, the metric $g_{init}^{(m)}$ is initially controlled in the productish and tip regions.
\end{lemma}
\begin{proof}
  That $g_{init}^{(m)}$ is initially controlled in the tip region is immediate, because the functions $v$ and $w$ for $g_{init}^{(m)}$ exactly agree with with the functions $V_{tip}$ and $W_{tip}$ in the tip region.

  Where $v$ and $w$ agree with $V_{prish}$ and $W_{prish}$, the assumptions in the productish region are automatic.  This is true in $M_{[4\zeta_* T_1^{(m)} \nu^{1/2}, m]}$.  What's left is to check the assumptions in $M_{[\sigma_*T_1^{(m)} \nu, 2 \zeta_*T_1^{(m)}\nu^{1/2}]}$ and $M_{[m, 2m]}$.

  Both conditions hold for $u_0 \leq \zeta_* t \nu^{1/2}$ by Lemma \ref{basic_deriv_closeness}, the definition of $V_{init}^{(m)}$ and $W_{init}^{(m)}$, and the separation of the barriers.  To check the conditions in $M_{[m, 2m]}$, note that they hold strictly in this compact set at time $t = 0$, so for sufficiently small $T_1^{(m)}$ they will continue to hold.
\end{proof}

\subsection{Controlling curvature and convergence}\label{section:ccc}
Since $g^{(m)}_{init}$ is smooth, there is a solution to Ricci flow $g^{(m)}(t)$ on $[T_1^{(m)}, T_{final}^{(m)})$ with $g^{(m)}(T_1^{(m)}) = g_{init}$.  We want to control $g^{(m)}(t)$.  By Lemmas  \ref{buckling_wrapup} and \ref{gm_good_init}, in order to get the conclusions of Lemmas \ref{main_prish_estimates} and \ref{main_tip_estimates} we just need the condition that the solution is between the barriers for $u_* < u < 2u_*$.  Let $T_2^{(m)}$ be the maximal time such that this condition holds on $[T_1^{(m)}, T_2^{(m)})$. In Corollary \ref{T_lower_bound} we will argue that we have a fixed lower bound on $T_2^{(m)}$. In each lemma we may decrease $T_*$.

We do something sort of silly here.  For this section, we assume that $(F, g_F)$ has constant sectional curvature.  This is so that we can have control on $|\Rm|$ via Corollaries \ref{prish_curvature_control_time} and \ref{tip_curvature_control}.  The control on $|\Rm|$ lets us use the full regularity theory for Ricci flow.  In the end, we can replace the constant sectional curvature fiber with anything we want, since the Ricci flow of warped products only cares about the Ricci curvature of the fiber.  In our case, we could also get this higher regularity by going through the regularity of the involved parabolic PDE on the interval, but it's easier to invoke generic Ricci flow estimates.

\begin{lemma}\label{rm_ctrl_in_x}
    Suppose $m < \barr m$. For any $k$, there is a constant $C_k$ depending only on $V_0$, $W_0$, and $u_*$ such that in $M_{[u_*/4, \infty]}$ and for $t \in [T_1^{(m)}, \min(T_*, T_{2}^{(m)})]$,
  \begin{align}
    |\nabla^k\Rm_{g^{(m)}}| < C_k
  \end{align}
\end{lemma}
\begin{proof}
  This is by now standard procedure, see e.g. Corollary A.5 of \cite{topping}.
  The curvatures of the metrics $g^{(m)}_{init}$ have a uniform bound on their curvature and the volume of small enough balls in $M_{u_*/4, \infty}$  Therefore we can apply the pseudolocality theorem (Theorem 10.3 of \cite{Perelman}) at any point there, to get control on $|\Rm|$, and then apply local derivative estimates (14.4.1 of \cite{bookanalytic}) to get control on higher derivatives.
\end{proof}
Since our barrier control is in terms of $u$, we need to be able to transfer the set written in terms of $u$ to being written in terms of $x$.
\begin{lemma}\label{u_ctrl_in_x}
  Suppose $m < \barr m$.  Then for all $t \in [T_1^{(m)}, \min(T_*, T_2^{(m)})]$,
  \begin{align}
    \{p \in M : u^{(m)}(x,t) \in [u_*, 2u_*]\}
    \subset
    \bar N_{[u_*/4, 4u_*]}
  \end{align}
\end{lemma}
\begin{proof}
  At time $t = T^{(m)}_1$, we have $u_*/4 \leq u^{(m)} \leq 4 u_*{(m)}$ in $M_{[u_*/4, 4u_*]}$ (just by definition).  By Lemma \ref{rm_ctrl_in_x}, there is a uniform speed limit on $u$ in $M_{[u_*/4, \infty]}$.  Therefore for $x \geq 4u_*$, $u$ cannot decrease too fast and so we can get a time $T_*$ so that $u$ will not go below $u_*$ before time $T_*$.

  Also, we can decrease $T_*$ so that $u$ cannot go above $u_*$ in $M_{[0, u_*/4]}$.  Since the conclusions of Lemmas \ref{main_prish_estimates} and \ref{main_tip_estimates} hold for $t \in [T_1^{(m)}, T_2^{(m)}]$, $v$ is between its barriers for these times, and is in particular positive for $u \in [0, 2u_*]$.  Therefore $u$ is increasing up to the value $2u_*$.  Therefore, $u$ is smaller than $u_*$ for $x< u_*/4$.
\end{proof}

\begin{lemma}\label{full_curvature_ctrl_time}
  For any  $k \in \nats \cup \{0\}$, there is a constant $C_{time, k}$ such that
  \begin{align}
    |\nabla^k \Rm_{g^{(m)}}| \leq \frac{C_{time,k}}{t_0^{k/2}t_0 \nu(t_0)}.
  \end{align}
  for all $t \in \left[\max(t_0, (2-2^{-k})T_1^{(m)}), \min(T_*, T_2^{(m)})\right]$.
\end{lemma}
\begin{proof}
  For $k=0$, this is exactly Lemma \ref{prish_curvature_control_time}, Lemma \ref{tip_curvature_control}, and Lemma \ref{rm_ctrl_in_x}.  (Remember that in this section we assume $(F, g_F)$ has constant sectional curvature.) For $k > 0$, we can apply Shi's derivative estimates (Theorem 1.1 of \cite{Shi}), using the result for $k-1$ and for times larger than $\max\left(t_0/2, (2-2^{-k+1})T_1\right)$.  The factor in front of $T_1$ ensures us that for any $t_0$ and $k$, and for any point where we want to apply the regularity, there is a uniformly sized (in $m$) parabolic ball of Ricci flow which has the order $k-1$ estimates.
\end{proof}

\begin{lemma}\label{T_lower_bound}
  $T_{final}^{(m)} > T_2^{(m)} > T_*$.
\end{lemma}
\begin{proof}
  By Lemma \ref{full_curvature_ctrl_time}, Ricci curvature is bounded at time $T_2^{(m)}$.  Therefore, $T_{final}^{(m)} > \min(T_*, T_2^{(m)})$.  
  By Lemmas \ref{rm_ctrl_in_x} and \ref{u_ctrl_in_x} the curvature and its derivatives are bounded for $u^{(m)}(x,t) \in [u_*, 2u_*]$.  This implies a speed limit on the functions $v^{(m)}$ and $w^{(m)}$ there.  Since the functions are uniformly separated from the barriers are time $t = T^{(m)}_1$, they cannot pass the barriers for some fixed time.  Therefore $T_2^{(m)} > T_*$, possibly taking $T_*$ smaller.
\end{proof}

We now have all of the conclusions of Lemmas \ref{main_prish_estimates} and \ref{main_tip_estimates}, for each $g^{(m)}(t)$, on $[T_1^{(m)}, T_*]$.  Now, we get estimates within fixed subsets of $M$, extending the crude ones from Lemmas \ref{rm_ctrl_in_x} and \ref{u_ctrl_in_x}.

\begin{lemma}\label{u_lower_bnd}
  There is a constant $C > 0$ depending only on $g_{mp}$ such that the following holds.  Let $u_\# \in [0, u_*]$ and suppose $m < u_{\#}$.  Then
  in $p \in M_{[u_{\#}, \infty)}$ and for $t < (C+1)^{-1}u_\#$, we have $u^{(m)} \geq u_\# - C t$.
\end{lemma}

\begin{proof}
    Since $T_1^{(m)} < m$, at the beginning time $T_1^{(m)}$ any point $p \in M_{[u_{\#}, u_*]}$ lies in the productish region, which is defined as the points where $u \geq t \nu(t) \sigma_*$.  (By restricting $T_*$, we can assume $\nu(T_1^{(m)}) < \frac{1}{\sigma_*}$.)

  The function $u^{(m)}$ satisfies the evolution equation
  \begin{align}
    \partial_t u^{(m)} = \lap_M u^{(m)} - 2 (u^{(m)})^{-1}|\nabla u^{(m)}|^2 - \ein
  \end{align}
  and as long as $u^{(m)}$ is in the productish region, we have the estimate (using the regularity in conclusion \ref{conc:prish_reg})
  \begin{align}
    |\lap_M u^{(m)}| + (u^{(m)})^{-1}|\nabla u^{(m)}|^2 \leq C.
  \end{align}
  Therefore,
  \begin{align}
    u(p, t) \geq u_\# - (\ein + C)(t - T_1^{(m)}) \geq u_{\#} - (\ein + C) t.
  \end{align}
  Now, $p$ continues to be in the productish region as long as $u \geq \sigma_* t \nu(t)$, so at least as long as
  $
    u_{\#} - (\ein + C)t \geq t \nu(t) \sigma_*
  $.
  Since we assume $\nu(t) < 1/\sigma_*$, this will be implied if $t \leq \frac{u_{\#}}{\ein + C + 1}$.  Therefore we have proven the first bullet.

\end{proof}

\begin{lemma}\label{full_ctrl_curvature_x}
  For $u_{\#} \in [0, u_*]$ and $k \in \nats \cup \{0\}$ there is a constant $C_{space,k}(u_{\#})$ such that the following holds.  Suppose $u_\# > (2-2^{-k})m$.  Then in $M_{[u_{\#}, \infty)}$ and for $t \in [T_1^{(m)}, T_*]$,
  \begin{align}
    |\nabla^k \Rm_{g^{(m)}}| \leq C_{space,k}(u_\#).
  \end{align}
\end{lemma}
\begin{proof}
  Once we prove the Lemma for $k = 0$, the result follows for $k > 0$ using local derivative estimates and the result for $k' = k-1$ and $x_0' = 2x_0$.  (In particular, we have to use the knowledge that all of the metrics $g^{(m)}(T_1^{(m)})$ agree in $M_{[u_{\#}, \infty)}$ and therefore have uniform curvature bounds, and we can use 14.4.1 of \cite{bookanalytic}.)

  Now we do $k = 0$.  For $t < (C+1)^{-1}u_{\#}$ we an apply \ref{u_lower_bnd} to find that $u^{(m)} > u_{\#} - C t$.  So, we stay in the productish region and we can apply Lemma \ref{prish_curvature_control} to get a bound on $|\Rm|$.  (This bound will depend on $W_0(u)$, but for example if $W_0(u) \gtrsim u$ we can get $|\Rm| \lesssim u_{\#}^{-1}$.)  On the other hand, for $t > (C+1)^{-1}u_{\#}$ we can apply Lemma \ref{full_curvature_ctrl_time} to find $|\Rm| \leq \frac{C_{time, 0}}{(C+1)^{-1}t_\# \nu((C+1)^{-1}t_\#)}$.  
\end{proof}
\begin{lemma}\label{full_ctrl_derivs_x}
  For $u_{\#} \in [0, u_*]$ and $k \in \nats \cup \{0\}$ there is a constant $C_{k}(u_{\#})$ such that the following holds.  Suppose $u_\# > 2m$.  Then in $M_{[u_\#, \infty)}$ and for $t \in [T_1^{(m)}, T_*]$,  
  \begin{align}
    |\left( \nabla^{g_{mp}} \right)^k g^{(m)}(x, t)|_{g_{mp}} \leq C(x_0, k).
  \end{align}
\end{lemma}
\begin{proof}
  This follows by integrating the Ricci flow equation and derivatives of the Ricci flow equation.
\end{proof}


Now we prove Theorem \ref{theorem:model_pinch_flow}.  Recall that we identify $M = I \times S^q \times F$ with $(\real^{1+q} \setminus \{0^{1+q}\}) \times F \subset \bar M \defeq \real^{1+q} \times F$, where $0^{1+q}$ is the origin in $\real^{1+q}$.  We will construct the Ricci flow $g(t)$ provided by Theorem \ref{theorem:model_pinch_flow} as a limit of our flows of mollified metrics $g^{(m)}(t)$.  As a little notational annoyance, we set $g_{shift}^{(m)}(t) = g^{(m)}(t-T_1^{(m)})$, which is a Ricci flow for times at least $[0, T_*/2]$.  ($g^{(m)}$ has the nice property that $g^{(m)}$ evaluated at time $t$ is approximately our approximate solution at time $t$, whereas $g_{shift}^{(m)}$ is nice because it always starts at time $0$.)  We let $P = \bar M \setminus M$.

\begin{lemma}\label{final_lemma}
  There is a sequence $m_k \searrow 0$, and a family of metrics $g_{wp}(t)$, $t \in [0, T_*/2]$ such that
  $
    g_{shift}^{(m_k)}(t) \to g_{wp}(t)
    $
  in $C^{\infty}_{loc}\left( \bar M \times [0, T_*/2] \setminus  P \times \{0\}  \right)$.

  $g_{wp}(t)$ is a Ricci flow satisfying all of the conclusions of Theorem \ref{theorem:model_pinch_flow}.
\end{lemma}

\begin{proof}
  First we get convergence in $C^{\infty}_{loc}(M \times [0, T_*/2])$.  For any $u_\# > 0$ and for $m > 2u_\#$ we have $C^{\infty}$ control on the derivatives of $g_{shift}^{(m)}$ in $M_{[u_\#, \infty)} \times [0, T_*/2]$  (Lemma \ref{full_ctrl_derivs_x}).  By the Arzel\`a-Ascoli theorem, we get convergence of a subsequence in any such region.  By taking a diagonal subsequence, we get convergence to a metric on $g(t)$ as desired.  Since the convergence happens in $C^{\infty}$, the Ricci flow equation and all the estimates pass to the limit.

  Next we can continue extracting subsequences to get convergence on $C^{\infty}_{loc}(\bar M \times [t_0, T_*/2))$, for any $t_0 > 0$, using Lemma \ref{full_curvature_ctrl_time}.  By performing another diagonal argument we can get convergence in the claimed $C^{\infty}_{loc}$ space.

  For any $t > 0$, the doubly-warped product metric $g_{wp}(t)$ satisfies the inequalities in the conclusions of Lemma \ref{main_prish_estimates} and \ref{main_tip_estimates}. (Perhaps with non-strict inequalities, but we can make the constants worse to make the inequalities strict.) These imply that the metric has an extension to $\bar M$.
\end{proof}

Here, we use the control that we have on $g_{wp}(t)$ to find more precise estimates on the convergence of $g_{wp}(t)$ to $g_{mp}$ as $t \searrow 0$.
\begin{corollary}\label{improved_convergence}
  Suppose we are in the setting of Theorem \ref{theorem:model_pinch_flow}. Write $g_{mp}$ as $g_{mp} = dx^2 + u_0(x)g_{S^q} + w_0(x) g_{F}$, set $v_0(x) = u_0^{-1}(x)|\nabla u_0(x)|^2$, and set $h(t) = g_{wp}(t) - \left( g_{mp} - t\ein g_{S^q} - t\ein_F g_F\right)$.  There are constants $\epsilon_0 = \epsilon_0(g_{mp})$ and $C_0 = C_0(g_{mp})$ such that the following holds.  For any $x$ with $u_0(x)<u_*/2$, and for $t < \epsilon_0u(x)$, we have $|\partial_t h|_{g_{mp}}(x, t) \leq C_0\left( u_0(x)^{-1}v_0(x) \right)$.

    Let $k \in \nats$.  If in addition to the regularity assumption \ref{modelpinch_reg} we assume
    \begin{align}
      \frac{|V_0|_{2+k, \eta; u/2, (du)^2}}{V_0}
      +
      \frac{|W_0|_{2+k, \eta; u/2, (du)^2}}{W_0}
      \leq C
    \end{align}
    then there are constants $\epsilon_k(g_{mp})$ and $C_k(g_{mp})$ such that for any $x$ with $u(x) < u_*/2$, and for $t < \epsilon_ku_0(x)$, we have
    \begin{align}
      |\partial_t \left(\nabla^{g_{mp}}\right)^{k}h|_{g_{mp}}(x, t) \leq C_k \left( u_0^{-1}v_0 \right)^{2+k}
    \end{align}
\end{corollary}
\begin{proof}
  We first consider $k = 0$.  Consider any $x$.  Initially $x$ is in the productish region.   As long as $x$ is in the productish region, we can use Corollary \ref{prish_curvature_control} to control $\Rc_{g_{wp}(t)} - \left( \ein g_{S^q} + \ein_F g_F\right)$.  In particular, for sufficiently small $\epsilon_0$ we will have $u_{g_{wp}}(x, t) > \oh u_0(x)$ for $t < \epsilon_0u_0$, so we stay in the productish region for these times.

  To get the higher regularity, note that the extra assumption allows us to improve the degree at which we are allowed to apply all interior Schauder estimates, so we get control on the gradients of $\Rm$ in Corollary \ref{prish_curvature_control} as well.
\end{proof}

 
\section{Stability of the Bryant soliton}\label{stabil_bryant}
In this section, we will prove a result that we use for stability of Ricci-DeTurck flow around the Bryant soliton, Theorem \ref{bry_stabil}, assuming a priori control at infinity.  This will be used to prove the short-time stability of flows from model pinches.  For a complete stability result for the Bryant soliton, see  \cite{steady_alix}.  That result does not suffice for us because being in the weighted $L^2$ space there requires exponential decay at infinity.

The main result we use from this section is the following.  We let $(\Bry, g_{bry})$ be the Bryant steady soliton metric on $\Bry \sim \real^{1 + q}$ which has soliton vector field $X$, and $(\Bry, g_{bry}) \times (\real^p, g_{\real_p}) = (\Bry \times \real^p, g_{sol})$ .
\begin{theorem}\label{bry_stabil}
  For $C_{reg} \geq 0$ there is a constant $\barr \epsilon(C_{reg}) > 0$ depending only on the dimension with the following property.  Suppose $\epsilon < \barr \epsilon$ and let $\bar F = \epsilon F$, where $F$ is defined in Lemma \ref{bry_rdt_supsoln}.  Suppose that $g(t) = g_{sol} + h(t)$ is a Ricci-DeTurck flow around $g_{sol}$ modified by $X$, on a time interval $I$.  Suppose that for all $P \in \Bry \times \real^p$ and $t \in I$,
  \begin{align}
    |h(P,t)| \leq \bar F(P).
  \end{align}

  Suppose that either $I = (-\infty, T]$, or $I = [0, T]$ with the condition at time $t=0$ that
  \begin{align}\label{init_nabla_bnds}
    u^{0,1/2}|\nabla h| + u^{0,1}|\nabla^2 h| < C_{reg} \bar F.
  \end{align}

  Then the \emph{strict} inequality $|h(P, t)| < \bar F$ holds for all $P \in \Bry$ and $t \in I$.
\end{theorem}

We wish to compare this result to Section 7 of \cite{uniquenessBK}.  Note that the function $F$ here is asymptotic to $1/u$, as is the scalar curvature $R$ on the Bryant soliton.  So, our hypothesis implies $|h| \leq \epsilon R$ (for some possibly smaller $\epsilon$).  In their Section 7, Bamler and Kleiner find estimates for the perturbation under the assumption that $|h| \leq C R^{1 + \chi}$ for some $\chi > 0$.  Indeed, they get an improvement of $\frac{|h|}{R^{1+\chi}}$, where we have only been able to get stability.  Both Theorem \ref{bry_stabil} and Section 7 of \cite{uniquenessBK} use the Anderson-Chow estimate; unfortunately our Theorem requires some specific calculations on the Bryant soliton and it's not clear what generalization is available.

\subsection{Anderson-Chow Estimate}
We begin with a version of the Anderson-Chow estimate.  In \cite{ac_estimate}, Anderson and Chow proved an estimate in three dimensions for solutions to the linearization of Ricci-DeTurck flow, in terms of the scalar curvature.  The key inequality for their estimate is 
\begin{align}\label{ac_inequality}
  |\Rc|^2 - R \rmplus \geq 0
\end{align}
valid on any three-dimensional manifold. Recall the definition
$$\rmplus = \max_{h \in Sym_2(M) : |h| = 1}\ip{\Rm[h]}{h}$$ from Section \ref{section:rdt}.  This estimate is useful in classifying solitons \cite{brendlesolitonimportant}, and was also vital in \cite{uniquenessBK}. In \cite{ac_higherdim}, Wu and Chen prove a higher-dimensional version of the Anderson-Chow estimate, assuming that the Weyl tensor vanishes identically along the flow (Claim 2.1 in \cite{ac_higherdim}).  For a singly-warped product, the Weyl tensor does vanish identically (since it is conformal to a cylinder) and therefore \cite{ac_higherdim} applies.  We also give a proof in the restricted setting we need, because it is more elementary and we need a statement about strictness.

For a singly warped product, $ds^2 + u(s) g_{S^q}$, we let $L = (1-\partial_su^{1/2})/u$ be the sectional curvature of a plane tangent to $S^q$, and $K = - u^{-1/2}\partial_s^2 u^{1/2}$ be the sectional curvature of a plane spanned by $\partial_s$ and a vector from $S^q$.
\begin{lemma}
  Let $g = ds^2 + u g_{S^q}$ be a warped product metric with nonnegative sectional curvature.  Then the Anderson-Chow inequality \eqref{ac_inequality} holds for $g$.  Equality is achieved only at points where the sectional curvature is constant or where either $K$ or $L$ is $0$.
\end{lemma}

\begin{proof}
  Note the calculation below is just done within the vector space $T_PM$ for an arbitrary $P \in M$.
  The scalar curvature of $g$ is
  $
    R = 2qK + q (q-1)L.
    $.
  The Ricci curvature of $g$ is
  $
    \Rc = qK ds^2 + (K + (q-1)L)(u g_{S^q})
    $
  so
  $
    |\Rc|^2 = q^2 K^2 + q(K + (q-1)L)^2
    $.
  Writing $\alpha = \frac{K}{(q-1)L}$ we can rewrite these as
  \begin{align}
    \frac{R}{(q-1)L}
    = q(2 \alpha + 1),
    \quad
    \frac{|\Rc|^2}{((q-1)L)^2} = q\left( q \alpha^2 + (\alpha + 1)^2 \right).
  \end{align}
  We can deal with $L = 0$ by taking the limit as $\alpha \to \infty$ in the end.
  
  Now let's find $h$ with $|h| = 1$ which maximizes $\Rm[h,h]$.  Take an orthonormal basis $V_0 = \partial_s, V_1, \dots, V_q$ for $T_pM$, such that $h$ is diagonal with respect to $V_1 \dots, V_q$, that is for $i$, $j$ nonzero and distinct,
  $h_{ii} = \lambda_i$ and $h_{ij} = 0$.
  Then,
  \begin{align}
    \Rm[h,h]
    &= \Rm_{aibj}h^{ij}h^{ab} \\
    &= \sum_{a = 1}^n h^{00}h^{aa}\Rm_{a0a0}
      + \sum_{i=1}^n h^{ii}h^{00}\Rm_{0i0i}
      + \sum_{i=1}^n \sum_{a=1}^n h^{ii}h^{aa}\Rm_{aiai} \\
    &+ \sum_{j=1}^n h^{0j}h^{j0}\Rm_{j00j}
    + \sum_{i=1}^n h^{i0}h^{0i}\Rm_{0ii0}
      + \sum_{i=1}^n \sum_{j=1, j \neq i}^n h^{ij}h^{ji} \Rm_{ijji}.
  \end{align}
  The first line is the case when $i = j$: the first term is when $i = 0$, the second term is when $a = 0$, and the third term is when neither is 0.  The second line is when $i \neq j$: the first term is when $i = 0$, the second term is when $i \neq 0$ but $j = 0$, and the third term is when neither is 0.  Note that actually this last term vanishes since $h^{ij} = 0$, and since $\Rm_{0ii0} = \Rm_{j00j} = -L$, the second line is negative.  Therefore to optimize $h$ we will take $h^{0i} = 0$. Let $b = h_{00}$.  Simplifying, we have
  \begin{align}
    \Rm[h, h]
    &= 2 b (\sum \lambda_i) K 
      + \left( \left( \sum \lambda_i \right)^2 - \sum \lambda_i^2 \right) L
  \end{align}
  We can assume $b > 0$, since negating $h$ does not change $\Rm[h,h]$.  Then, to maximize either
  \begin{align}
    \sum \lambda_i \quad \text{ or } \quad \left( \left( \sum \lambda_i \right)^2 - \sum \lambda_i^2 \right)
  \end{align}
  we would take the $\lambda_i$ all equal.  Since this maximizes either term, and since $K$ and $L$ are positive, it  maximizes all of $\Rm[h,h]$.  Define $\lambda = \sqrt{q}\lambda_i$, with the motivation that $\lambda$ is the norm of the restriction of $h$ to $TS^q$, so $b^2 + \lambda^2 = 1$.
  So, recalling the definition $\alpha = \frac{K}{(q-1)L}$ we arrive at
  \begin{align}
    \frac{\Rm[h,h]}{(q-1)L}
    = 2  \sqrt{q} \alpha (b \lambda) + (\lambda^2).
  \end{align}
  The positive eigenvalue of the matrix
  $\begin{pmatrix}0 & \sqrt{q}\alpha \\ \sqrt{q}\alpha & 1 \end{pmatrix}$ is $\oh (1 + \sqrt{4 q\alpha^2 + 1})$.  Therefore, since $b$ and $\lambda$ optimize $2 \sqrt{q}\alpha b \lambda + \lambda^2$ with $b^2 + \lambda^2 = 1$, we have, 
  \begin{align}
    \frac{\rmplus}{(q-1)L} = \oh (1 + \sqrt{4q \alpha^2 + 1})
  \end{align}

  Therefore,
  \begin{align}
    A = 
    \frac{|\Rc|^2 - R \Rm[h,h]}{q (q-1)^2L^2}
    &= q \alpha^2 + (\alpha + 1)^2 - \oh (2 \alpha + 1)(1 + \sqrt{4 q \alpha^2 + 1})
  \end{align}
  Now, for $q = 2$ and for each $\alpha$, we have $A\geq 0$ by the three dimensional Anderson-Chow estimate. (We could also check by hand.)  We claim $A$ doesn't decrease as we increase $q$.
  Calculate,
  \begin{align}
    \frac{dA}{dq}
    &= \alpha^2 - (2 \alpha + 1)(4 q \alpha^2 + 1)^{-1/2}\alpha^2 \\
  \end{align}
  So for $q \geq 2$
  \begin{align}
    \frac{dA}{dq} 
    &\geq \alpha^2 \left( 1 - (2 \alpha + 1)(8 \alpha^2 + 1)^{-1/2} \right) \geq 0.
  \end{align}
  
\end{proof}

\begin{corollary}
  The Anderson-Chow inequality \eqref{ac_inequality} holds for $g_{sol}$.
\end{corollary}
\begin{proof}
  The extra flat factor does not affect any of the terms in \eqref{ac_inequality}.  The Bryant soliton has nonnegative curvature, so the previous lemma applies.
\end{proof}

\subsection{Constructing a supersolution}
Using the Anderson-Chow estimate, we construct a supersolution to linearized Ricci-DeTurck flow around $g_{sol}$.  We write $g_{Bry} = ds^2 + u(s) g_{S^q}$.
\begin{lemma}\label{bry_rdt_supsoln}
  Let $(\Bry \times \real^p, g_{sol}, X)$ be the Bryant soliton crossed with a euclidean factor.  There is a function $F : \Bry \times \real^p  \to \real_{>0}$, which is just a function of $u$, with the following properties.
  \begin{enumerate}
  \item For some $c > 0$, $\lap_X F + 2 \rmplus F \leq - c u^{0,-2}\log( 2 + u )F$.
  \item For some $c_1, c_2 > 0$, as $u \to \infty$, $F = c_1u^{-1} \left( 1 + c_{2} \frac{\log u}{u} \right)(1 + o(1)) $.
  \end{enumerate}
\end{lemma}

\begin{proof}
  First recall that if $R_0$ is the maximum scalar curvature, $f$ is the soliton potential, and $\bar f(p) = -\frac{f(p)-f(0)}{R_0}$ then $\bar f$ satisfies
  \begin{align}
    \label{eq:124}
    \lap_X \bar f = 1,
    \quad \bar f(0) = 0,
    \quad \nabla \bar f(0) = 0,
  \end{align}
  and has the asymptotics at $\infty$,
  \begin{align}
    \bar f = \ein^{-1} u \left( 1 - c_{\bar f} \frac{\log u}{u} \right) (1 + o(1; u \to \infty) )
  \end{align}
  for some constant $c_{\bar f}$ (see Appendix \ref{bryant_facts} and especially \eqref{bry_second_order_expansion}).  Also, $\bar f$ attains its minimum of $0$ at $u=0$.
  
  Now let $F_1 = \left( \bar f + a \right)^{-1}$ for some $a>0$ to be determined.  Calculate using $\lap_X \bar f = 1$ that
  $
    \lap_X F_1
    = - \left( F_1 - 2 F_1^2 |\nabla \bar f|^2 \right)F_1
    $,
  so
  \begin{align}
    \label{eq:151}
    -\left( \lap_X + 2 \Lambda_{\Rm}  \right) F_1
    &=  \left( F_1 - 2 \Lambda_{\Rm} - 2 F_1^2 |\nabla \bar f|^2 \right) F_1
  \end{align}

  We claim that for large enough $B$, the function $F = F_1 + B R$ satisfies the properties in the lemma.  The asymptotics at infinity (i.e. item (2) of the conclusion) are immediate from the asymptotics for $F_1$ and $R = c_1 u^{-1} + O( u^{-2})$.  Now calculate,
  \begin{align}
    \label{eq:156}
    -(\lap_X + 2\rmplus) F
    &= \left( \frac{-(\lap_X + 2\rmplus) F_1}{F_1 + B R} + B\frac{ -(\lap_X +2 \rmplus) R}{F_1 + B R} \right) F \\
    &\eqdef \left( T_1 + T_2 \right) F
  \end{align}
  Note the term $T_2$ is positive everywhere by the singly-warped Anderson-Chow estimate and the equation satisfied by $R$ under Ricci flow:
  \begin{align}
    \lap_X R + 2 \Lambda_{\Rm} R
    &= - 2 |\Rc|^2 + 2 \Lambda_{\Rm} R \leq 0.
  \end{align}

  \begin{claim}
    Let $K$ be a compact subset of $\Bry$ not containing the origin.  If $B$ is sufficiently large, then there is a $c$ so that $T_1 + T_2 > c$ on $K \times \real^p$.
  \end{claim}
  \begin{claimproof}
    On $K \times \real^p$, the singly-warped Anderson-Chow estimate is not sharp and $R$ is bounded from above, so for some $c_K>0$, $|\Rc|^2 - R \rmplus > c_KR$ in $K \times \real^p$.  Therefore,
    \begin{align}
     -\left( \lap_X + 2 \rmplus \right)R  \geq c_K R \quad \text{in} \quad K \times \real^p
    \end{align}
    By compactness, in $K$, $-(\lap_X + 2 \rmplus) F_1$ is bounded from below and $R$ is strictly positive.  Therefore, examining the dependence of $T_1$ and $T_2$ on $B$, we can chose $B$ large enough so that $T_1 \geq -c_K/4$ and $T_2 \geq c_K/2$ on $K$.
  \end{claimproof}

  \begin{claim}
    For sufficiently small $a$ in the definition of $F_1$ (independent of $B$), and sufficiently small $u_1$ (independent of $B$), there is a $c$ (which may depend on $B$) such that $T_1>c$ in $\{u < u_1\}$.
  \end{claim}
  \begin{claimproof}
    Choose $a = \frac{1}{4 \Lambda_{\Rm}(0)}$.  Then $F_1 - 2 \Lambda_{\Rm} > 0$ in a neighborhood of $0$.  Also, $|\nabla \bar f|^2(0) = 0$. The claim follows from \eqref{eq:151} by choosing $u_1$ and $c$ small enough.
  \end{claimproof}

  \begin{claim}
    For sufficiently large $u_2$ (depending on $a$, but independent of $B$) and sufficiently small $c$ (depending on $B$ and $a$), $T_1$ satisfies $T_1 \geq c u^{0,-2}\log(2 + u)$ on the set $\{u > u_2\}$.
  \end{claim}
  \begin{claimproof}
    The Bryant soliton satisfies, as $u \to \infty$,
    \begin{align}
      \label{eq:152}
      \Rm = u^{-1} \left( u g_{S^q} \odot u g_{S^q} \right) + O(u^{-2}|\nabla u|^2) = u^{-1} \left( u g_{S^q} \odot u g_{S^q} \right) + O(u^{-2}) 
    \end{align}
    Note that the largest eigenvalue of $u^{-1}\left( u g_{S^q} \odot u g_{S^q} \right)$ is $(q-1) = \oh \ein$.
    We can calculate the asymptotics of $F_1$ from the asymptotics of $\bar f$ from \eqref{bry_second_order_expansion} in Section \ref{bryant_facts_nextorder}:
    \begin{align}
      \label{eq:153}
      F_1 = \ein u^{-1} \left( 1 + c_{\bar f} \frac{\log u}{u} \right) ( 1 + o(1; u \to \infty)).
    \end{align}
    Also, $|\nabla \bar f|^2 = O(1; u \to \infty)$.  From this we find,
    \begin{align}
      \label{eq:154}
      (F_1 - 2 \rmplus - 2 F_1^2 |\nabla \bar f|^2)
      &= \ein c_{\bar f} u^{-2} \log u + O(u^{-2}; u \to \infty)
    \end{align}
    The claim follows by choosing $u_2$ large enough and $c$ small enough.
  \end{claimproof}

  To prove the lemma, choose $u_1$ and $u_2$ in accordance with the second and third claims above, and then choose $B$ large enough so the conclusion of the first claim holds on the complement of $\{u_1 < u < u_2\}$.  Then the conclusion of the lemma holds (taking the minimum over the values of $c$).
\end{proof}
\subsection{Completion of the proof of Theorem \ref{bry_stabil}}
\begin{proof}
  In this proof the ever-increasing constant $C$ is chosen independently of $\epsilon$.  First, we write the inequality solved by $\bar F$ in terms of the laplacian $\lap_{X,g_{\Bry}, g}$.  By Lemma \ref{bry_rdt_supsoln}, we have
  \begin{align}\label{barF_ineq}
    - \left( \lap_X \bar F + 2 \rmplus \bar F \right)
    \geq c u^{0,-2}\log(2 + u)\bar F.
  \end{align}
  Since $|\nabla\nabla F| \leq C u^{0,-3} \leq C u^{0,-2}F$, and $|h| \leq \epsilon F$, we have
  \begin{align}
    |\lap_{X, g_{\Bry}, g}F - \lap_X F| \leq C u^{0,-2}\epsilon F^2
    = C\epsilon u^{0,-3}F
  \end{align}
  and multiplying through by $\epsilon$, $|\lap_{X, g, \bar g} \bar F - \lap_X \bar F| \leq C \epsilon u^{0,-3}\bar F$.  Therefore, decreasing $c$ and demanding that $\epsilon$ is sufficiently small, we can replace \eqref{barF_ineq} with
  \begin{align}
    - \left( \lap_{X, g_{\Bry}, g} \bar F + 2 \rmplus \bar F \right)
    \geq c u^{0,-2}\log(2 + u)\bar F.\label{barF_ineq_better}
  \end{align}  

  Next we note the regularity available.  We claim that for some $C$ (independent of $\epsilon$, $P_*$, and $t_*$, but depending on $C_{reg}$), we have $|\nabla h|(P_*,t_*) < C u^{0,-1/2}\bar F(P_*)$.  To see this, let $a = u^{0,-1}(P_*)$ and scale the parabolic system by $a$:
  \begin{align}
    \tilde g_{\Bry} = a g, \quad \tilde h = a h, \quad \tilde t = a t, \quad \tilde X = a^{-1} X, \quad \tilde u = a u.
  \end{align}
  We want to apply regularity in a parabolic neighborhood of some sufficiently small size $r>0$.  The Bryant soliton has a bound $|\nabla u|^2 \leq C$ for some $C$.  So, for any $r$, for all
  $
    P \in B_{\tilde g}(P_*, r) = B_{g}(P_*, r/\sqrt{a})
  $
  we have
  \begin{align}
    |\tilde u(P) - \tilde u(P_*)|
    &= a |u(P)-u(P_*)| \\
    &\leq C a \frac{r}{\sqrt{a}} = C r \sqrt{a}
    = C r u^{0, -1/2} \leq C r\label{uchangebnd}.
  \end{align}
  Therefore for sufficiently small $r$, the ball of radius $r$ around $P_*$, with respect to $\tilde g$, is close to a euclidean ball, uniformly in $P_*$.
  
  To continue the regularity argument, in the case when $I = [0,T]$, the parabolic neighborhood of size $r$ around $P_*$ may see the initial condition.  We need to check what the bounds on the initial condition \eqref{init_nabla_bnds} says about $\tilde h$.  At the initial time,
  \begin{align}
    |\nabla \tilde h|_{\tilde g}(P,0)
    = a^{-1/2}|\nabla h|_{g}(P,0)
    \leq c_h\frac{u^{0,-1/2}(P)}{u^{0,-1/2}(P_*)} \bar F
    = c_h\frac{u^{0,1/2}(P_*)}{u^{0,1/2}(P)} \bar F
    \leq C \bar F
  \end{align}
  where we used \eqref{uchangebnd} and forced $r$ sufficiently small.
  Similarly scaling shows $|\nabla \nabla \tilde h|_{\tilde g}(P,0) \leq C \bar F$.

  Therefore, in the parabolic neighborhood of size $r$ we may apply parabolic regularity to find that $|\nabla \tilde h|_{\tilde g_{\Bry}}$ is bounded by $C\bar F$.  Scaling back, we find $|\nabla h|(P_*, t_*) = a^{1/2} |\nabla h|(P_*, t_*) \leq C u^{0,-1/2}(P_*, t_*) \bar F(P_*, t_*)$.
  
  Now, by the bound on the evolution of $|h|$ \eqref{rdt_norm}, we have that $Z = |h|$ satisfies
  \begin{align}
    \square_{X, g_{\Bry}, g} Z - 2\rmplus Z \leq C |\Rm_{g_{\Bry}}|Z^2 + C |\nabla h|^2.
  \end{align}
  Or, since we have assumed $Z \leq \bar F$, and also $|\nabla h| < Cu^{0, -\onf2}\bar F$ by the discussion on regularity,
  \begin{align}
    \square_{X, g_{\Bry}, g} Z - 2\rmplus Z
    &\leq C u^{0,-1}Z^2 + C u^{0,-1}\bar F^{2}.
  \end{align}
  Then since $Z \leq \bar F \leq \epsilon C u^{0,-1}$,
  \begin{align}
    \square_{X,g_{\Bry}, g} Z - 2\rmplus Z&\leq \epsilon C  u^{0, -2} \bar F.
  \end{align}
  In particular, we can choose $\epsilon$ sufficiently small so that
  \begin{align}
    \square_{X,g_{\Bry}, g} Z - 2\rmplus Z&\leq (c/2) u^{0, -2} \bar F.
  \end{align}
  where $c$ is the constant from \eqref{barF_ineq_better}.

  Therefore
  \begin{align}
    \square_{X, g_{\Bry}, g} (\bar F-Z) - 2 \rmplus (\bar F - Z) \geq (c/2) u^{0,-2} \bar F > 0
  \end{align}
  and the lemma follows by the maximum principle.
\end{proof}


\section{Local stability of forward evolutions}\label{section:local_stability}
Let $g_{mp}$ be an $\Rm$-permissible model pinch and let $(\bar M, g_{wp}(t))$ be a forward evolution from $g_{mp}$ given by Theorem \ref{theorem:model_pinch_flow}.  In this section we prove local stability of $(\bar M, g_{wp}(t))$ assuming a priori control at the boundary of a neighborhood of the origin.    We let $u_0 = u(p, t)$. For any $u_1, u_2$ we let $\bar M_{[u_1, u_2]} = \{p : u_0(p) \in [u_1, u_2]\}$.  Note that while $\{(p,t): u(p,t)\in [0, u_1]\}$ is a subset of space-time which is different for each time-slice, $\bar M_{[0, u_1]}$ is a fixed subset of $M$.  However, note that as in Lemma \ref{u_ctrl_in_x} we can argue that for any $u_1$ there is a $T(u_1)$ such that $\bar M_{[0, u_1/2]} \subset \{p : u(p, t) < u_1\} \subset \bar M_{[0, 2u_1]}$ for $t < T(u_1)$.

Let $\hat u = u + \ein t$ and $Q = \hat u / u$, as in Appendix \ref{nearly_cnst_pde_sect}.  Also let $F$ be the function defined in Lemma \ref{bry_rdt_supsoln}.  For parameters $u_*$,  $\sigma_*$, $\sigma_{**}$, $D$, and $b$ to be chosen, we let $F_{full}$ be defined as
\begin{align}
  F_{full} =
  \begin{cases}
    (1 + DV)^{1/2}Q V_0(\hat u) & \sigma > \sigma_{**}  \\
    \min\left( (1 + DV)^{1/2}Q V_0(\hat u), bF(\sigma) \right)
    & \sigma_* \leq \sigma < \sigma_{**} \\
    bF(\sigma) & \sigma \leq \sigma_*
  \end{cases}
\end{align}


\begin{theorem}\label{local_stability}
  For $r_0 < \barr r_0(g_{mp})$, $\epsilon < \barr \epsilon(g_{mp}, r_0)$,  $D > \berr D(g_{mp})$, $u_* < \bar u_*(g_{mp}, D)$, and $\sigma_* = \sigma_*(g_{mp}, D)$, $\sigma_{**} = \sigma_{**}(g_{mp}, D)$, $b=b(g_{mp}, D)$, there is a  $T_*(g_{mp}, u_*, r_0)$ and $C(g_{mp}, u_*, r_0)$ with the following property.
  Let $0 < T_1 < T_2 < T_*$ be given and suppose that $g(t)$ is a a solution to Ricci-DeTurck flow with background metric $g_{wp}(t)$ on $\bar M_{[0, 2 u_*]}$ for times $[T_1, T_2]$.
  Suppose for all $p \in \bar M_{[0, 2u_*]}$,
  \begin{align}
    |g(T_1) - g_{bg}|_{2, \eta; r_0|\Rm|^{-1/2}}(p, T_1) \leq \epsilon F_{full}(p, T_1),
  \end{align}
  and for all $(p,t) \in \bar M_{[u_*, 2u_*]} \times [T_1, T_2]$,
  \begin{align}
    |g(t) - g_{bg}|_{2, \eta; r_0|\Rm|^{-1/2}}(p, t) \leq \epsilon F_{full}(p, t).
  \end{align}
  Then for all $(p, t) \in \bar M_{[0,u_*]} \times [T_1, T_2]$,
  \begin{align}
    |g(t) - g_{bg}|_{2, \eta; r_0|\Rm|^{-1/2}}(p, t) \leq C\epsilon F_{full}(p, t).
  \end{align}

\end{theorem}

\begin{proof}
The $C^{0}$ estimate follows from Lemmas \ref{lemma:asymmetric_prish_barrier}, \ref{lemma:asymmetric_tip_barrier}, and \ref{lemma:gluing_asymmetric}.  The $C^{2, \eta}$ estimate follows from interior Schauder estimates, once we have the $C^{0}$ estimate.
\end{proof}

\subsection{Control in the productish region}\label{section:asym_productish}
In this section we control the Ricci DeTurck flow in the productish region of the warped-product solution $g_{wp}(t)$.  This uses the general sub- and supersolutions from Appendix \ref{nearly_cnst_pde_sect}.  

Recall the definition
$\rmplus = \max_{h \in Sym_2(M) : |h| = 1}\ip{\Rm[h]}{h}$ from
Section \ref{section:rdt}.  Here, and in this section, by default we are taking all inner products, curvatures, and covariant derivatives with respect to $g_{wp}(t)$.  First we write down a bound for $\rmplus$ on our warped product solution $g_{wp}(t)$.  Remember we are assuming that $g_{mp} = g_{wp}(0)$ is an $\Rm$-permissible model pinch.
\begin{lemma}\label{rmplus_prish}
  There is a constant $C(g_{mp})$ such that in the productish region $g_{wp}(t)$ satisfies
$
  \rmplus \leq (q-1) u^{-1} + C u^{-1}v
  = \oh \ein u^{-1} + C u^{-1}v.
$
\end{lemma}
\begin{proof}
  Note the use of the metric in $\Rm[h]$ to contract tensors. Let us write
  \begin{align}
    \left(\Rm_{g_1}[h; g_2]\right)_{ef}
    = (g_2)^{ab}(g_2)^{cd}\left(\Rm_{g_1}\right)_{acef}h_{bd}
  \end{align}
  so that $\Rm[h] = \Rm_{g_{wp}}[h; g_{wp}]$.  Now we can compute some scaling for the components of $\Rm_{g_{wp}}$ given by Corollary \ref{prish_curvature_control}.
  \begin{align}
    \left(u\Rm_{g_{S^q}} \right)[h; g_{wp}]
    = \left(u\Rm_{g_{S^q}} \right)[h; u g_{S^q}] 
    &= u^{-1} \left( \Rm_{g_{S^q}} \right)[h; g_{S^q}].
  \end{align}
  Therefore
  \begin{align}
    \max_{|h|_{g_{wp}} = 1}
    \ip{(u \Rm_{g_{S^q}})[h; g_{wp}]}{h}_{g_{wp}} 
    &= u^{-1} \max_{|h|_{g_{S^q}} = 1}\ip{\Rm_{g_{S^q}}[h; g_{S^q}]}{h}_{g_{S^q}} \\
    &\defeq u^{-1}\Lambda_{S^q} = u^{-1}(q-1). \label{lambdacalc0}
  \end{align}
  Similarly,
  \begin{align}
    \max_{|h|_{g_{wp}} = 1}
    \ip{(w \Rm_{g_{F}})[h; g_{wp}]}{h}_{g_{wp}} 
    &= w^{-1}\Lambda_{F}.\label{lambdacalc1}
  \end{align}
  Let $k = \max\left( \frac{\Lambda_F}{\Lambda_{S^q}}, (1+c)\frac{\ein_F}{\ein}\right)$.  By the $\Rm$-permissible assumption, and the assumption \ref{w_big} of all model pinches, $W_0(u + \ein t) \geq k \cdot (u + \ein t)$.  Now, we use that $w$ is barricaded in the productish region, i.e. it satisfies \ref{conc:prish_barrier} of Definition \ref{productish_barricaded}:
  \begin{align}
    \frac{w}{u}
    &\geq \frac{W_0(u + \ein t)-\ein_F t}{u} - \frac{DVW_0(u + \ein t)}{u} \\
    &\geq \frac{k\cdot (u + \ein t) - \ein_F t}{u} - \frac{kDV\cdot (u + \ein t)}{u}.
  \end{align}
  Next, first using $k \geq (1+c) \frac{\ein_F}{\ein}$ and then $k \geq \frac{\Lambda_F}{\Lambda_{S^q}}$,
  \begin{align}
    \frac{w}{u}
    &\geq k + \frac{kDV\cdot (u + \ein t)}{u} \geq \frac{\Lambda_F}{\Lambda_{S^q}} - C u^{-1}V
  \end{align}
  for some $C$ depending only on $g_{wp}$ and the parameter $D$.  Therefore, coming back to \eqref{lambdacalc1}, and increasing $C$,
  \begin{align}
    \max_{|h|_{g_{wp}} = 1}
    \ip{(w \Rm_{g_{F}})[h; g_{wp}]}{h}_{g_{wp}} 
    &\leq u^{-1}\Lambda_{S^{q}} + Cu^{-1}V.\label{lambdacalc2}
  \end{align}

  Now we put together \eqref{lambdacalc0} and \eqref{lambdacalc2}.  Since $\Rm_{g_F}$ and $\Rm_{g_{S^q}}$ act only on the orthogonal components $Sym_2(TF) \subset Sym_2(TM)$ and $Sym_2(TS^q) \subset Sym_2(TM)$ respectively, we can take the maximum of the two pieces to find
  \begin{align}
    \max_{|h|_{g_{wp}} = 1}
    \ip{(u \Rm_{g_S^q} + w \Rm_{g_F})[h; g_{wp}]}{h}_{g_{wp}}
    \leq u^{-1}\Lambda_{S^{q}} + Cu^{-1}V.\label{lambdacalc3}
  \end{align}

  Finally, adding in $\Rm_{warp}$ can only change this result by something proportional to its norm.  So, increasing $C$,
  \begin{align}
    \max_{|h|_{g_{wp}} = 1}
    \ip{\Rm_{g_{wp}}[h; g_{wp}]}{h}_{g_{wp}}
    &=
    \max_{|h|_{g_{wp}} = 1}
    \ip{(u \Rm_{g_S^q} + w \Rm_{g_F} + \Rm_{warp})[h; g_{wp}]}{h}_{g_{wp}}\\
    &\leq u^{-1}\Lambda_{S^{q}} + Cu^{-1}V.\label{lambdacalc3}
  \end{align}

\end{proof}

Consider a solution to Ricci-DeTurck flow around $g_{wp}(t)$ given by $g(t) = g_{wp}(t) + h(t)$.  Let $y = |h(t)|^2$.  By the equation for the evolution of the norm of the perturbation \eqref{dtevo_square}, in the productish region $y$ satisfies 
$
  \square_{g_{wp}, g} y \leq u^{-1}\left(2\ein + Cv + Cy^{1/2}\right) y
$,
or, just rewriting,
\begin{align}
  (\square_{g_{wp}, g} - 2 \mu u^{-1})y\leq C u^{-1}\left(v + y^{1/2}\right) y.  \label{eq:54}
\end{align}
We now use the supersolutions found in Appendix \ref{nearly_cnst_pde_sect} to control $y$ in the productish region.  The parameters $D$, $u_*$, $\sigma_*$ here are not the same as the parameters of control of $g_{wp}(t)$.  Recall $\hat u = u + \ein t$ and $Q = u^{-1}\hat u$.

\begin{lemma}\label{lemma:asymmetric_prish_barrier}
  Suppose $D > \berr D(g_{wp})$, $0<\epsilon<1$, $u_{*} < \barr u_{*}(g_{wp}, D), \sigma_* < \barr \sigma_*(g_{wp}, D)$, and $T_* < \barr T_*(g_{wp}, D)$.  Suppose $0 < T_1 < T_2 < T_*$.
  Set
  \begin{align}
    \Omega_{prish, [T_1, T_2]} =
    \left\{(p, t) : u < u_*, \sigma = \frac{u}{t \nu(t)} > \sigma_*, t \in [T_1, T_2]\right\}.
  \end{align}

  Suppose $g(t) = g_{wp}(t) + h(t)$ solves Ricci-DeTurck flow around $g_{wp}(t)$.  Let $y = |h|$, $Y^+_{prish} = (1 + D V)Q\left(V_0 \circ \hat u \right) = (1+DV)V$, and $\bar Y^+_{prish} = \epsilon Y^+$.  If $y < \bar Y^+$ on the parabolic boundary of $\Omega_{prish, [T_1, T_2]}$, then $y < \bar Y^+$ in $\Omega_{prish, [T_1, T_2]}$.  
\end{lemma}

\begin{proof}
  In the end, we will choose $\barr u_*$, $\berr \sigma_*$, and $\barr T_*$ to ensure that $DV < 1$ and $\bar Y^+$ is smaller than $\oh$ in the region under consideration.  (We may do this since we can make $V$ arbitrarily small by Lemma \ref{lemma:V_small}.)  Therefore equation \eqref{rdt_norm} is valid.
  
  By Lemma \ref{sup_solns} we have that, for some $c > 0$,
  \begin{align}
    \label{eq:112}
    \left(\square_{g_{wp}}  -  2 \mu u^{-1}\right)\bar Y^{+} \geq (c D) u^{-1}v \bar Y^{+}.
  \end{align}
  
Since $(\bar Y^+)^{1/2} \leq CV$, we find by decreasing $c$,
\begin{align}
  \label{eq:120}
  \left(\square_{g_{wp}} - 2 \mu u^{-1}\right)  \bar Y^+ \geq (c D) u^{-1} \left( v + (\bar Y^+)^{1/2} \right) \bar Y^+.
\end{align}
  We can change the $\square_{g_{wp}}$ to $\square_{g_{wp}, g}$.  As long as $y < \bar Y^{+}$ we have
\begin{align}
  |\square_{g_{wp}}\bar Y^{+} - \square_{g_{wp},g}\bar Y^{+}|
  \leq C (\bar Y^+)^{1/2}|\nabla \nabla \bar Y^{+}|
  \leq Cu^{-1}v(\bar Y^+)^{3/2}.
\end{align}
In the second inequality we use Lemma \ref{dd1z_bnd},  and the bound $|\nabla \nabla u| < Cv$ (see the calculation in Lemma \ref{lemma:ycontrol_productish}).  Again decreasing $c$, we have
\begin{align}
  \left( \square_{g_{wp}, g} -2 \mu u^{-1}\right) \bar Y^{+} \geq
  (cD) u^{-1} \left( v +  (\bar Y^+)^{1/2} \right)  \bar Y^+ \label{Yplus_pdi}
\end{align}
The lemma follows from the maximum principle by comparing \eqref{Yplus_pdi} to the evolution for $y$ \eqref{eq:54} and choosing $D$ large enough.
\end{proof}

\subsection{Control in the tip region}\label{assym_tip_section}
Recall the rescaled coordinates, $\alpha(t) = t \nu(t)$, $\partial_{\theta} = \alpha \partial_t$, and  $\sigma = u/\alpha$.

\begin{lemma}\label{lemma:asymmetric_tip_barrier}
  Let $u_*$, $\sigma_*$, $\zeta_*$, $r_0$, and $\epsilon < \barr \epsilon(g_{wp}, r_0)$ be given.   Let $F$ be the function from Lemma \ref{bry_rdt_supsoln} and let $\bar F = \epsilon F$.  There is a $T_*(u_*, \sigma_*, g_{wp})$ such that we have the following. 

  Suppose $g(t) = g_{wp}(t) + h(t)$ is a solution to Ricci-DeTurck flow with background metric $g_{wp}(t)$ in $\bar M_{[0, u_*]}$, on a time interval $[T_1, T_2]$, and $T_2 < T_*$.  Suppose
  \begin{align}
    \twoeta{h}{r_0|\Rm|^{-1/2}}{g_{wp}} < \bar F \quad \text{for} \quad t = T_1 \text{ and } \sigma < \nu^{-1/2}(T_1)\zeta_*,
  \end{align}
  \begin{align}
    |h|_{g_{wp}} < \bar F \quad \text{for} \quad t \in [T_1, T_2]\text{ and }\sigma \in [\sigma_*, \nu^{-1/2}\zeta_*].
  \end{align}

  Then $|h|_{g_{wp}} \leq \bar F$ for $\sigma < \sigma_*$ and $t \in [T_1, T_2]$.
\end{lemma}

\begin{proof}
  We will choose $\barr \epsilon$ sufficiently small in the end.  We use a contradiction-compactness argument to move the situation to Ricci-DeTurck flow around the Bryant soliton crossed with a euclidean factor.

  For contradiction, assume that there is no such $T_*$.  This means that there is a sequence of counterexamples: there are solutions $g^{(i)} = g_{wp} + h^{(i)}$ to the Ricci-DeTurk flow around $g_{wp}$, defined on intervals $[T_1^{(i)}, T_2^{(i)}]$, satisfying the conditions of the Lemma, but $|h^{(i)}(p^{(i)}, T_2^{(i)})| = \bar F(\sigma(p^{(i)}, T_2^{(i)}))$ for some sequence $p^{(i)}$ with $\sigma(p^{(i)}, T_2^{(i)}) \leq \sigma_*$ and $T_2^{(i)} \searrow 0$.  Let $\sigma^{(i)} = \sigma(p^{(i)}, T_2^{(i)})$.  We may pass to subsequence so that the $\sigma^{(i)}$ converge to some $\sigma^{(\infty)} \leq \sigma_*$.  

  Let $\alpha^{(i)} = \alpha(T_1^{(i)})$.  We claim that there is a $T_{**}$ depending on $g_{wp}$ and $\sigma_*$ such that $T_2^{(i)} - T_1^{(i)} \geq \alpha^{(i)}T_{**}$.  Indeed for $t < T_2^{(i)}$, $|h^{(i)}|$ is bounded by $\bar F$.  Therefore we can apply the interior Schauder estimates at the scale $\alpha(T_1^{(i)})$, to get a bound on how far $|h^{(i)}|$ can move.  Here we use our bound $\twoeta{h^{(i)}(T_1)}{r_0|\Rm|^{-1/2}}{g_{wp}} < \bar F$, as well as our control on the geometry of $g_{wp}$ at scale $\alpha$.

  Now, this says that in terms of the rescaled time coordinate $\theta$, the time difference is bounded from below.  Indeed, since by definition $d\theta = \alpha^{-1}dt$:
  $$\theta(T_2^{(i)}) - \theta(T_1^{(i)})
  = \int_{T_1^{(i)}}^{T_2^{(i)}} \alpha^{-1} dt
  \geq \int_{T_1^{(i)}}^{T_1^{(i)}+\alpha^{(i)}T_{**}} \alpha^{-1}(t) dt
  \geq T_{**}
  \frac
  {\alpha\left(T_1^{(i)} \right)}
  {\alpha\left(T_1^{(i)} + \alpha(T_{1}^{(i)}) T_{**} \right)}.
  $$
  For the last inequality we just take the value of the decreasing integrand at the right endpoint and multiply by the length of the interval.
  Since $\frac{t|\partial_t\alpha|}{\alpha} + \frac{t^2|\partial_t^2\alpha|}{\alpha}$ is bounded (using assumption \ref{modelpinch_reg} of Definition \ref{definition:model_pinch}) and $\alpha(T_1^{(i)}) = o(T_1^{(i)})$ we can argue by using a Taylor expansion on the denominator that the right hand side is bounded from below by some $\Theta_* > 0$.  So, passing to a subsequence, the sequence $\theta(T_2^{(i)}) - \theta(T_1^{(i)})$ either converges to $\infty$ or converges to some $\Theta_1 > 0$.

Let $G_{wp}$ be the family of metrics $G_{wp} = \alpha^{-1}(\sigma^{-1})^*g_{wp}$ which is $g_{wp}$ modified by scaling by $\alpha^{-1}$ and pulling back by $\sigma$.  Also let $G^{(i)} = \alpha^{-1}(\sigma^{-1})^* g^{(i)}$ and $H^{(i)} = G^{(i)} - G_{wp} = \alpha^{-1}(\sigma^{-1})^* h^{(i)}$.

Now $G^{(i)}$ satisfies
\begin{align}
  \partial_{\theta} G^{(i)} = -2 \Rc[G^{(i)}] - \lie_{X + V[G^{(i)}, G_{wp}]}G^{(i)} - \beta G^{(i)},\label{rescaled_rcf}
\end{align}
for $\theta \in [\theta(T_1^{(i)}), \theta(T_2^{(i)})]$.  Here $X$ is the vector field $\partial_\theta \sigma$.

  Translate the $\theta$ intervals so that the times $\theta(T_2^{(i)})$ all land at time $0$.  By Corollary \ref{lemma:barricaded_convergence}, the background metrics $G_{wp}$ converge to the Bryant soliton crossed with $\real^{dim(F)}$, and the vector field $X$ converges to the soliton vector field.  Passing to a subsequence, the $H^{(i)}$ converge to a solution $H$ of Ricci-DeTurck flow around the Bryant soliton, modified by the Bryant soliton vector field $X$.  (Note that the term $\beta G$ in \eqref{rescaled_rcf} converges to zero.)  The time interval is either $\theta \in (-\infty, 0]$, or $\theta \in [-\Theta_1, 0]$.  In the second case, we can translate the regularity we assume at time $T_1$, and we find that the bounds in Theorem \ref{bry_stabil} are satisfied for some $C_{reg}$ (independent of $\epsilon$).  So, provided we take $\barr \epsilon$ small enough in this lemma to satisfy Theorem \ref{bry_stabil}, we can apply Theorem \ref{bry_stabil}.

  However, at time $0$ and at some point $p \in \Bry \times \real^{dim(F)}$ with $\sigma_{Bry}(p) = \sigma^{(\infty)}$, we will have $|H| = |\bar F|$.  This contradicts the strict inequality in the conclusion of Theorem \ref{bry_stabil}.
\end{proof}

\subsection{Buckling Barriers}
In this lemma, we show that the function
$
  Y^+ = (1 + D V)Q^2 (V_0 \circ \hat u)^2,
$
which we use as a barrier for $|h|^2$ in the productish region,
crosses the function $F^2$, which we use as a barrier in the tip region.  This shows that they ensure each others' boundary conditions.

The following Lemma deals with the unscaled functions $Y^+$ and $F^2$.  Of course, the inequalities \eqref{eq:cross_ineq1} and \eqref{eq:cross_ineq2} also hold for $\bar Y^+ = \epsilon^2 Y$ and $\bar F^2 = \epsilon^2 F^2$.

\begin{lemma}\label{lemma:gluing_asymmetric}
  Let the constant $D$, in the definition of $Y^+$ be given.  There are $\sigma_* > 0$, $\sigma_2 > 0$, $\zeta_* >0$, and $b \in \real_+$ such that we have the following inequalities.

  For $t < T_*$, at $\sigma = \sigma_*$, we have
  \begin{align}\label{eq:cross_ineq1}
     bF^2 < Y^+.
  \end{align}
  For $t < T_*$, and $\sigma \in [\sigma_2, \zeta_*\nu^{-1/2}]$, we have
  \begin{align}\label{eq:cross_ineq2}
    Y^+ <  bF^2.
  \end{align}
\end{lemma}
\begin{proof}
  Below $c_i$ are positive constants, and all asymptotics are as $\sigma \to \infty$ and $t \searrow 0$.  Recall the asymptotics of $F$ from Theorem \ref{bry_stabil}:
  \begin{align}
    F &= c_1 \sigma^{-1} - c_2 \sigma^{-2}\log\sigma + o(\sigma^{-2}\log \sigma),\\
    F^2 &=  \sigma^{-2}\left(c_3 - c_4 \sigma^{-1} \log \sigma + o(\sigma^{-1}\log \sigma) \right).
  \end{align}
  Recall the asymptotics of $V$ from \eqref{lemma:V_small}:
  \begin{align}
    V &= c_5 \sigma^{-1}\left(1 + O(\nu + \nu^2 \sigma) \right),\\
    Y^+
    &= (1 + DV)V^2 \\
    &=  \sigma^{-2}\left(c_6 + c_7D \sigma^{-1} + O(\nu+ \nu^2 \sigma) + O(D\sigma^{-2}) \right).
  \end{align}

  Letting $d = b c_3 - c_6$, we find,
  \begin{align}
    \sigma^2 (bF^2 - Y^+) =
    d
    -  \sigma^{-1}\left(c_4 \log \sigma + c_7 D \sigma^{-1} + o(\log \sigma) + O(D \sigma^{-2}) \right) + O(\nu + \nu^2\sigma)
  \end{align}
  Now choose $\sigma_*$ large enough so that for $\sigma > \sigma_*$ the asymptotic terms $o(\log \sigma)$ and $O(D\sigma^{-2})$ above apply well.  Furthermore, since $\sigma \nu^2 < \zeta_* \nu^{3/2}$ in the region $\{\sigma \leq \zeta_* \nu^{-1/2}\}$ under consideration, we can choose $T_*$ small enough so that the $O(\nu + \nu^2 \sigma)$ term is smaller, in absolute value, than $d/2$.   Specifically, for $\sigma \in [\sigma_*, \zeta_* \nu^{-p}]$ and $t < T_*$ we have
  \begin{align}
    &\tfrac{1}{2}d - \tfrac{3}{2} \sigma^{-1}\left( c_4\log \sigma + c_7 D \sigma^{-1} \right) \\
    &\leq \sigma^2 (bF^2 - Y^+) \\
    &\leq
    \tfrac{3}{2}d - \oh \sigma^{-1}\left( c_4 \log \sigma + c_7 D \sigma^{-1} \right).
  \end{align}
  Now choose $b = b(\sigma_*, D)$ so that $d = b c_3 - c_6$ is positive but small enough that
  \begin{align}
    \tfrac{3}{2}d - \oh \sigma_*^{-1} \left( c_4\log \sigma_* + c_7 D \sigma_*^{-1} \right) < 0
  \end{align}
  so the desired inequality holds at $\sigma = \sigma_*$.  Then choose $\sigma_1$ large enough so that
  \begin{align}
    \oh d - \tfrac{3}{2}\sigma^{-1}\left( c_4\log \sigma + c_7 D \sigma^{-1} \right) > 0
  \end{align}
  for $\sigma > \sigma_1$.  Then the desired inequality for $\sigma \in [\sigma_1, \zeta_* \nu^{-1/2}]$ also holds, for small enough times.
\end{proof}


\section{Constructing asymmetric forward evolutions}\label{asymmetric}
\begin{figure}[t]
  \centering
  \includegraphics[scale=.8]{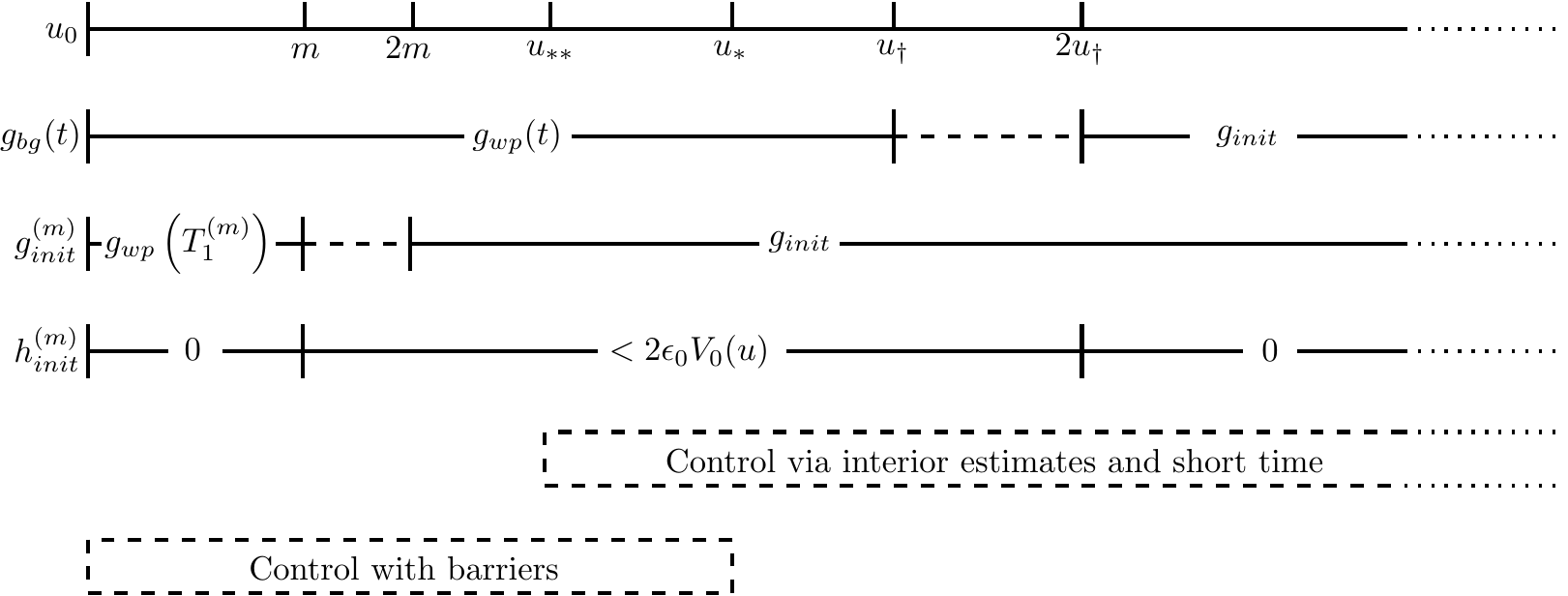}
  \caption[Map of mollification of asymmetric metrics]{
    A map of our background metric and mollified metrics.  The background metric $g_{bg}$ is defined in Section \ref{section:bg_setup} and the mollified metric $g^{(m)}_{init}$ is defined in Section \ref{section:mollified_setup}.  The dashed lines indicate that the metric is being interpolated between the value on the left and the value on the right.  
    }
  \label{np_basic}
\end{figure}

In this section we prove Theorem \ref{theorem:unsymmetrical_flow}.  We begin by gluing the control from Theorem \ref{local_stability} to a uniformly smooth Ricci flow using pseudolocality and regularity in the region strictly away from $u = 0$.  Then we will construct mollified metrics and take a limit.

Assume all the setup of Section \ref{section:local_stability}, including an $\Rm$-permissible model pinch $g_{mp}$ with a forward Ricci flow $(\bar M, g_{wp}(t))$. Fix a compact manifold $\bar N$ with an open subset $U$ and a diffeomorphism $\Phi : U \to \bar M$.  Assume that for some time interval $(0, T_{max}]$ and for some $u_{max} > 0$, the image of $\Phi$, $\Phi(U) \subset \bar M$, contains the set $\{(p, t): u(p, t) < u_{max}\}$.  Hereafter we suppress the diffeomorphism $\Phi$ and consider all of the functions that we had on $\bar M$ related to $g_{wp}$-- such as $u(p, t)$, $\sigma(p,t) = u(p, t)/\alpha(t)$, $w(p, t)$, and $v(p, t)$-- as functions on $\bar N \cap U$. The function $u_0: \bar N \cap U \to \real_{\geq 0}$ is still the value of $u$ for $g_{mp}$, as a notational convenience we extend the function $u_0$ to all of $\bar N$ so that $u_0 \geq u_{max}$ in $ \bar N \setminus U$.  For $u_1, u_2$ we define, similarly to how we defined $\bar M$, $\bar N_{[u_1, u_2]} = \{p: u_0(p) \in [u_1, u_2]\}$; for example $\bar N_{[u_1, \infty)} = \{p \in U: u_0(p) \geq u_1\} \cup \bar N \setminus U$.

\subsection{Flowing complete manifolds near smoothed model pinches}

\begin{lemma}\label{lemma:global_stable}
  Suppose $\epsilon < \barr \epsilon(g_{mp})$, $r_0 < \barr r_0(g_{mp})$, $u_* < \bar u_*(g_{mp})$, and $B>0$.  Suppose $g_{bg}(t)$ is a complete and smoothly time-dependent metric on $\bar N$ which agrees with $g_{wp}(t)$ in $\bar N_{[0, 4u_*]} \times [0, T_*]$, and $g_{init}$ is a smooth complete metric on $\bar N$.  Suppose,
  \begin{itemize}
  \item $g_{init}$ is close to $g_{\#} = g_{bg}(T_1)$ globally:
    \begin{align}
      \sup_{p \in \bar N_{[0, 4u_*]}}&|g(T_1) - g_\#|_{2, \eta; r_0|\Rm|^{-1/2}_{g_\#}, g_\#  } \leq \epsilon F_{full}(p, T_1).\\
      \sup_{p \in \bar N_{[4u_*, \infty)}}&|g(T_1) - g_\#|_{2, \eta; r_0, g_\#} \leq \epsilon
    \end{align}
  \item $g_{bg}$ does not change much at any point $p \in \bar N_{[u_*/2, \infty)}$:
    \begin{align}
      \sup_{p \in \bar N_{[u_*/2, \infty)}}\sum_{k=0}^3\left|\partial_t \left(\nabla^{g_{bg}(0)}\right)^kg_{bg}(t)\right|_{g_{bg}(0)} \leq B
    \end{align}
  \end{itemize}

  Then there is a  $T_*$ and $C$, depending on $g_{mp}$, $u_*$, $r_0$, and $\epsilon$ as well as $g_{init}$ restricted to the compact set $\bar N_{[u_*/2, \infty)}$, with the following property.  If $T_1 < T_*$ then there is a solution $g(t)$ to Ricci-DeTurck flow with background metric $g_{bg}$, which exists at least on the time interval $[T_1, T_*]$, with $g(T_1) = g_{init}$ and for all $t  \in [T_1, T_*]$,
  \begin{align}
    \sup_{p \in \bar N}\twoeta{g(t) - g_{bg}}{r_0|\Rm|^{-1/2}}{g_{bg}}(p, T_1) \leq C\epsilon F_{full}(p, T_1).
  \end{align}
\end{lemma}

\begin{proof}
  By standard theory there is a solution to Ricci flow on some time interval $[T_1, T_{final}]$ with $g(T_1) = g_{init}$.
  Let $u_1 = (5/8)u_*$ and $u_2 = (7/8)u_*$ so $u_*/2 < u_1 < u_2 < u_*$.

  We can apply pseudolocality followed by regularity (i.e. Lemma A.5 of \cite{topping}) for any point in $\bar N_{[u_1, \infty)}$, which gives us control $|(\nabla^{g_{init}})^k \Rm_{g_{init}}|_{g_{rcf}} < C$ for $k = 0, 1,2,3$, for $t < \min(T_*, T_{final})$.  Here $C$  and $T_*$ depend on the things $C$ and $T_*$ are allowed to in the statement of the theorem.  We used compactness of $\bar N_{[u_1, \infty)}$ to get a lower bound on the radius at which we can apply pseudolocality.  By differentiation the Ricci flow equation, we get that $|\partial_t \left(\nabla^{g_{init}}\right)^k g_{rcf}|_{g_{init}} \leq C$ for $k = 0, 1, 2, 3$.

  Now let $\Psi:\bar N \times [T_1, T_{final}) \to \bar N$ be the solution to $\partial_t \Psi = \lap_{g_{rcf}, g_{bg}} \Psi$ with $\Psi(p, 0) = p$, so that $g = (\Psi^{-1})^*g_{rcf}$ is a solution to Ricci-DeTurck flow with background metric $g_{bg}$. (The existence of $\Psi$, as well as of $g_{rcf}(t)$, is given by Theorem 6.7 of \cite{Shi}.)  In local harmonic coordinates, the equation $\partial_t \Psi = \lap_{g_{rcf}, g_{bg}} \Psi$ has the form
  \begin{align}
    \partial_t \Psi^l
    = g^{ij} \partial_i \partial_{j}\Psi^l
    - g^{ij}\left( \Gamma_{g_{rcf}} \right)_{ij}^k \partial_{k}\Psi^l
    + g^{ij} \left(\left(\Gamma_{g_{bg}}\right)^l_{mk} \circ \Psi \right)\partial_i\Psi^m\partial_{j}\Psi^k.    
  \end{align}
  We apply regularity within $\bar N_{[u_2, \infty)}$.  We can use compactness again to get a lower bound on the radius at which we can find harmonic coordinates around any point, as well as to get a bound on derivatives of $\Gamma_{g_{rcf}}$ and $\Gamma_{g_{bg}}$ in those coordinates.  All in all, we can get that in $\bar N_{[u_2, \infty)}$ and for $t < \min(T_*, T_{final})$ we have $|\partial_t \nabla^k \Psi| \leq C$ for $k = 0, 1, 2, 3$ (say with connection and norms with respect to $g_{init}$, although now we know that all the metrics are comparable).  

  Now we can take $T_*$ small enough so that for $t < \min(T_*, T_{final})$ we have $\Psi(\bar N_{[u_2, \infty)}, t) \subset \bar N_{[u_*, \infty)}$ and $\Psi^{-1}(\bar N_{[u_2, \infty)}, t) \subset \bar N_{[u_*, \infty)}$.  Then we can use our estimates on time derivatives to restrict $T_*$ and get the bound $|(\nabla^{g_{bg}})^k\left( g(t) - g_{bg}(t) \right)| \leq 2\epsilon F_{full}$ in $\bar N_{[u_*, \infty)}$ for $k = 0, 1, 2, 3$.  Now we can apply Lemma \ref{local_stability} (with a larger $\epsilon$) to get the desired $C^{2, \eta}$ control in $\bar N_{[0, u_*]}$ for $t < \min(T_*, T_{final})$.

  Finally, we can use our $C^2$ bounds on $g(t)-g_{bg}(t)$ to estimate that the curvature at time $\min(T_*, T_{final})$ is bounded, we find $T_{final} > \min(T_*, T_{final})$ and so the Ricci-DeTurck flow exists for $t \in [T_1, T_*]$.
  
\end{proof}

\subsection{Setup of the background metric}\label{section:bg_setup}

Now, we wish to set up a background metric to use for Ricci-DeTurck flow.  We chose constants $u_*$ and $u_{\dagger}$ with $4 u_{*} < u_{\dagger}$ and $4u_{\dagger}<u_{max}$.  Let $\eta:[0, \infty) \to \real$ be a fixed smooth cutoff function satisfying $\eta(x) \in [0, 1]$ and 
\begin{align}
  \eta(x) = 1 \text{ for } x < 1
  \quad
  \eta(x) = 0 \text{ for } x > 2,
\end{align}
and define $\eta_{r}(x) = \eta(x/r)$.
Then define
\begin{align}
  g_{bg}(t)
  =
  \eta_{u_\dagger}\left(u_0\right) g_{wp}(t)
  +
  \left(1-\eta_{u_\dagger}(u_0)\right)g_{init}. \label{eq:gbg_def}
\end{align}
We define $u_0(p) \geq u_{max}$ for $p \not \in U$.  So, $g_{bg}(t)$ is a time-dependent metric which agrees with $g_{wp}(t)$ for points $p \in  \bar N_{[0, u_\dagger]}$, and agrees with $g_{init}$ for points $p \in \bar N_{[2u_{\dagger}, \infty)}$.  

Note that we can always choose $T_*$ small enough (depending on $u_*$ and $u_\dagger$) so that for $t < T_*$ we have $u_0(p) < u_\dagger$ wherever $u(p,t) < 4u_*$.  Therefore $g(p,t) = g_{wp}(t)$ on the set $\{(p,t) : u(p,t) < 4u_*\}$, and $g_{bg}$ will satisfy the hypotheses of Lemma \ref{lemma:global_stable}.

\subsection{Setup of the mollified initial metrics}\label{section:mollified_setup}
As in the proof of Theorem \ref{theorem:model_pinch_flow}, we will construct the forward evolution from $g_{init}$ as a limit of mollified flows.  A parameter $m \in [0, 1]$ determines the space scale of the mollification.  For $T_1^{(m)}$ to be chosen, we define the mollified initial metric $g_{init}^{(m)}$ by 
\begin{align}
  g^{(m)}_{init} = \eta_{m}\left( u_0 \right) g_{wp}(T_1^{(m)}) + \left( 1 - \eta_m\left(u_0 \right) \right) g_{init}.
\end{align}
Let $h^{(m)}_{init} = g^{(m)}_{init} - g_{bg}(T_1^{(m)})$.  We derive bounds on $h_{init}^{(m)}$ and its derivatives.

\begin{lemma}\label{inith_close}
  There is a constant $C > 0$ such that for $m < \barr m$ and $T_1^{(m)} < \barr T_1^{(m)}$, we have
  \begin{align}
    \sup_{p \in \bar N_{[0, 4 u_*]}}&|h^{(m)}_{init}|_{2, \eta; r_0|\Rm|^{-1/2}, g_{bg}(T_1)} \leq C\epsilon_0 F_{full} \\
    \sup_{p \in \bar N_{[4u_*, \infty)}}&|h^{(m)}_{init}|_{2, \eta; r_0, g_{bg}(T_1)} \leq C\epsilon_0
  \end{align}
\end{lemma}
\begin{proof}
  Note that $g^{(m)}_{init}$ agrees with $g_{bg}(T_1^{(m)})$ in $\bar N_{[0, m]}$ and in $\bar N_{[2u_{\dagger}, \infty)}$ so we just have to worry about the compact set $\bar N_{[m/2, 4u_{\dagger}]}$.  In the region $\bar N_{[0, u_{\dagger}/2}$ which is strictly in the interior of where $g_{bg}(T_1)$ agrees with $g_{wp}(T_1)$, we can estimate
  \begin{align}
    |h^{(m)}_{init}|_{2, \eta; r_0|\Rm|^{-1/2}, g_{bg}(T_1^{(m)})}
    &\leq C|g_{init} - g_{mp}|_{2, \eta; r_0|\Rm|^{-1/2}, g_{mp}} \\
    &+ C\left(|g_{init}-g_{mp}|_{2, \eta; r_0|\Rm|^{-1/2}, g_{wp}(T_1^{(m)})} - |g_{init} - g_{mp}|_{2, \eta; r_0|\Rm|^{-1/2}, g_{mp}} \right) \\
    &+ C\left(|g_{mp} - g_{wp}(T_1^{(m)})|_{2, \eta; r_0|\Rm|^{-1/2}, g_{wp}(T_1^{(m)})}\right).
  \end{align}
  The constant $C$ comes from estimating terms coming from the cutoff function $\eta$.
  The first line is bounded by $C \epsilon_0 F_{full}$ by the assumption on $g_{init}$, and the following lines can be bounded by $C \epsilon_0 F_{full}(m)$ by taking $T_1^{(m)}$ sufficiently small (using the convergence of $g_{wp}(t)$ to $g_{mp}$ as $t \searrow 0$).  
  In the region $\bar N_{[u_\dagger/2, \infty)}$ we can use compactness similarly.
\end{proof}

\subsection{Global control and convergence}
The following lemma implies Theorem \ref{theorem:unsymmetrical_flow}.
\begin{lemma}
  Suppose we choose $\epsilon_0 < \bar \epsilon/C$ where $C$ is the constant from Lemma \ref{inith_close} and $\bar \epsilon$ is the constant from Lemma \ref{lemma:global_stable}.

  Let $g^{(m)}(t)$, $t \in [T_1^{(m)}, T_{final}^{(m)})$ be the Ricci-DeTurck flow starting from $g^{(m)}_{init}$.  Then $T_{final}^{(m)} > T_*$ for some $T_*$ independent of $m$.
  
  There is a sequence $m_j \searrow 0$ such that the time-dependent metrics $g^{(m_j)}(t)$ converge to a solution $g(t)$ of Ricci-DeTurck flow around $g_{bg}(t)$, with $g(0) = g_{init}.$  The convergence happens in $C^{2, \eta/2}_{loc}\left( \bar M \times [0, T_*]\setminus P\times \{0\} \right)$, where $P = \bar M \setminus M$.  Furthermore, the DeTurck vector fields $V[g^{(m)}, g_{bg}]$ converge, in $C^{1, \eta/2}_{loc}$, to $V[g(t), g_{bg}]$.
\end{lemma}
\begin{proof}
  By the previous sections, Lemma \ref{lemma:global_stable} applies to $g_{init}^{(m)}$ with background metric $g_{bg}$.  This implies that in any set $K$ compactly contained in $M$ we have $C^{2, \eta}$ control on $g^{(m)}$ for $t \in [0, T_*]$. Also, for any $t_0 > 0$ we have $C^{2, \eta}$ control on $g^{(m)}$ for $t \in [t_0, T_*]$.  Therefore, we can apply Arzela-Ascoli and a diagonalization argument to get convergence on $C^{2, \eta/2}_{loc} \left( \bar M \times [0,T_*] \setminus P \times \{0\} \right)$ of a subsequence.

  Since the convergence happens in $C^{2, \eta/2}$, the equation passes to the limit.  Since $V[g^{(m)}, g_{bg}]$ depends on one derivative of $g^{(m)}$ with respect to $g_{bg}$, we get the convergence of the vector fields.
\end{proof}


\appendix
\section{Nearly constant regions of reaction-diffusion equations}\label{nearly_cnst_pde_sect}
Let $\ein > 0$ and $c_v \in \real$.  We study solutions to
\begin{align}
  \label{eq:29}
  \square u
  &= -\ein + c_v u^{-1}|\nabla u|^2 \\
  &= -\ein + c_v v \label{mypde}
\end{align}
where we have defined $v = u^{-1}|\nabla u|^2$.  
We consider $u:M \times [T_1,T_2] \to \real$ which satisfies \eqref{mypde}, on an evolving Riemannian manifold $(M, g(t))$ which satisfies Ricci flow. (If $(M, g(t))$ does not satisfy Ricci flow, there is another term in \eqref{v_eqn_crude} below.)  The value of $c_v$ does not come into play very much here.  We are interested investigating regions where $v$ is small, and controlling other functions (and in particular $v$) in terms of $u$.  All constants in this section may implicitly depend on $c_v$ and $\ein$.

Applying the parabolic version of the Bochner formula ((1.6) of \cite{weakhaslhofernaber}) yields the following: $v$ satisfies
\begin{align}
  \square v &= u^{-1}\ein v + E_{error} v^2  \label{EQ:22}
\end{align}
where $E_{error}:M \times [T_1, T_2] \to \real$ satisfies
\begin{align}
  \label{v_eqn_crude}
  -C \left(1 +\frac{|\nabla \nabla u|^2}{v^2}   \right)
  \leq E_{error}
  \leq C 
\end{align}
for some constant $C$. Using this equation, Lemma \ref{sup_solns} will allow us to control $v$ from above, and also from below if we obtain a priori that $\frac{|\nabla \nabla u|^2}{v^2}$ is bounded.

We will use functions of $u$ and $t$ to create sub- and supersolutions to other PDE, and in particular to control $v$.
If $F$ depends on $u$ and $t$ alone then
\begin{align}
  \label{eq:116}
  \square F = \left( (\square u)\dd1 F + \dd{t} F - v \dd{2} F \right) \left( u^{-1} F \right).
\end{align}
Here, we use the notation $\dd{k}{F} = u^k \frac{1}{F}\partial_u^k F$ and $\dd{t} F = u \frac{1}{F} \partial_t F$.  These are both invariant under scaling the system or $F$, and in our situation they will always be bounded.  Since we have an equation for $\square u$ we can calculate further,
\begin{align}
  \square F = \left( \dd{t} F - \ein \dd1 F + v\left(c_v \dd{1}F -  \dd{2} F\right) \right) \left( u^{-1} F \right). \label{utdep_square_partic}
\end{align}
This formula tells us that when $v$ is small and $\dd{1}F, \dd{2}F$ are controlled, $\square F$ is approximately the first order linear operator $L[F] \defeq \left( \partial_{t;u} - \ein \partial_u \right) F = \left(\dd{t}F - \ein \dd1 F \right) (u^{-1}F) $.

A relevant function is $\hat u(u,t) \defeq u + \ein t$.  We will also use $Q(u,t) \defeq u^{-1}\hat u$.  These are related to the linear operator $L[F]$.  $\hat u$ gives the characteristic curves of the equation, and $Q$ is a solution to $L[F] = \ein u^{-1}F$ with constant initial data 1.

The following lemma uses these functions to partially solve certain linear equations on the evolving manifold.  We claim that for a given smoooth initial function $Z_0 : \real_+ \to \real_+$, $Z \defeq Q^a \cdot \left( Z_0 \circ \hat u \right)$ approximately solves a certain equation. To be explicit,
\begin{align}
  Z(x, t) = \left( \frac{u(x, t) + \ein t}{u(x, t)} \right)^a Z_0\left( u(x, t) + \ein t \right).
\end{align}
\begin{lemma}{(Approximate solutions to equations)}\label{approx_soln}
  $Z$ satisfies
  \begin{align}
    \left( \square - a \ein u^{-1} \right) Z = E u^{-1}v {Z} \label{squareZ}\\
    {Z}(p, 0) = Z_0(u(p,0)) 
  \end{align}
  where $E : M \times [0, T) \to \real$ satisfies $|E| \leq C(1 + |\dd1{Z_0}| + |\dd2{Z_0}|)$, for some constant $C$.  Also,
  \begin{align}
    |\dd1 Z | + |\dd2 Z| \leq C(1 + |\dd1 Z_0| + |\dd2 Z_0|).
  \end{align}
\end{lemma}
\begin{proof}
  This is just a calculation, but the following steps give the idea.  First, note that $L[\hat u] = 0$ and $L[Q] = \ein u^{-1}Q$.  This lets us calculate that $L[Z_0 \circ \hat u] = 0$ and then $L[Z] = a \ein u^{-1}Z$.  Therefore by \eqref{utdep_square_partic} all that's left to see is that $|\dd1 Z| + |\dd2 Z| \leq C ( 1 + |\dd1 Z_0| + |\dd2 Z_0|)$, which is just some more calculus.
\end{proof}  

Now suppose the term $E_{error}$ in \eqref{EQ:22} is bounded.  Then we may expect $v$ itself to be approximately given by a solution to $\square v - \ein u^{-1}v = 0$.  By Lemma \ref{approx_soln} we find that $v$ should be approximately given by $V \defeq Q \cdot \left( V_0 \circ \hat u \right)$ for some initial data $V_0:M \to \real$.  This, in turn, will give us control on the error term $E u^{-1} v Z$ in Lemma \ref{approx_soln}.

Now we create sub- and supersolutions to $\square z = a \ein u^{-1}z$, based off of the approximate solution $Z$, which beat the error in this approximate solution.  The supersolution is defined as $Z^+ = (1 + DV)Z$, and the subsolution as $Z^- = (1-DV)Z$, for some sufficiently large $D > 0$.  We will assume that
\begin{align}
  \sup_{\real_+} |\dd1 V_0| + |\dd2 V_0| + |\dd1 Z_0| + |\dd2 Z_0| \leq C_0.
\end{align}

\begin{lemma}{(Supersolutions to parabolic equations)}\label{sup_solns}
  Suppose $D > \berr D(C_0, a) >0$.  There is a  $c > 0$ and $\epsilon(D, C_0, a) > 0$ with the following property.

  Let $\Omega$ be a subset of space-time where
  $v(p,t) \leq 2 V$ and $V < \epsilon$.
  Then $Z^-$ and $Z^+$ are sub- and supersolutions to $(\square - a \ein u^{-1})$ on $\Omega$.  More precisely, 
  \begin{align}
    \left( \square - a \ein u^{-1} \right) Z^+ &\geq (cD) u^{-1}v Z^+ \label{supsoln_ineq}\\
    \left( \square - a \ein u^{-1} \right) Z^- &\leq - (cD) u^{-1}v Z^- \label{subsoln_ineq}
  \end{align}
  on $\Omega$.
\end{lemma}
\begin{proof}
  Write $Z^+ = Z + Z_2$ with $Z_2 = D V Z = DQ^{p+1}((V_0\cdot Z_0) \circ \hat u)$.  Then we can use Lemma \ref{approx_soln} and in particular \eqref{squareZ} to calculate the heat operator applied to $Z_2$:
  \begin{align}
    \label{eq:52}
    \square Z_2 - (a+1) \ein u^{-1}Z_2 = E_2 u^{-1}v Z_2
  \end{align}
  where $E_2$ is some error which is absolutely bounded depending on $C_0$. 
  In terms of the linear equation we are interested in, this means
  \begin{align}
    \label{eq:52}
    \left( \square  - a \ein u^{-1} \right)Z_2
    &=\ein u^{-1} \left(1 +  E_2 v \right) Z_2 
    =\ein Du^{-1} \left(1 + E_2 v \right) VZ.
  \end{align}
  By choosing $\epsilon$ small enough we can force $1 + E_2 v \geq \on2$ to hold in $\Omega$.  Now using again equation \eqref{squareZ} from Lemma \ref{approx_soln}, but now applied to $Z$, we find
  \begin{align}
    \label{eq:53}
    (\square - a \ein u^{-1})(Z^+)
    &= (\square - a \ein u^{-1})(Z + Z_2) \\
    &\geq \left( E u^{-1}vZ + \on2 \ein u^{-1}D VZ \right) \\
    &= 
      \left( \frac{E}{D}  + \on2\ein \frac{V}{v} \right)
      \frac{1}{1 + D V} Du^{-1}v Z^+
  \end{align}
  Here, $E$, another error term of unknown sign, is bounded independently of $D$. The lemma follows by using the assumption that $v \leq \oh V$, choosing $\berr D$ large enough to force $|\frac{E}{D}| \leq \on8 \ein$, and then choosing $\epsilon$ small enough so that $\frac{1}{1 + D V} \geq \oh$.  Then we take $c = \on{16} \ein$.
\end{proof}

The next lemma claims that the bounds on $|\dd{i}{(Z^+)}|$ carry over to the sub- and supersolutions.
\begin{lemma}\label{dd1z_bnd}
There is a constant $C$ depending on $C_0$ and $p$, and in particular independent of $D$, such that
\begin{align}
  |\dd1 {(Z^+)}| + |\dd2{(Z^+)}| \leq C,
\end{align}
and similarly for the subsolution $Z^-$.  

If in addition we assume that $|\nabla \nabla u| \leq C_{hess} v$, then
$|\nabla\nabla Z^+| \leq Cv$ for a constant depending on $C_0$ and $C_{hess}$.
\end{lemma}
\begin{proof}
  First, derive bound for $V = Q V_0 \circ \hat u$ and $Z = Q^a Z_0 \circ \hat u$.
  \begin{align}
    \dd1 V &= \dd1 Q + \left( \dd1{(V_0)}\circ \hat u \right) \dd1{(\hat u)}\\
           &= -(1-Q^{-1}) + \left( \dd1{(V_0)}\circ \hat u \right)Q^{-1},
  \end{align}
  so $|\dd1 V| \leq 1 + \sup |\dd1 V_0|$.  Similarly, we can bound $\dd1 Z$ by $p + \sup |\dd1 Z_0|$.

  Now calculate,
  \begin{align}
    \label{eq:32}
    \dd1{(1 + D V)}
    &= \frac{u \partial_u (1 + D V)}{1 + DV} \\
    &= \frac{DV}{1 + DV} \frac{u \partial_u V}{V} \leq \dd1{V}.
  \end{align}
  Once we have this, the full bound on $\dd1{(Z^+)}$ follows from
  \begin{align}
    \dd1{(Z^+)} = \dd1{\left((1 + DV)Z\right)} = \dd1{\left( 1+DV \right)} + \dd1{Z}.
  \end{align}
  The bound on $\dd2{(Z^+)}$ is similar.

  To get the second claim, use the following which is valid for any function $F$ of $u$ and $t$:
  \begin{align}
    \nabla \nabla F
    &= (\partial_u^2F)\nabla u \otimes \nabla u + \partial_u F \nabla \nabla u \\
    &= u^{-1}F \cdot \left( (\dd2 F)u^{-1}\nabla u \otimes \nabla u + \dd1 F \nabla \nabla u \right).
\end{align}
\end{proof}

\section{Calculations on the Bryant soliton}\label{bryant_facts}
Let $(Bry, g_{Bry}, X)$ be the Bryant steady soliton with minimum scalar curvature $R_0$.  Bryant's original work is \cite{bryant_orig}, see also Section 1.4 of \cite{bookgeometric} for an exposition of the construction.  The extra analysis carried out here is generally justified by the analyticity of the solution.  Let
$
  g_{Bry} = ds^2 + u_{Bry}g_{S^q} = ds^2 + \phi^2_{Bry}g_{S^q}
$
and
$
  X = \grad f$.

On any steady soliton we have $R + |\nabla f|^2 = R_0$  (Corollary 1.16 in \cite{bookgeometric}).  Taking the trace of the soliton equation we have $R + \lap f = 0$, so we find
$
  \lap_f (-f) = R_0
$.  Since the Bryant soliton is a singly warped product, we have more precisely $df = - \sqrt{R_0 - R} ds$.  

Either \cite{bryant_orig} or \cite{bookgeometric} show that $\phi_{Bry} = O(\sqrt{s})$ as $s \to \infty$ and $R = O(s^{-1})$.  To find the exact coefficient use the equation for $\phi$,
\begin{align}
  0 = \phi_{ss} - f_s \phi_s - (q-1)\phi^{-1}\left(1 - \phi_s^2\right),
\end{align}
so $\phi \sim R_0^{-1/4}\sqrt{\ein s}$ and $u \sim R_0^{-1/2}\ein s$ at $\infty$.

\subsection{Next order approximation}\label{bryant_facts_nextorder}

So far we have found as $s \to \infty$
\begin{align}
  f &= -(1 + o(1)) R_0^{-1/2} s  \\
  u &= (1 + o(1)) \mu R_0^{-1/2} s.
\end{align}
Now we seek the next term in the asymptotic expansion.  The function $u$ satisfies
\begin{align}
  0 = u_{ss} - f_s u_s + c_v u^{-1}u_s^2 - \ein\label{u_eqn_bry1}
\end{align}
where $c_v = \oh (\oh \ein - 1)$.  We also have $\lap_f (-f) = R_0$ or
\begin{align}
  0 = (-f)_{ss} - f_s(-f)_s + q \phi^{-1}\phi_s (-f_s) = R_0.
\end{align}
Strictly in terms of $u$ and $\bar f = -f/R_0$ we have
\begin{align}
  u_{ss} + R_0 \bar f_s u_s + c_v u^{-1}u_s^2 &= \mu \\
  \bar f_{ss} + R_0 \bar f_s^2 + \oh q u^{-1}u_s \bar f_s &= 1 
\end{align}
Write $G = \bar f_s$. 
\begin{align}
  u_{ss} + R_0 G u_s + c_v u^{-1}u_s^2 &= \mu \label{ug_1}\\
  G_{s} + R_0 G^2 + \oh q u^{-1}u_s G &= 1  \label{ug_2}
\end{align}
Now write $u = \mu R_0^{-1/2} s + u_1$ and $G = R_0^{-1/2} + G_1$.  Partially writing out \eqref{ug_1} and \eqref{ug_2},
\begin{align}
  u_{1, ss}
  + R_0
  \left( R_0^{-1} \ein + \ein R_0^{-1/2} G_1 + R_0^{-1/2}u_{1,s} + u_{1,s}G_1 \right)
  + c_v u^{-1}u_s^2 &= \ein, \\
  G_{1,s}
  + R_0
  \left( R_0^{-1} + 2 R_0^{-1/2} G_1 + G_1^2 \right)
  + \oh q u^{-1}u_s G 
  &= 1.
\end{align}
Simplifying,
\begin{align}
  u_{1,ss} + \ein R_0^{1/2} G_1 + R_0^{1/2} u_{1,s} + R_0 u_{1,s} G_1 + c_v u^{-1}u_s^2 = 0,\\
  G_{1,s} + 2 R_0^{1/2}G_1 + R_0 G_1^2 + \oh q u^{-1}u_s G = 0.
\end{align}
We have $u^{-1} = \ein^{-1} R_0^{1/2} s^{-1} \left( 1 - u_1 + o(u_1) \right)$.  The highest order terms in the equation for $G_1$ are $2 R_0^{1/2}G_1 + \oh q R_0^{-1/2}s^{-1}$,
therefore
\begin{align}
  G_1 = (1 + o(1)) \left( - \on4 R_0^{-1} s^{-1} \right).
\end{align}
Then the highest order terms in the equation for $u_1$ are
$
  \ein R_0^{1/2} G_1 + R_0^{1/2}u_{1,s} - c_v \ein R_0^{-1/2}s^{-1}
$
which gives
\begin{align}
  u_1 = (1 + o(1))R_0^{-1}\left(\on4 q \ein + c_v \ein \right) \log s
\end{align}

Unravelling definitions, we have found
\begin{align}
  \bar f &= R_0^{-1/2}s + \on4 q R_0^{-1}\log s + o(\log s),  \\
  u &= \ein R_0^{-1/2}s + R_0^{-1}\left( \on4 q + c_v \right) \ein \log s + o(\log s),
\end{align}
so writing $f$ in terms of $u$,
\begin{align}\label{bry_second_order_expansion}
  \bar f = \ein^{-1} u - \on4 q R_0^{-1} \log u + o(\log u).
\end{align}

\subsection{Continuation of the proof of Lemma \ref{wpert_exist}}
\label{wpert_exist_continuation}

For the Bryant soliton, we have \eqref{u_eqn_bry1} or equivalently
\begin{align}
  \lap_X u - \ein + (c_v - \oh q)v_{\Bry} = 0\label{u_eqn_bry}.
\end{align}
(The term $-\oh q v_{\Bry}$ comes because $\lap u = u_{ss} + \oh q u^{-1}u_s^2$.)
We also have, defining $\bar f = -f/R_0$,
$
  \lap_X \bar f = 1
$. (See the introduction of this appendix.)
Thinking in terms of $u$ this says,
\begin{align}
  \partial_uf \lap_X u + u_s^2\partial_u^2f &= -R_0
\end{align}
Then, using \eqref{u_eqn_bry}
\begin{align}
  ( \ein - (c_v - \oh q)v_{\Bry})\partial_uf +  u_s^2\partial_u^2f &= -R_0 \\
  ( \ein - (c_v - \oh q)v_{\Bry})\partial_uf +  uv_{\Bry}\partial_u^2f &= -R_0
\end{align}
The asymptotics claimed in the Lemma are given in \eqref{bry_second_order_expansion}.


\section{Equations for warped products and Ricci flow}\label{warped_product_section}
In this section we review some of the properties of Ricci flow on warped products.  The  metrics are on  the topology $M = B^m \times N^q$, for some manifold $B$ which we call the base.    The metrics have the form
\begin{align}
  g = g_B + \phi^2(b)g_{N},
\end{align}
where $g_B$ is a metric on $B$, $g_{N}$ is a metric on $N$, and $\phi: B \to \real_+$.  We assume that $g_N$ is an Einstein manifold: $2\Rc[g_N] = \ein_N g_N$.  

In this thesis we are mostly concerned with doubly warped products over intervals, i.e. metrics on $I \times S^q \times F$ of the form
$
  a(x)dx^2 + \phi^2(x) g_{S^q} + \psi^2(x) g_{F}
$.
These are singly warped products in two ways: with base $I \times S^q$ and fiber $F$ or with base $I \times F$ and fiber $S^q$.  Both points of view have been useful for our intuition.  A big simplification for a doubly warped product over an interval is that the hessian of a function of $x$ is much simpler than that of a function of a general base.

Everything in this section can be found or derived from Section 7 of \cite{semiriem}.

\subsection{Curvatures}
The curvature of a warped product can be described as follows.  If $U$ and $V$ are perpendicular unit vectors on the fiber, then
\begin{align}
  R(U,V,U,V) = \frac{R_N(U,V,U,V) - |\nabla \phi|^2}{\phi^2}.
\end{align}
In particular, if $(g_N, N)$ is the metric of constant sectional curvature $Sec$, then
\begin{align}
  R(U,V,U,V) = \frac{Sec - |\nabla \phi|^2}{\phi^2}.
\end{align}

For vectors $U$ on the fiber and $X$, $Y$ on the base, we have
\begin{align}
  R(U,X,U,Y) = - \frac{\nabla_X \nabla_Y \phi}{\phi}, \label{hess_sec}
\end{align}
and if both $W$, $X$, $Y$, and $Z$ are all vectors on the base, then
\begin{align}
  R(X,Y,Z,W) = R_B(X,Y,Z,W).
\end{align}

From these formulae, we can calculate the Ricci curvature directly from definition. Using $2\Rc[g_N] = \ein_N g_N$, 
\begin{align}
  \Rc(U,V) &= \left( \Rc_{B}(X,Y) - q \phi^{-1}\nabla_X \nabla_Y \phi \right)\\
  &+ 
  \left(- \phi\lap_{B} \phi
             + \oh \mu \left(1 - \frac{2(q-1)}{\mu} |\nabla \phi|^2 \right)
             \right) \phi^2g_{N}
\end{align}

\subsection{Ricci flow for warped products}\label{rf_warped_products}
If $g$ evolves by Ricci flow, then
\begin{align}
  \Rf[g_B] &= 2 q \phi^{-1}\nabla \nabla \phi \\
  \square_B \phi &= - \oh \ein \phi^{-1}\left( 1 - \frac{2(q-1)}{\ein}|\nabla \phi|^2 \right) \\
  \square_M \phi &= - \phi^{-1}|\nabla \phi|^2 - \oh \ein \phi^{-1}
\end{align}
and $u = \phi^2$ satisfies, setting $v = u^{-1}|\nabla u|^2$,
\begin{align}
  \square_M u
  &= \left( - u^{-1}\mu  \right)u - v\label{evo_u_M}\\
  \square_B u
  &= - \ein + \on4 (2(q-1) - 2) v\label{evo_u_singlywarped}.
\end{align}
Recall in the case $g_N = g_{S^q}$ we have $\ein = 2(q-1)$.

\subsection{Curvatures for doubly warped products}
Now consider a metric of the form
\begin{align}
  g = a(x) dx^2 + \phi^2 g_{F_1} + \psi^2 g_{F_2}, \quad x \in I.
\end{align}
We define an arclength coordinate $s$ (up to a constant) by $ds^2 = a dx^2$.
We can view $g$ as a warped product with fiber $g_{F_1}$ over base $I \times F_2$, as well as a warped product with fiber $g_{F_2}$ over the base $I \times F_1$.  Consider for simplicity the case when $g_{F_1}$ has constant sectional curvature $Sec_1$ and $g_{F_2}$ has constant sectional curvature $Sec_2$.  Then there are five special sectional curvatures:
\begin{align}
  L_1 &= \frac{Sec_1 - |\nabla \phi|^2}{\phi^2} = \frac{Sec_1 - \phi_s^2}{\phi^2},
  \quad L_2 = \frac{Sec_2 - \psi_s^2}{\psi^2}, \\
  K_1 &= - \frac{\phi_{ss}}{\phi},
        \quad K_2 = - \frac{\psi_{ss}}{\psi},
        \quad K_{mix} = - \frac{\phi_s \psi_s}{\phi \psi}.
\end{align}
The curvatures $L_1$ and $L_2$ are those that we get from planes spanned by two perpendicular vectors tangent to the same fiber.  $K_1$ and $K_2$ come from planes spanned by $\partial_s$ and a vector on one of the fibers.  $K_{mix}$ comes from a plane spanned by a vector on $F_1$ and a vector on $F_2$; this comes from the extra terms (compared to a product) in computing the hessian in \eqref{hess_sec}.

\subsubsection{Curvatures in terms of $u$, $v$ and $w$}\label{curvatures_our_coordinates}
We put the curvatures of a doubly warped product in terms of $v$ and $w$, and their $u$ derivatives.  Recall the definitions
\begin{align}
  u = \phi^2, \quad w = \psi^2, \quad v = u^{-1}|\nabla u|^2 = 4 |\nabla \phi|^2
\end{align}
First, we have
\begin{align}
  L_1 = u^{-1}Sec_1 - \on4 u^{-1} v.
\end{align}
Now calculate,
$
  \partial_s u \partial_u v = (\partial_s u) 4 (\partial_s \phi) (\partial_u \partial_s \phi) = 4 (\partial_s \phi) (\partial_s^2 \phi)
$
so
$
  2 \phi \partial_u v = 4 (\partial_s^2 \phi)
$
and
\begin{align}
  K_1 =-\on2  \partial_u v = -\frac{\partial_s^2 \phi}{\phi}.
\end{align}
Now we calculate the curvatures involving $\psi$.
\begin{align}
  \psi_s = \on2 w^{-1/2}w_s = \on2 w^{-1/2}w_uu_s = \on2 w^{-1/2}u^{1/2}v^{1/2}w_u
\end{align}
so
\begin{align}
  L_2 &= \frac{Sec_2}{w} - \on4 u^{-1}v \left( u^2 w^{-2}w_u^2 \right) \\
  K_{mix} &= (\on2 u^{-1/2}v^{1/2}) (\on2 w^{-1}u^{1/2}v^{1/2}w_u) = \on4 u^{-1}v \left( uw^{-1}w_u \right) .
\end{align}
Finally, we calculate
\begin{align}
  \psi_{ss} &= \on4 \left( w^{-3/2} u^{1/2} v^{1/2}w_u w_s + w^{-1/2}u^{-1/2}v^{1/2}w_uu_s + w^{1/2}u^{1/2}v^{-1/2}w_uv_s + w^{-1/2}u^{1/2}v^{-1/2}w_{us} \right) \\
  &= \on4 w^{-1/2} u v (w^{-1}w_u^2 + u^{-1} w_u + v^{-1} w_u v_u + w_{uu}).
\end{align}
Therefore,
\begin{align}
  K_2 = - \on4 u^{-1} v (u^2w^{-2}w_u^2 + uw^{-1}w_u + u^2v^{-1} w^{-1}w_u v_u + u^2w^{-1}w_{uu})
\end{align}

\subsection{Deriving the evolution of $v$.}\label{v_deriving_section}

In this Lemma, $\Rf[g_B] = \partial_t g_B - \left( - 2 \Rc_{g_B} \right)$.

\begin{lemma}\label{z_calc_lemma}
  Suppose $(B, g_B)$ is an evolving Riemannian manifold and $\phi:B \times [T_1, T_2] \to \real_+$ is an evolving function on $B$.  Suppose $g_B$ and $\phi$ satisfy
  \begin{align}
    \label{eq:70}
    \Rf[g_B]
    &= 2 c_1 \phi^{-1}\nabla\nabla \phi \\
    \square_B \phi &= \oh\phi^{-1}\cdot (-\ein+ c_z z) 
  \end{align}
  where $z = |\nabla \phi|^2$.  Let $\kappa(p,t)$ be the norm of the second fundamental form of the level set of $u$ passing through $p$ at time $t$.  Then $z$ satisfies
  \begin{align}
    \label{eq:65}
    \square z
    &= \phi^{-2}(\ein -  c_z z) z \\
    &+ (c_z -  c_1) \ip{\nabla z}{\nabla \log \phi}  - z^{-1}|\nabla z|^2 + \oh \phi^2 z^{-2}\left( \ip{\nabla z}{ \nabla \log \phi} \right)^2\\
    &- 2 z \kappa^2 
  \end{align}
\end{lemma}

\begin{proof}

We can apply the parabolic Bochner formula ((1.6) of \cite{weakhaslhofernaber}) to these equations to find
\begin{align}
  \label{eq:55}
  \square |\nabla \phi|^2
  &= 2 \ip{\nabla \square \phi }{\nabla \phi} - 2 |\nabla \nabla \phi|_B^2 - \Rf(\nabla \phi, \nabla \phi)  \\
  &= 2 \ip{\nabla \square \phi }{\nabla \phi} - 2 |\nabla \nabla \phi|_B^2 - 2c_1 \phi^{-1}\nabla_{\nabla \phi}\nabla_{\nabla \phi}\phi
\end{align}

We calculate the first term:
\begin{align}
  \label{eq:61}
  2 \ip{\nabla \square \phi}{\nabla \phi}
  &=  \phi^{-2}(\ein-c_zz) |\nabla \phi|^2 + c_z \phi^{-1}\ip{\nabla z}{\nabla \phi} \\
  &=  \phi^{-2}(\ein - c_z z)z  + c_z \ip{\nabla z}{\nabla \log \phi}
\end{align}
For the second term, we can change the hessian to
\begin{align}
  \label{eq:68}
  -2 |\nabla \nabla \phi|^2
  &= -2 z \fund^2 -  z^{-1}|\nabla z|^2  + \on2 z^{-2}\ip{\nabla z}{\nabla \phi}^2 \\
  &= -2 z \fund^2 -  z^{-1}|\nabla z|^2  + \on2 z^{-2}\phi^2\ip{\nabla z}{\nabla \log \phi}^2 
\end{align}
And for the third term, we can change the hessian using
\begin{align}
  \label{eq:64}
  -2c_1\phi^{-1}\nabla_{\nabla \phi}\nabla_{\nabla \phi}\phi
  &= -  c_1 \phi^{-1}\ip{\nabla z}{ \nabla \phi} = - c_1 \ip{\nabla z}{\nabla \log \phi} 
\end{align}

Putting everything together, we find the desired equation.
\end{proof}

\begin{corollary}\label{v_calc_lemma}
  In the setting of Lemma \ref{z_calc_lemma}, suppose $g_B$ and $u$ satisfy
  \begin{align}
    \Rf[g_B]
    &= 2 c_1 u^{-1/2}\nabla\nabla u^{1/2} \label{base_eqn}\\
    \square_B u &= - \ein + c_v v \label{u_eqn}
  \end{align}
  where $v = u^{-1}|\nabla u|^2$.  
Define the constants $c_z = (4 c_v + 2)$, $c_{v}' =  \on4 c_z$,  and $c_3 = \oh \left( c_z -  c_1 \right)$.  Then $v$ satisfies
  \begin{align}
  \square v
  &= u^{-1}(\ein -  c_v' v) v - 2 v \fund^2 \\
  &+ c_3 \ip{\nabla v}{\nabla \log u}  - v^{-1}|\nabla v|^2 + \oh u v^{-2}\left( \ip{\nabla v}{ \nabla \log u} \right)^2\\
  \end{align}
\end{corollary}

\subsubsection{Equidistant Level Sets}
Now, suppose that the level sets of $u$ are equidistant.  Then $v$ is dependent on $u$ and $t$ alone  so we find $\nabla v = |\nabla u|^{-1}\ip{\nabla v}{\nabla u}$.  Then from Corollary \ref{v_calc_lemma},
\begin{align}
  \label{square_v_edh1}
  \square_B v
  &= u^{-1}(\ein - c_v' v)v - 2 \fund^2 v - u^{-1}\bar T v \\
  &+ c_3v(\partial_u v) - \oh u(\partial_u v)^2   
\end{align}
On the other hand, since the level sets of $u$ are equidistant, we can use that $v$ is a function of $u$ and $t$ to calculate $\square_B v$ in terms of derivatives with respect to $u$, using \eqref{u_eqn}.  
\begin{align}
  \label{square_v_edh}
  \square_B v
  &=  (- \ein + c_vv)\partial_uv + \partial_{t;u}v - uv\partial_u^2v.   
\end{align}
From \eqref{square_v_edh1} and \eqref{square_v_edh} it follows that,
\begin{align}
  \label{c_4_intro}
  \partial_{t;u} v
  &= u v \partial_u^2v - \oh u (\partial_uv)^2+ u^{-1}(\ein - c_v' v)v + (c_4 v + \ein )\partial_uv \\
  &- 2 \fund^2 v - u^{-1} \bar T v 
\end{align}
where $c_4 = c_3 - c_v$.

\subsubsection{The case of warped product Ricci flow}\label{wprcf_eqns}
In the case of Ricci flow of a metric $g = g_B + u g_{S^q}$, where the Ricci curvature of $g_{S^q}$ is $\ein g_{S^q} = 2(q-1)g_{S^q}$, we have
\begin{align}
  \Rf[g_B] &= 2 q u^{-1/2}\nabla\nabla u^{1/2} \\
  \square_B u &= - \ein + \on4 (\ein-2)v \label{u_eqn_wpsphere}
\end{align}
Therefore in Lemma \ref{v_calc_lemma} we have $c_v = \on4(\ein-2)$ and $c_1 = q = \oh \ein + 1$.  Then we find 
$c_z = 4c_v + 2 = \ein$, $c_v' = \on4 c_z = \on4 \ein$, $c_3 =  \on4 \ein -\on2$, and $c_4 = c_3 - c_v = \on4 \ein - \oh - (\on4 \ein - \oh) = 0$.  So, from \eqref{c_4_intro},
\begin{align}
  \partial_{t;u} v
  &= u v \partial_u^2 v - \oh u (\partial_u v)^2 
  + \ein \left(1 - \on4 v \right) u^{-1} v +  \ein \partial_u v \label{vevo_basic}\\
  &- 2 (\fund^2) v
\end{align}
One convenient way to write this is as,
\begin{align}
  \label{V_evo}
  \partial_{t;u}v = u^{-1}\qop[v,v] + u^{-1}\lop[v] - 2 (\fund^2) v
\end{align}
where $\lop$ and $\qop$ are the operators
\begin{align}
  \label{eq:74}
  \lop[w]
  = L(w, \partial_uw), \quad L(A,B) = \ein A + \ein B.
\end{align}
and
\begin{align}
  \label{eq:75}
  \qop[w,w]
  = Q(w, u\partial_uw, u^2 \partial_u w_{uu}),\quad
  Q(A,B,C)
  = AC - \oh B^2 - \on4 \ein C^2.
\end{align}
For $w_1$ and $w_2$ different functions, we define $\qop[w_1,w_2]$ to be the extension of $\qop$ to a symmetric bilinear operator.

\subsubsection{Writing the evolution in terms of $L$ and $\phi$}
It is also convenient to consider the evolution of $L = \frac{1 - \on4 v}{u}$.  $L$ is a sectional curvature, so it is a geometrically natural quantity to consider.  If the metric is smooth near $u=0$ then $L$ will be bounded there, which gives us more information that $v$ being bounded.  

Coming from \eqref{vevo_basic}, replace $v = 4(1-uL)$ and divide through by $-4u$ to find
\begin{align}
  \partial_{t;u} L
  &= 4u\left(1 - uL\right)\partial_u^2L
    + 2 u^2 (\partial_u L)^2 \\
  &+ (\ein + 8 - 4uL)\partial_uL
    + (\ein + 2)L^2 + \oh u^{-1} \kappa^2 v.
\end{align}
An important point here is that the terms $u^{-1}L$ cancel.  This is expected, since for example the sphere has constant non-zero curvature $L$ despite $u$ going to zero.
Let us also put this in terms of derivatives with respect to $\phi = \sqrt{u}$.  Note 
\begin{align}
  \partial_u = \oh \phi^{-1}\partial_\phi , \quad
  u \partial^2_u = \on4 \left( \partial^2_\phi - \phi^{-1}\partial_\phi \right). \label{u_to_phi_1}
\end{align}
Since $\phi$ is a function of $u$, $\partial_{t;u} = \partial_{t;\phi}$.  So, we have
\begin{align}
  \partial_{t;\phi} L
  &= \left( 1 - \phi^2L \right)(\partial^2_\phi L - \phi^{-1}\partial_\phi L)
    + \on2\phi^2 (\partial_\phi L)^2 \\
  &+ \phi^{-1}(\oh \ein + 4 - 2\phi^2L)(\partial_\phi L) + (\ein + 2)L^2
    + \oh \kappa^2 v \\
  &= \left(1 - \phi^2 L\right)\partial^2_\phi L + \oh \phi^2 (\partial_\phi L)^2\\
  &+ \phi^{-1}(\oh \ein + 5 - \phi^2L)\partial_\phi L + (\ein + 2)L^2
    + \oh \kappa^2 v.\label{evo_L_in_phi}
\end{align}
The advantage of this is the clear regularity around $\phi = 0$ provided $L$ is bounded.

\subsection{Deriving the evolution of $w$}\label{additional_wp}
We continue considering the Ricci flow of a metric of the form $g = g_B + u g_{S^q}$:
\begin{align}
  \label{eq:10}
  \Rf[g_B] &= 2 c_1 u^{-1/2}\nabla \nabla u^{1/2}, \\
  \square_B u &= - \ein + c_v v. 
\end{align}
Here $c_v = \on4(\ein - 2)$.  Suppose that $g_B$ itself has a warped product factor: $B = B_2 \times F^{p}$ and $g_B = g_{B_2} + w g_F$.  Take $y = w^{-1}|\nabla w|^2$ and suppose that $2\Rc[g_F] = \ein_F g_F$.  We make no assumptions on the sign on $\ein_F$.

To quickly derive an equation for $h$ in terms of $\square_B$, go from \eqref{evo_u_M} which says
\begin{align}
  \square_{B_2 \times F \times S^q} w = - \ein_F - y
\end{align}
where $y = w^{-1}|\nabla w|^2$.
Since
\begin{align}
  \square_{B_2 \times F \times S^q} w
  &= 
  \partial_t w - \left( \lap_{B_2 \times F \times S^q}w \right) \\
  &=
    \partial_t w -
    \left(\lap_{B}w + \oh q u^{-1}\ip{\nabla u}{\nabla w} \right)\\
  &= \square_B w - \oh q u^{-1}\ip{\nabla u}{\nabla w} ,
\end{align}
we find,
\begin{align}
  \square_B w
  &= - \ein_F - y - \oh q u^{-1}\ip{\nabla u}{\nabla w} \\
  &= - \ein_F - y - \oh q v \partial_u w. \label{squareb_w} 
\end{align}
Now, using  and the fact that $w$ is a function of $u$ and $t$,
\begin{align}
  \square_B w = (- \ein + c_v v) \partial_u w + \partial_{t;u}w - uv \partial_u^2 w
\end{align}
so by \eqref{squareb_w} we find
\begin{align}
  \partial_{t;u}w - u v \partial_u^2 w
  &= - \ein_F - y + \ein \partial_u w - c_v v \partial_u w - \oh q v \partial_u w \\
  &= - \ein_F - y + \ein \partial_u w - \ein/2 v\partial_u w \label{evo_w_in_u}.
\end{align}
Note we may also write $y = w^{-1}|\nabla w|^2 = w^{-1}uv(\partial_u w)^2$.

\subsubsection{Writing the evolution in terms of $\phi$}
We also write \eqref{evo_w_in_u} in terms of $\phi$.  Using \eqref{u_to_phi_1} we have
\begin{align}
  \partial_{t;u}w
  &= v \left( \partial_\phi^2 w - \phi^{-1}\partial_\phi w \right) \\
  &- \ein_F - y + (\ein - \ein/2 v) \on2 \phi^{-1}\partial_\phi w
\end{align}
or, simplifying,
\begin{align}
  \partial_{t;\phi}w
  &= v \partial_\phi^2 w  - \ein_F - y + (\on2\ein - (\on4\ein - 1) v) \phi^{-1}\partial_\phi w. \label{eqn_w}
\end{align}

\subsection{Second fundamental form for doubly warped products}
Consider the case of a doubly warped product over an interval.  The second fundamental form $\kappa$ in Section \ref{v_deriving_section} is the second fundamental form of a surface $(s, p) \times F$, which is
\begin{align}
  \on4 dim(F) w^{-1}y
  &= \on4 dim(F) w^{-2}u^2v (\partial_u w)^2.
\end{align}
Therefore the term $-2(\kappa^2)v$ is $- \on2 dim(F) w^{-2}u^2 v^2 (\partial_u w)^2 = - \on8 dim(F)w^{-2} v^2 \phi^2(\partial_\phi w)^2$.


\section{Ricci-DeTurck flow}\label{rcdt}

Under Ricci-DeTurck flow with background metric $\tilde g$, the evolution $h(t) = g(t) - \tilde g(t)$ is given by the following lemma.  This copies the calculation in Lemma 3.1 of \cite{steady_alix} and generalizes it.
\begin{lemma}\label{dtevo}
  Let $\tilde g$ be a time-dependent family of metrics, and let $X$ be a time-dependent vector field. Let $g$ be a metric satisfying
  \begin{align}
    \Rf_X[g] =  \lie_{V[g, \tilde g]}g.    \label{rdt_modified}
  \end{align}
  Let $g = \tilde g + h$,  $g^{-1} = \tilde g^{-1} - \bar h$, and $\hat h^{ij} = \tilde g^{ai}\tilde g^{bj}h_{ab} - \bar h^{ab}$.
  
  Then
  \begin{align}
    \label{eq:5}
    \square_{X, \tilde g, g} h
    &= 2\Rm[h] + \UT[h] + Q[h] + \Cov[g, \nabla h]\\
    &- \Rf_X[\tilde g] - \left( \Rf_X[\tilde g] \cdot h \right) 
  \end{align}
  where all covariant derivatives and curvatures are with respect to $\tilde g$, and the terms are as follows:
  \begin{align}
    \label{eq:6}
    \left( \lap_{\tilde g,g} h \right)_{ij} = g^{ab} \nabla_a \nabla_b h_{ij},&\quad
    \lap_{X,\tilde g, g} = \lap_{\tilde g, g} - \nabla_X ,\quad
    \square_{X, \tilde g, g} = \partial_t - \lap_{X, \tilde g, g}, \\
    \Rm[h] = \tilde g^{ac}\tilde g^{bd}\Rm_{ajbi}h_{cd},&\quad
    Q[h] = \symme{i}{j}{\Rm^p_{ajb}h^{ab}h_{ip} - \Rm_{ajbi}\hat h^{ab}},\\
    (A \cdot B)_{ij} = \oh
    \on2\symme{i}{j}{ \tilde g^{ab}A_{ai}B_{bj}}
    ,&\quad
    \UT[B] = \left( (\partial_t \tilde g) \cdot B \right), \\
    |\Cov[g, \nabla h]| \leq c_0\left(1 + |h| \right)|\nabla h|^2.&
  \end{align}
  In the last line $c_0$ is a constant depending only on the dimension.
\end{lemma}

\begin{proof}
  The convention in this proof is that all curvatures and covariant derivatives are taken with respect to $\tilde g$.  By Lemma 2.1 of \cite{Shi} we have
  \begin{align}
    \label{eq:7}
    \partial_t g_{ij} = g^{ab}\nabla_a \nabla_b g_{ij} - \symme{i}{j}{g^{ab}g_{ip}\Rm_{ajb}^p} - \left( \lie_X g \right)_{ij}+ \Cov(g, \nabla g)
  \end{align}
  Since $g = \tilde g + h$ and $\nabla$ is the connection of $\tilde g$,
    \begin{align}
      \partial_t h_{ij} &= g^{ab}\nabla_a \nabla_b h_{ij} \label{dtevo:shi1}\\
                        &- \partial_t \tilde g - \symme{i}{j}{g^{ab}g_{ip}\Rm_{ajb}^p} - \left( \lie_X g \right)_{ij} \label{dtevo:shi2}\\
                        &+ \Cov(g, \nabla h) \label{dtevo:shi3}
  \end{align}
  
  \textbf{Rewriting the curvature term.}  Let $ g^{ij} = \tilde g^{ij} - \bar h^{ij}$.  Expand $g^{ab}g_{ip} = (\tilde g^{ab} - \bar h^{ab})(\tilde g_{ip} + h_{ip})$ in the curvature term.
  \begin{align}
    \label{eq:8}
    -g^{ab}g_{ip}\Rm_{ajb}^p
    &= \left( -\tilde g^{ab}\tilde g_{ip} + \bar h^{ab}\tilde g_{ip} - \tilde g^{ab}h_{ip} + \bar h^{ab}h_{ip}\right)\Rm_{ajb}^p \\
    &= - \Rc_{ij} + \bar h^{ab}\Rm_{ajbi} - \Rc_j^ph_{ip} + \bar h^{ab}h_{ip}\Rm^p_{ajb} \label{dtevo:l1}
  \end{align}
  Now let
  \begin{align}
    \label{eq:9}
    \hat h^{ij} = \tilde g^{ci}\tilde g^{dj}h_{cd} - \bar h^{ab}
  \end{align}
  so that
  \begin{align}
    \label{eq:10}
    \bar h^{ab} \Rm_{ajbi} = \tilde g^{ca}\tilde g^{db}\Rm_{cjdi}h_{ab} - \Rm_{ajbi}\hat h^{ab}.
  \end{align}
  Putting this together with \eqref{dtevo:l1} we have
  \begin{align}
    \label{eq:11}
    -g^{ab}g_{ip}\Rm_{ajb}^p
    &= - \Rc_{ij} + \tilde g^{ac}\tilde g^{bd}\Rm_{cjdi}h_{ab} - \Rc_j^ph_{ip} \\
    &+ h^{ab}h_{ip}\Rm^p_{ajb} - \Rm_{ajbi}\hat h^{ab}.
  \end{align}
  Finally taking the symmetrization we find
  \begin{align}
    \symme{i}{j}{-g^{ab}g_{ip}\Rm_{ajb}^p}
    &= - 2 \Rc_{ij} + 2\Rm[h]_{ij} - (\Rc \cdot h)_{ij} + Q(h)_{ij} \label{dtevo:curvrewrite}
  \end{align}

  \textbf{Rewriting the Lie term}
  We have
  $
    -\lie_X g = -\lie_X \tilde g - \lie_X h
    $.
  We can relate the lie derivative with the covariant derivative and the lie derivative of the metric,
  \begin{align}
    \label{eq:14}
    \left( -\lie_X h \right)_{ij}
    &= \left( -\nabla_X h \right)_{ij} - \oh \symme{i}{j}{h_{pi}\tilde g^{pq}(\lie_X \tilde g)_{qj}} \\
    &= \left( -\nabla_X h \right)_{ij} - \oh \left( \left( \lie_X \tilde g \right) \cdot h \right)_{ij}
  \end{align}
  (The first line is true in general, the second line uses that $h_{ij}$ is symmetric.)
  Thus
  \begin{align}
    -\lie_X g = -\lie_X \tilde g - \nabla_X h - \oh \left( \lie_X \tilde g \right) \cdot h \label{dtevo:lierewrite}
  \end{align}

  \textbf{Coming back to the evolution.} Using \eqref{dtevo:curvrewrite} and \eqref{dtevo:lierewrite}, the evolution \eqref{dtevo:shi1}-\eqref{dtevo:shi2} becomes
  \begin{align}
    \label{eq:17}
    \partial_t h
    &= \hat \lap h - \nabla_X h \\
    &- \partial_t \tilde g - 2 \Rc[\tilde g] - \lie_X \tilde g \\
    &- \oh \left( \left( 2 \Rc + \lie_X \tilde g \right) \cdot h \right) \\
    &+ 2\Rm[h] + Q(h) + \Cov(g, h). \\
  \end{align}
  So unraveling definitions,
  \begin{align}
    \label{eq:18}
    \hat \square_X h
    &= - \Rf_X[\tilde g] \\
    &+ \oh \left( \left(\partial_t g \right) \cdot h \right) - \oh \left( \left( \partial_t g + 2 \Rc + \lie_X \tilde g \right) \cdot h \right) \\
    &+ 2\Rm[h] + Q(h) + \Cov(g, h) \\
    &\\
    &= - \Rf_X[\tilde g] \\
    &+ \oh \UT[h] - \oh \left( \Rf_X[\tilde g]\cdot h \right) \\
    &+ 2\Rm[h] + Q(h) + \Cov(g, h) 
  \end{align}
  as desired.
\end{proof}

  What we will use is the following evolution of $|h|$ and $|h|^2$.  For $p \in M$ We set 
$
  \rmplus(p) = \max_{h \in Sym_2(T_pM) : |h| = 1}\ip{\Rm[h]}{h}(p)
$.
Now consider the case when $\Rf_X[\tilde g] = 0$.  Then, for $y = |h|^2$ we have (we allow $c_0$ to change from line to line)
\begin{align}
  \square_{X, \tilde g, g} y
  &\leq 4 \rmplus y - 2 g^{ab}\tilde g^{cd}\tilde g^{ef}\nabla_a h_{ce} \nabla_b h_{df} \\
  &+ c_0\left( |\Rm|y^{3/2} + |\nabla h|^2y^{1/2} \right).
\end{align}
Note that the linear term $\UT[h]$ may be removed using the Uhlenbeck trick, and disappears in the evolution of the norm.  The term  $-2 g^{ab}\tilde g^{cd}\tilde g^{ef}\nabla_a h_{ce}\nabla_b h_{df}$ is strictly negative.  If $|h| < \oh$ we can use $g^{ab}\tilde g^{cd}\tilde g^{ef}\nabla_a h_{ce} \nabla_b h_{df} \geq (1-c_0 y^{1/2})|\nabla h|^2$ and find
\begin{align}
  \square_{X, \tilde g, g} y
  &\leq 4 \rmplus y - 2(1 - c_0 y^{1/2})|\nabla h|^2 + c_0 |\Rm|y^{3/2}. \label{dtevo_square}
\end{align}
Alternatively, we can derive the evolution of $z = |h|$ and use the inequality $g^{ab}\nabla_a |h|^2 \nabla_b |h|^2
  \leq
  4|h|^2 g^{ab}\tilde g^{cd}\tilde g^{ef}\nabla_a h_{ce}\nabla_b h_{df}$
  to find 
\begin{align}
  \square_{X, \tilde g, g}z
  &\leq 2 \rmplus z + c_0 \left(|\Rm| z^2 + |\nabla h|^2 \right).\label{rdt_norm}
\end{align}


\section{Notation}\label{appendix_notation}
The heat operator is $\square u = \partial_t u - \lap u$.  If $X$ is a vector field then $\lap_Xu = \lap u - \nabla_Xu$ and $\square_X = \partial_t - \lap_X$.

The curvature tensors are $\Rm$ for the full Riemannian $(0, 4)$ tensor, $\Rc$ for the Ricci curvature, and $R$ for the scalar curvature.  The indices of $\Rm$ are such that $\Rm_{ijij}$ is a sectional curvature in an orthonormal frame.

The vector field $V[g, \tilde g]$, the operator $\lap_{g, \tilde g}$, and $\Rf[g]$ are defined in Section \ref{section:rdt}.  There we also define $\Rm[h]$ for a symmetric two-tensor $h$, and $\Lambda_{\Rm}:M \to \real$.

Everywhere $g_{S^q}$ is the metric of sectional curvature $1$ on the $q$ dimensional sphere $S^q$.  We define $\ein = 2(q-1)$ so that $2\Rc_{g_{S^q}} = \ein g_{S^q}$.  We also have a general Einstein manifold $(F, g_F)$ in play, its Ricci curvature satisfies $2\Rc_F = \ein_F g_F$ for some $\ein_F \in \real$.

Usually we have a metric of the form
\begin{align}
  a dx^2 + u g_{S^q} + w g_F
\end{align}
for $x$ in some interval $I$.  Here $a$, $u$, and $w$ are functions of $I$.  The functions $a$, $u$, and $w$ may also depend on time.  On these manifolds we have the derived functions $v = u^{-1}|\nabla u|^2$ and $y = w^{-1}|\nabla w|^2$.  Rarely we also use $\phi = \sqrt{u}$ and $\psi = \sqrt{w}$.

We have a lot of scaling.  Briefly:
\begin{align}
  \nu(t) = V_0(\ein t),
  \quad \omega(t) = W_0(\ein t),
  \quad \alpha(t) = t \nu(t),
  \quad \beta(t) = \alpha'(t)\\
  \rho = t^{-1}u,
  \quad \sigma = (t \nu(t))^{-1}u,
  \quad \zeta = t \nu(t)^{-1/2}u = \nu(t)\sigma,\\
  \hat u = u + \ein t,
  \quad \hat w = w + \ein_F t,
  \quad \bar w = \omega(t)^{-1}(w + \ein_F t).
\end{align}

We have some functions which are written in terms of $u$.  Generally capital letters denote known functions which are written in terms of $u$, whereas lowercase letters denote unknown functions.  The functions $V_0$ and $W_0$ are the initial values for $v$ and $w$ in a model pinch.  $V_{prish}$ and $W_{prish}$ are our approximations for $v$ and $w$ in the productish region, and $V^{\pm}_{prish}$ and $W^{\pm}_{prish}$ are upper and lower barriers for $v$ and $w$ based on these approximations.  Similarly these names with the subscript $tip$ are approximations and barriers in the tip region.  In Section \ref{section:productish}, we only refer to the functions for the produtish region, and therefore we drop the subscripts for cleanliness. Similarly in Section \ref{section:tip} we only refer to the tip functions, so we drop the subscript there as well.

Other functions of $u$ are $V_{bry}$, $V_{pert}$, and $W_{pert}$ (introduced in Section \ref{section:tip}, and with an overview in Section \ref{function_summary}).

We define $x^{a,b} = x^a(1 + x)^{b-a}$.  The point is that it's a smooth function on $(0, \infty)$ which behaves like $x^a$ at 0 and $x^{b}$ at $\infty$.


\bibliographystyle{amsalpha}
\bibliography{bibliography}

\end{document}